\documentclass[12pt] {article}
\usepackage{amsmath,amsthm}
\usepackage{amssymb,latexsym, hyperref}

\usepackage{breakurl}   

\usepackage{empheq}
\usepackage{amscd}
\usepackage{chngcntr}
\usepackage{enumitem}
\newlist{paragraphlist}{enumerate}{1}

\setlist[paragraphlist,1]{leftmargin=*,label={\bfseries \arabic*}}

\counterwithin{paragraphlisti}{subsubsection}

\title{New invariants of Gromov-Hausdorff limits of Riemannian surfaces with curvature bounded below}
\date{}

\author{Semyon Alesker\footnote{Partially supported by the US-Israel BSF grant 2018115 and the ISF grant 743/22.}\\
{ \normalsize Department of Mathematics, Tel Aviv University, Ramat Aviv}\\
{ \normalsize 69978 Tel Aviv, Israel}\\
{ \normalsize Email: alesker.semyon75@gmail.com}
\and
Mikhail G. Katz\footnote{Partially supported by the BSF grant Number 2020124 and the ISF grant 743/22.}
\\  {\normalsize Department of Mathematics, Bar Ilan University}
\\  {\normalsize Ramat Gan 5290002, Israel}
\\   {\normalsize Email address: katzmik@math.biu.ac.il}
\and
Roman Prosanov\footnote{This research was funded in part by the Swiss National Science Foundation grant $200021_-179133$ and was funded in part by the Austrian Science Fund (FWF) ESPRIT grant ESP-12-N.}
\\   {\normalsize University of Vienna, Faculty of Mathematics}
\\   {\normalsize Oskar-Morgenstern-Platz 1, A-1090 Vienna, Austria}
\\     {\normalsize E-mail: roman.prosanov@univie.ac.at}}

\def\RR{\mathbb{R}}
\def\CC{\mathbb{C}}

\def\NN{\mathbb{N}}
\def\ZZ{\mathbb{Z}}
\def\HH{\mathbb{H}}

\def\PP{\mathbb{P}}

\def\FF{\mathbb{F}}
\def\SS{\mathbb{S}}
\def\BB{\mathbb{B}}
\def\TT{\mathbb{T}}
\def\KK{\mathbb{K}}

\def\One{{1\hskip-2.5pt{\rm l}}}

\def\eps{\varepsilon}
\def\alp{\alpha}

\def\lam{\lambda}

\def\to{\longrightarrow}

\def\qed { Q.E.D. }

\def\inj{\hookrightarrow}

\newcommand{\e}{\varepsilon}
\swapnumbers
\newtheorem{theorem}{Theorem}[section]
\newtheorem{corollary}[theorem]{Corollary}
\newtheorem{lemma}[theorem]{Lemma}
\newtheorem{proposition}[theorem]{Proposition}
\newtheorem{claim}[theorem]{Claim}
\theoremstyle{definition}

\newtheorem{definition}[theorem]{Definition}
\newtheorem{remark}[theorem]{Remark}
\theoremstyle{conjecture}

\theoremstyle{principle}

\theoremstyle{question}
\newtheorem{question}[theorem]{Question}

\def\ca{{\cal A}} \def\cb{{\cal B}} \def\cc{{\cal C}}
  
\def\cg{{\cal G}} \def\ch{{\cal H}} \def\ci{{\cal I}}
  \def\cl{{\cal L}}
 \def\cn{{\cal N}} \def\co{{\cal O}}

  \def\cu{{\cal U}}
  \def\cx{{\cal X}}
 
\def\pt{\partial}
\usepackage{comment}

\makeatletter

\def\diagram{\m@th\leftwidth=\z@ \rightwidth=\z@ \topheight=\z@
\botheight=\z@ \setbox\@picbox\hbox\bgroup}

\def\enddiagram{\egroup\wd\@picbox\rightwidth\unitlength
\ht\@picbox\topheight\unitlength \dp\@picbox\botheight\unitlength
\hskip\leftwidth\unitlength\box\@picbox}

\def\bfig{\begin{diagram}}
\def\efig{\end{diagram}}
\newcount\wideness \newcount\leftwidth \newcount\rightwidth
\newcount\highness \newcount\topheight \newcount\botheight

\def\ratchet#1#2{\ifnum#1<#2 \global #1=#2 \fi}

\def\putbox(#1,#2)#3{%
\horsize{\wideness}{#3} \divide\wideness by 2
{\advance\wideness by #1 \ratchet{\rightwidth}{\wideness}}
{\advance\wideness by -#1 \ratchet{\leftwidth}{\wideness}}
\vertsize{\highness}{#3} \divide\highness by 2
{\advance\highness by #2 \ratchet{\topheight}{\highness}}
{\advance\highness by -#2 \ratchet{\botheight}{\highness}}
\put(#1,#2){\makebox(0,0){$#3$}}}

\def\putlbox(#1,#2)#3{%
\horsize{\wideness}{#3}
{\advance\wideness by #1 \ratchet{\rightwidth}{\wideness}}
{\ratchet{\leftwidth}{-#1}}
\vertsize{\highness}{#3} \divide\highness by 2
{\advance\highness by #2 \ratchet{\topheight}{\highness}}
{\advance\highness by -#2 \ratchet{\botheight}{\highness}}
\put(#1,#2){\makebox(0,0)[l]{$#3$}}}

\def\putrbox(#1,#2)#3{%
\horsize{\wideness}{#3}
{\ratchet{\rightwidth}{#1}}
{\advance\wideness by -#1 \ratchet{\leftwidth}{\wideness}}
\vertsize{\highness}{#3} \divide\highness by 2
{\advance\highness by #2 \ratchet{\topheight}{\highness}}
{\advance\highness by -#2 \ratchet{\botheight}{\highness}}
\put(#1,#2){\makebox(0,0)[r]{$#3$}}}

\def\adjust[#1]{} 

\newcount \coefa
\newcount \coefb
\newcount \coefc
\newcount\tempcounta
\newcount\tempcountb
\newcount\tempcountc
\newcount\tempcountd
\newcount\xext
\newcount\yext
\newcount\xoff
\newcount\yoff
\newcount\gap%
\newcount\arrowtypea
\newcount\arrowtypeb
\newcount\arrowtypec
\newcount\arrowtyped
\newcount\arrowtypee
\newcount\height
\newcount\width
\newcount\xpos
\newcount\ypos
\newcount\run
\newcount\rise
\newcount\arrowlength
\newcount\halflength
\newcount\arrowtype
\newdimen\tempdimen
\newdimen\xlen
\newdimen\ylen
\newsavebox{\tempboxa}%
\newsavebox{\tempboxb}%
\newsavebox{\tempboxc}%

\newdimen\w@dth

\def\setw@dth#1#2{\setbox\z@\hbox{\m@th$#1$}\w@dth=\wd\z@
\setbox\@ne\hbox{\m@th$#2$}\ifnum\w@dth<\wd\@ne \w@dth=\wd\@ne \fi
\advance\w@dth by 1.2em}

\def\t@^#1_#2{\allowbreak\def\n@one{#1}\def\n@two{#2}\mathrel
{\setw@dth{#1}{#2}
\mathop{\hbox to \w@dth{\rightarrowfill}}\limits
\ifx\n@one\empty\else ^{\box\z@}\fi
\ifx\n@two\empty\else _{\box\@ne}\fi}}
\def\t@@^#1{\@ifnextchar_{\t@^{#1}}{\t@^{#1}_{}}}
\def\to{\@ifnextchar^{\t@@}{\t@@^{}}}

\def\t@left^#1_#2{\def\n@one{#1}\def\n@two{#2}\mathrel{\setw@dth{#1}{#2}
\mathop{\hbox to \w@dth{\leftarrowfill}}\limits
\ifx\n@one\empty\else ^{\box\z@}\fi
\ifx\n@two\empty\else _{\box\@ne}\fi}}
\def\t@@left^#1{\@ifnextchar_{\t@left^{#1}}{\t@left^{#1}_{}}}
\def\toleft{\@ifnextchar^{\t@@left}{\t@@left^{}}}

\def\two@^#1_#2{\allowbreak
\def\n@one{#1}\def\n@two{#2}\mathrel{\setw@dth{#1}{#2}
\mathop{\vcenter{\lineskip\z@\baselineskip\z@
                 \hbox to \w@dth{\rightarrowfill}%
                 \hbox to \w@dth{\rightarrowfill}}%
       }\limits
\ifx\n@one\empty\else ^{\box\z@}\fi
\ifx\n@two\empty\else _{\box\@ne}\fi}}
\def\tw@@^#1{\@ifnextchar _{\two@^{#1}}{\two@^{#1}_{}}}
\def\two{\@ifnextchar ^{\tw@@}{\tw@@^{}}}

\def\tofr@^#1_#2{\def\n@one{#1}\def\n@two{#2}\mathrel{\setw@dth{#1}{#2}
\mathop{\vcenter{\hbox to \w@dth{\rightarrowfill}\kern-1.7ex
                 \hbox to \w@dth{\leftarrowfill}}%
       }\limits
\ifx\n@one\empty\else ^{\box\z@}\fi
\ifx\n@two\empty\else _{\box\@ne}\fi}}
\def\t@fr@^#1{\@ifnextchar_ {\tofr@^{#1}}{\tofr@^{#1}_{}}}
\def\tofro{\@ifnextchar^ {\t@fr@}{\t@fr@^{}}}

\def\mon{\mathop{\m@th\hbox to
      14.6\P@{\lasyb\char'51\hskip-2.1\P@$\arrext$\hss
$\mathord\rightarrow$}}\limits} 
\def\leftmono{\mathrel{\m@th\hbox to
14.6\P@{$\mathord\leftarrow$\hss$\arrext$\hskip-2.1\P@\lasyb\char'50%
}}\limits} 
\mathchardef\arrext="0200       

\def\settypes(#1,#2,#3){\arrowtypea#1 \arrowtypeb#2 \arrowtypec#3}
\def\settoheight#1#2{\setbox\@tempboxa\hbox{#2}#1\ht\@tempboxa\relax}%
\def\settodepth#1#2{\setbox\@tempboxa\hbox{#2}#1\dp\@tempboxa\relax}%
\def\settokens`#1`#2`#3`#4`{%
     \def\tokena{#1}\def\tokenb{#2}\def\tokenc{#3}\def\tokend{#4}}
\def\setsqparms[#1`#2`#3`#4;#5`#6]{%
\arrowtypea #1
\arrowtypeb #2
\arrowtypec #3
\arrowtyped #4
\width #5
\height #6
}
\def\setpos(#1,#2){\xpos=#1 \ypos#2}

\def\settriparms[#1`#2`#3;#4]{\settripairparms[#1`#2`#3`1`1;#4]}%

\def\settripairparms[#1`#2`#3`#4`#5;#6]{%
\arrowtypea #1
\arrowtypeb #2
\arrowtypec #3
\arrowtyped #4
\arrowtypee #5
\width #6
\height #6
}

\def\resetparms{\settripairparms[1`1`1`1`1;500]\width 500}

\resetparms

\def\mvector(#1,#2)#3{
\put(0,0){\vector(#1,#2){#3}}%
\put(0,0){\vector(#1,#2){26}}%
}
\def\evector(#1,#2)#3{{
\arrowlength #3
\put(0,0){\vector(#1,#2){\arrowlength}}%
\advance \arrowlength by-30
\put(0,0){\vector(#1,#2){\arrowlength}}%
}}

\def\horsize#1#2{%
\settowidth{\tempdimen}{$#2$}%
#1=\tempdimen
\divide #1 by\unitlength
}

\def\vertsize#1#2{%
\settoheight{\tempdimen}{$#2$}%
#1=\tempdimen
\settodepth{\tempdimen}{$#2$}%
\advance #1 by\tempdimen
\divide #1 by\unitlength
}

\def\putvector(#1,#2)(#3,#4)#5#6{{%
\ifnum3<\arrowtype
\putdashvector(#1,#2)(#3,#4)#5\arrowtype
\else
\ifnum\arrowtype<-3
\putdashvector(#1,#2)(#3,#4)#5\arrowtype
\else
\xpos=#1
\ypos=#2
\run=#3
\rise=#4
\arrowlength=#5
\ifnum \arrowtype<0
    \ifnum \run=0
        \advance \ypos by-\arrowlength
    \else
        \tempcounta \arrowlength
        \multiply \tempcounta by\rise
        \divide \tempcounta by\run
        \ifnum\run>0
            \advance \xpos by\arrowlength
            \advance \ypos by\tempcounta
        \else
            \advance \xpos by-\arrowlength
            \advance \ypos by-\tempcounta
        \fi
    \fi
    \multiply \arrowtype by-1
    \multiply \rise by-1
    \multiply \run by-1
\fi
\ifcase \arrowtype
\or \put(\xpos,\ypos){\vector(\run,\rise){\arrowlength}}%
\or \put(\xpos,\ypos){\mvector(\run,\rise)\arrowlength}%
\or \put(\xpos,\ypos){\evector(\run,\rise){\arrowlength}}%
\fi\fi\fi
}}

\def\putsplitvector(#1,#2)#3#4{
\xpos #1
\ypos #2
\arrowtype #4
\halflength #3
\arrowlength #3
\gap 140
\advance \halflength by-\gap
\divide \halflength by2
\ifnum\arrowtype>0
   \ifcase \arrowtype
   \or \put(\xpos,\ypos){\line(0,-1){\halflength}}%
       \advance\ypos by-\halflength
       \advance\ypos by-\gap
       \put(\xpos,\ypos){\vector(0,-1){\halflength}}%
   \or \put(\xpos,\ypos){\line(0,-1)\halflength}%
       \put(\xpos,\ypos){\vector(0,-1)3}%
       \advance\ypos by-\halflength
       \advance\ypos by-\gap
       \put(\xpos,\ypos){\vector(0,-1){\halflength}}%
   \or \put(\xpos,\ypos){\line(0,-1)\halflength}%
       \advance\ypos by-\halflength
       \advance\ypos by-\gap
       \put(\xpos,\ypos){\evector(0,-1){\halflength}}%
   \fi
\else \arrowtype=-\arrowtype
   \ifcase\arrowtype
   \or \advance \ypos by-\arrowlength
       \put(\xpos,\ypos){\line(0,1){\halflength}}%
       \advance\ypos by\halflength
       \advance\ypos by\gap
       \put(\xpos,\ypos){\vector(0,1){\halflength}}%
   \or \advance \ypos by-\arrowlength
       \put(\xpos,\ypos){\line(0,1)\halflength}%
       \put(\xpos,\ypos){\vector(0,1)3}%
       \advance\ypos by\halflength
       \advance\ypos by\gap
       \put(\xpos,\ypos){\vector(0,1){\halflength}}%
   \or \advance \ypos by-\arrowlength
       \put(\xpos,\ypos){\line(0,1)\halflength}%
       \advance\ypos by\halflength
       \advance\ypos by\gap
       \put(\xpos,\ypos){\evector(0,1){\halflength}}%
   \fi
\fi
}

\def\putmorphism(#1)(#2,#3)[#4`#5`#6]#7#8#9{{%
\run #2
\rise #3
\ifnum\rise=0
  \puthmorphism(#1)[#4`#5`#6]{#7}{#8}#9%
\else\ifnum\run=0
  \putvmorphism(#1)[#4`#5`#6]{#7}{#8}#9%
\else
\setpos(#1)%
\arrowlength #7
\arrowtype #8
\ifnum\run=0
\else\ifnum\rise=0
\else
\ifnum\run>0
    \coefa=1
\else
   \coefa=-1
\fi
\ifnum\arrowtype>0
   \coefb=0
   \coefc=-1
\else
   \coefb=\coefa
   \coefc=1
   \arrowtype=-\arrowtype
\fi
\width=2
\multiply \width by\run
\divide \width by\rise
\ifnum \width<0  \width=-\width\fi
\advance\width by60
\if l#9 \width=-\width\fi
\putbox(\xpos,\ypos){#4}
{\multiply \coefa by\arrowlength
\advance\xpos by\coefa
\multiply \coefa by\rise
\divide \coefa by\run
\advance \ypos by\coefa
\putbox(\xpos,\ypos){#5} }%
{\multiply \coefa by\arrowlength
\divide \coefa by2
\advance \xpos by\coefa
\advance \xpos by\width
\multiply \coefa by\rise
\divide \coefa by\run
\advance \ypos by\coefa
\if l#9%
   \putrbox(\xpos,\ypos){#6}%
\else\if r#9%
   \putlbox(\xpos,\ypos){#6}%
\fi\fi }%
{\multiply \rise by-\coefc
\multiply \run by-\coefc
\multiply \coefb by\arrowlength
\advance \xpos by\coefb
\multiply \coefb by\rise
\divide \coefb by\run
\advance \ypos by\coefb
\multiply \coefc by70
\advance \ypos by\coefc
\multiply \coefc by\run
\divide \coefc by\rise
\advance \xpos by\coefc
\multiply \coefa by140
\multiply \coefa by\run
\divide \coefa by\rise
\advance \arrowlength by\coefa
\ifcase\arrowtype
\or \put(\xpos,\ypos){\vector(\run,\rise){\arrowlength}}%
\or \put(\xpos,\ypos){\mvector(\run,\rise){\arrowlength}}%
\or \put(\xpos,\ypos){\evector(\run,\rise){\arrowlength}}%
\fi}\fi\fi\fi\fi}}

\newcount\numbdashes \newcount\lengthdash \newcount\increment

\def\howmanydashes{
\numbdashes=\arrowlength \lengthdash=40
\divide\numbdashes by \lengthdash
\lengthdash=\arrowlength
\divide\lengthdash by \numbdashes
\increment=\lengthdash
\multiply\lengthdash by 3
\divide\lengthdash by 5
}

\def\putdashvector(#1)(#2,#3)#4#5{%
\ifnum#3=0 \putdashhvector(#1){#4}#5
\else
\ifnum#2=0
\putdashvvector(#1){#4}#5\fi\fi}

\def\putdashhvector(#1,#2)#3#4{{%
\arrowlength=#3 \howmanydashes
\multiput(#1,#2)(\increment,0){\numbdashes}%
{\vrule height .4pt width \lengthdash\unitlength}
\arrowtype=#4 \xpos=#1
\ifnum\arrowtype<0 \advance\arrowtype by 7 \fi
\ifcase\arrowtype
\or \advance\xpos by 10
    \put(\xpos,#2){\vector(-1,0){\lengthdash}}
    \advance\xpos by 40
    \put(\xpos,#2){\vector(-1,0){\lengthdash}}
\or \advance \xpos by 10
    \put(\xpos,#2){\vector(-1,0){\lengthdash}}
    \advance\xpos by  \arrowlength
    \advance\xpos by  -50
    \put(\xpos,#2){\vector(-1,0){\lengthdash}}
\or \advance\xpos by 10
    \put(\xpos,#2){\vector(-1,0){\lengthdash}}
\or \advance\xpos by \arrowlength
    \advance\xpos by -\lengthdash
    \put(\xpos,#2){\vector(1,0){\lengthdash}}
\or {\advance\xpos by 10
    \put(\xpos,#2){\vector(1,0){\lengthdash}}}
    \advance\xpos by \arrowlength
    \advance\xpos by -\lengthdash
    \put(\xpos,#2){\vector(1,0){\lengthdash}}
\or \advance\xpos by \arrowlength
    \advance\xpos by -\lengthdash
    \put(\xpos,#2){\vector(1,0){\lengthdash}}
    \advance\xpos by -40
    \put(\xpos,#2){\vector(1,0){\lengthdash}}
   \fi
}}

\def\putdashvvector(#1,#2)#3#4{{%
\arrowlength=#3 \howmanydashes
\ypos=#2 \advance\ypos by -\arrowlength
\multiput(#1,#2)(0,\increment){\numbdashes}%
    {\vrule width .4pt height \lengthdash\unitlength}
\arrowtype=#4 \ypos=#2
\ifnum\arrowtype<0 \advance\arrowtype by 7 \fi
\ifcase\arrowtype
\or \advance\ypos by \arrowlength \advance\ypos by -40
    \put(#1,\ypos){\vector(0,1){\lengthdash}}
    \advance\ypos by -40
    \put(#1,\ypos){\vector(0,1){\lengthdash}}
\or \advance\ypos by 10
    \put(#1,\ypos){\vector(0,1){\lengthdash}}
    \advance\ypos by \arrowlength \advance\ypos by -40
    \put(#1,\ypos){\vector(0,1){\lengthdash}}
\or \advance\ypos by \arrowlength \advance\ypos by -40
    \put(#1,\ypos){\vector(0,1){\lengthdash}}
\or \advance\ypos by 10
    \put(#1,\ypos){\vector(0,-1){\lengthdash}}
\or \advance\ypos by 10
    \put(#1,\ypos){\vector(0,-1){\lengthdash}}
    \advance\ypos by \arrowlength \advance\ypos by -40
    \put(#1,\ypos){\vector(0,-1){\lengthdash}}
\or \advance\ypos by 10
    \put(#1,\ypos){\vector(0,-1){\lengthdash}}
    \advance\ypos by 40
    \put(#1,\ypos){\vector(0,-1){\lengthdash}}
\fi
}}

\def\puthmorphism(#1,#2)[#3`#4`#5]#6#7#8{{%
\xpos #1
\ypos #2
\width #6
\arrowlength #6
\arrowtype=#7
\putbox(\xpos,\ypos){#3\vphantom{#4}}%
{\advance \xpos by\arrowlength
\putbox(\xpos,\ypos){\vphantom{#3}#4}}%
\horsize{\tempcounta}{#3}%
\horsize{\tempcountb}{#4}%
\divide \tempcounta by2
\divide \tempcountb by2
\advance \tempcounta by30
\advance \tempcountb by30
\advance \xpos by\tempcounta
\advance \arrowlength by-\tempcounta
\advance \arrowlength by-\tempcountb
\putvector(\xpos,\ypos)(1,0)\arrowlength\arrowtype
\divide \arrowlength by2
\advance \xpos by\arrowlength
\vertsize{\tempcounta}{#5}%
\divide\tempcounta by2
\advance \tempcounta by20
\if a#8 %
   \advance \ypos by\tempcounta
   \putbox(\xpos,\ypos){#5}%
\else
   \advance \ypos by-\tempcounta
   \putbox(\xpos,\ypos){#5}%
\fi}}

\def\putvmorphism(#1,#2)[#3`#4`#5]#6#7#8{{%
\xpos #1
\ypos #2
\arrowlength #6
\arrowtype #7
\settowidth{\xlen}{$#5$}%
\putbox(\xpos,\ypos){#3}%
{\advance \ypos by-\arrowlength
\putbox(\xpos,\ypos){#4}}%
{\advance\arrowlength by-140
\advance \ypos by-70
\ifdim\xlen>0pt
   \if m#8%
      \putsplitvector(\xpos,\ypos)\arrowlength\arrowtype
   \else
   \putvector(\xpos,\ypos)(0,-1)\arrowlength\arrowtype
   \fi
\else
   \putvector(\xpos,\ypos)(0,-1)\arrowlength\arrowtype
\fi}%
\ifdim\xlen>0pt
   \divide \arrowlength by2
   \advance\ypos by-\arrowlength
   \if l#8%
      \advance \xpos by-40
      \putrbox(\xpos,\ypos){#5}%
   \else\if r#8%
      \advance \xpos by40
      \putlbox(\xpos,\ypos){#5}%
   \else
      \putbox(\xpos,\ypos){#5}%
   \fi\fi
\fi
}}

\def\putsquarep<#1>(#2)[#3;#4`#5`#6`#7]{{%
\setsqparms[#1]%
\setpos(#2)%
\settokens`#3`%
\puthmorphism(\xpos,\ypos)[\tokenc`\tokend`{#7}]{\width}{\arrowtyped}b%
\advance\ypos by \height
\puthmorphism(\xpos,\ypos)[\tokena`\tokenb`{#4}]{\width}{\arrowtypea}a%
\putvmorphism(\xpos,\ypos)[``{#5}]{\height}{\arrowtypeb}l%
\advance\xpos by \width
\putvmorphism(\xpos,\ypos)[``{#6}]{\height}{\arrowtypec}r%
}}

\def\putsquare{\@ifnextchar <{\putsquarep}{\putsquarep%
   <\arrowtypea`\arrowtypeb`\arrowtypec`\arrowtyped;\width`\height>}}
\def\square{\@ifnextchar< {\squarep}{\squarep
   <\arrowtypea`\arrowtypeb`\arrowtypec`\arrowtyped;\width`\height>}}
\def\squarep<#1>[#2`#3`#4`#5;#6`#7`#8`#9]{{
\setsqparms[#1]
\diagram
\putsquarep<\arrowtypea`\arrowtypeb`\arrowtypec`
\arrowtyped;\width`\height>
(0,0)[#2`#3`#4`{#5};#6`#7`#8`{#9}]
\enddiagram
}}                                                 
\def\putptrianglep<#1>(#2,#3)[#4`#5`#6;#7`#8`#9]{{%
\settriparms[#1]%
\xpos=#2 \ypos=#3
\advance\ypos by \height
\puthmorphism(\xpos,\ypos)[#4`#5`{#7}]{\height}{\arrowtypea}a%
\putvmorphism(\xpos,\ypos)[`#6`{#8}]{\height}{\arrowtypeb}l%
\advance\xpos by\height
\putmorphism(\xpos,\ypos)(-1,-1)[``{#9}]{\height}{\arrowtypec}r%
}}

\def\putptriangle{\@ifnextchar <{\putptrianglep}{\putptrianglep
   <\arrowtypea`\arrowtypeb`\arrowtypec;\height>}}
\def\ptriangle{\@ifnextchar <{\ptrianglep}{\ptrianglep
   <\arrowtypea`\arrowtypeb`\arrowtypec;\height>}}
\def\ptrianglep<#1>[#2`#3`#4;#5`#6`#7]{{
\settriparms[#1]
\diagram
\putptrianglep<\arrowtypea`\arrowtypeb`
\arrowtypec;\height>
(0,0)[#2`#3`#4;#5`#6`{#7}]
\enddiagram
}}                                            

\def\putqtrianglep<#1>(#2,#3)[#4`#5`#6;#7`#8`#9]{{%
\settriparms[#1]%
\xpos=#2 \ypos=#3
\advance\ypos by\height
\puthmorphism(\xpos,\ypos)[#4`#5`{#7}]{\height}{\arrowtypea}a%
\putmorphism(\xpos,\ypos)(1,-1)[``{#8}]{\height}{\arrowtypeb}l%
\advance\xpos by\height
\putvmorphism(\xpos,\ypos)[`#6`{#9}]{\height}{\arrowtypec}r%
}}

\def\putqtriangle{\@ifnextchar <{\putqtrianglep}{\putqtrianglep
   <\arrowtypea`\arrowtypeb`\arrowtypec;\height>}}
\def\qtriangle{\@ifnextchar <{\qtrianglep}{\qtrianglep
   <\arrowtypea`\arrowtypeb`\arrowtypec;\height>}}
\def\qtrianglep<#1>[#2`#3`#4;#5`#6`#7]{{
\settriparms[#1]
\width=\height                                
\diagram
\putqtrianglep<\arrowtypea`\arrowtypeb`
\arrowtypec;\height>
(0,0)[#2`#3`#4;#5`#6`{#7}]
\enddiagram
}}

\def\putdtrianglep<#1>(#2,#3)[#4`#5`#6;#7`#8`#9]{{%
\settriparms[#1]%
\xpos=#2 \ypos=#3
\puthmorphism(\xpos,\ypos)[#5`#6`{#9}]{\height}{\arrowtypec}b%
\advance\xpos by \height \advance\ypos by\height
\putmorphism(\xpos,\ypos)(-1,-1)[``{#7}]{\height}{\arrowtypea}l%
\putvmorphism(\xpos,\ypos)[#4``{#8}]{\height}{\arrowtypeb}r%
}}

\def\putdtriangle{\@ifnextchar <{\putdtrianglep}{\putdtrianglep
   <\arrowtypea`\arrowtypeb`\arrowtypec;\height>}}
\def\dtriangle{\@ifnextchar <{\dtrianglep}{\dtrianglep
   <\arrowtypea`\arrowtypeb`\arrowtypec;\height>}}
\def\dtrianglep<#1>[#2`#3`#4;#5`#6`#7]{{
\settriparms[#1]
\width=\height                                
\diagram
\putdtrianglep<\arrowtypea`\arrowtypeb`
\arrowtypec;\height>
(0,0)[#2`#3`#4;#5`#6`{#7}]
\enddiagram
}}

\def\putbtrianglep<#1>(#2,#3)[#4`#5`#6;#7`#8`#9]{{%
\settriparms[#1]%
\xpos=#2 \ypos=#3
\puthmorphism(\xpos,\ypos)[#5`#6`{#9}]{\height}{\arrowtypec}b%
\advance\ypos by\height
\putmorphism(\xpos,\ypos)(1,-1)[``{#8}]{\height}{\arrowtypeb}r%
\putvmorphism(\xpos,\ypos)[#4``{#7}]{\height}{\arrowtypea}l%
}}

\def\putbtriangle{\@ifnextchar <{\putbtrianglep}{\putbtrianglep
   <\arrowtypea`\arrowtypeb`\arrowtypec;\height>}}
\def\btriangle{\@ifnextchar <{\btrianglep}{\btrianglep
   <\arrowtypea`\arrowtypeb`\arrowtypec;\height>}}
\def\btrianglep<#1>[#2`#3`#4;#5`#6`#7]{{
\settriparms[#1]
\width=\height                               
\diagram
\putbtrianglep<\arrowtypea`\arrowtypeb`
\arrowtypec;\height>
(0,0)[#2`#3`#4;#5`#6`{#7}]
\enddiagram
}}

\def\putAtrianglep<#1>(#2,#3)[#4`#5`#6;#7`#8`#9]{{%
\settriparms[#1]%
\xpos=#2 \ypos=#3
{\multiply \height by2
\puthmorphism(\xpos,\ypos)[#5`#6`{#9}]{\height}{\arrowtypec}b}%
\advance\xpos by\height \advance\ypos by\height
\putmorphism(\xpos,\ypos)(-1,-1)[#4``{#7}]{\height}{\arrowtypea}l%
\putmorphism(\xpos,\ypos)(1,-1)[``{#8}]{\height}{\arrowtypeb}r%
}}

\def\putAtriangle{\@ifnextchar <{\putAtrianglep}{\putAtrianglep
   <\arrowtypea`\arrowtypeb`\arrowtypec;\height>}}
\def\Atriangle{\@ifnextchar <{\Atrianglep}{\Atrianglep
   <\arrowtypea`\arrowtypeb`\arrowtypec;\height>}}
\def\Atrianglep<#1>[#2`#3`#4;#5`#6`#7]{{
\settriparms[#1]
\width=\height                                     
\diagram
\putAtrianglep<\arrowtypea`\arrowtypeb`
\arrowtypec;\height>
(0,0)[#2`#3`#4;#5`#6`{#7}]
\enddiagram
}}

\def\putAtrianglepairp<#1>(#2)[#3;#4`#5`#6`#7`#8]{{%
\settripairparms[#1]%
\setpos(#2)%
\settokens`#3`%
\puthmorphism(\xpos,\ypos)[\tokenb`\tokenc`{#7}]{\height}{\arrowtyped}b%
\advance\xpos by\height
\puthmorphism(\xpos,\ypos)[\phantom{\tokenc}`\tokend`{#8}]%
{\height}{\arrowtypee}b%
\advance\ypos by\height
\putmorphism(\xpos,\ypos)(-1,-1)[\tokena``{#4}]{\height}{\arrowtypea}l%
\putvmorphism(\xpos,\ypos)[``{#5}]{\height}{\arrowtypeb}m%
\putmorphism(\xpos,\ypos)(1,-1)[``{#6}]{\height}{\arrowtypec}r%
}}

\def\putAtrianglepair{\@ifnextchar <{\putAtrianglepairp}{\putAtrianglepairp%
   <\arrowtypea`\arrowtypeb`\arrowtypec`\arrowtyped`\arrowtypee;\height>}}
\def\Atrianglepair{\@ifnextchar <{\Atrianglepairp}{\Atrianglepairp%
   <\arrowtypea`\arrowtypeb`\arrowtypec`\arrowtyped`\arrowtypee;\height>}}

\def\Atrianglepairp<#1>[#2;#3`#4`#5`#6`#7]{{
\settripairparms[#1]
\settokens`#2`
\width=\height                                
\diagram
\putAtrianglepairp                            
<\arrowtypea`\arrowtypeb`\arrowtypec`
\arrowtyped`\arrowtypee;\height>
(0,0)[{#2};#3`#4`#5`#6`{#7}]
\enddiagram
}}

\def\putVtrianglep<#1>(#2,#3)[#4`#5`#6;#7`#8`#9]{{%
\settriparms[#1]%
\xpos=#2 \ypos=#3
\advance\ypos by\height
{\multiply\height by2
\puthmorphism(\xpos,\ypos)[#4`#5`{#7}]{\height}{\arrowtypea}a}%
\putmorphism(\xpos,\ypos)(1,-1)[`#6`{#8}]{\height}{\arrowtypeb}l%
\advance\xpos by\height
\advance\xpos by\height
\putmorphism(\xpos,\ypos)(-1,-1)[``{#9}]{\height}{\arrowtypec}r%
}}

\def\putVtriangle{\@ifnextchar <{\putVtrianglep}{\putVtrianglep
   <\arrowtypea`\arrowtypeb`\arrowtypec;\height>}}
\def\Vtriangle{\@ifnextchar <{\Vtrianglep}{\Vtrianglep
   <\arrowtypea`\arrowtypeb`\arrowtypec;\height>}}
\def\Vtrianglep<#1>[#2`#3`#4;#5`#6`#7]{{
\settriparms[#1]
\width=\height                                 
\diagram
\putVtrianglep<\arrowtypea`\arrowtypeb`
\arrowtypec;\height>
(0,0)[#2`#3`#4;#5`#6`{#7}]
\enddiagram
}}

\def\putVtrianglepairp<#1>(#2)[#3;#4`#5`#6`#7`#8]{{
\settripairparms[#1]%
\setpos(#2)%
\settokens`#3`%
\advance\ypos by\height
\putmorphism(\xpos,\ypos)(1,-1)[`\tokend`{#6}]{\height}{\arrowtypec}l%
\puthmorphism(\xpos,\ypos)[\tokena`\tokenb`{#4}]{\height}{\arrowtypea}a%
\advance\xpos by\height
\puthmorphism(\xpos,\ypos)[\phantom{\tokenb}`\tokenc`{#5}]%
{\height}{\arrowtypeb}a%
\putvmorphism(\xpos,\ypos)[``{#7}]{\height}{\arrowtyped}m%
\advance\xpos by\height
\putmorphism(\xpos,\ypos)(-1,-1)[``{#8}]{\height}{\arrowtypee}r%
}}

\def\putVtrianglepair{\@ifnextchar <{\putVtrianglepairp}{\putVtrianglepairp%
    <\arrowtypea`\arrowtypeb`\arrowtypec`\arrowtyped`\arrowtypee;\height>}}
\def\Vtrianglepair{\@ifnextchar <{\Vtrianglepairp}{\Vtrianglepairp%
    <\arrowtypea`\arrowtypeb`\arrowtypec`\arrowtyped`\arrowtypee;\height>}}
\def\Vtrianglepairp<#1>[#2;#3`#4`#5`#6`#7]{{
\settripairparms[#1]
\settokens`#2`
\diagram
\putVtrianglepairp                             
<\arrowtypea`\arrowtypeb`\arrowtypec`
\arrowtyped`\arrowtypee;\height>
(0,0)[{#2};#3`#4`#5`#6`{#7}]
\enddiagram
}}

\def\putCtrianglep<#1>(#2,#3)[#4`#5`#6;#7`#8`#9]{{%
\settriparms[#1]%
\xpos=#2 \ypos=#3
\advance\ypos by\height
\putmorphism(\xpos,\ypos)(1,-1)[``{#9}]{\height}{\arrowtypec}l%
\advance\xpos by\height
\advance\ypos by\height
\putmorphism(\xpos,\ypos)(-1,-1)[#4`#5`{#7}]{\height}{\arrowtypea}l%
{\multiply\height by 2
\putvmorphism(\xpos,\ypos)[`#6`{#8}]{\height}{\arrowtypeb}r}%
}}

\def\putCtriangle{\@ifnextchar <{\putCtrianglep}{\putCtrianglep
    <\arrowtypea`\arrowtypeb`\arrowtypec;\height>}}
\def\Ctriangle{\@ifnextchar <{\Ctrianglep}{\Ctrianglep
    <\arrowtypea`\arrowtypeb`\arrowtypec;\height>}}
\def\Ctrianglep<#1>[#2`#3`#4;#5`#6`#7]{{
\settriparms[#1]
\width=\height                               
\diagram
\putCtrianglep<\arrowtypea`\arrowtypeb`
\arrowtypec;\height>
(0,0)[#2`#3`#4;#5`#6`{#7}]
\enddiagram
}}                                           
\def\putDtrianglep<#1>(#2,#3)[#4`#5`#6;#7`#8`#9]{{%
\settriparms[#1]%
\xpos=#2 \ypos=#3
\advance\xpos by\height \advance\ypos by\height
\putmorphism(\xpos,\ypos)(-1,-1)[``{#9}]{\height}{\arrowtypec}r%
\advance\xpos by-\height \advance\ypos by\height
\putmorphism(\xpos,\ypos)(1,-1)[`#5`{#8}]{\height}{\arrowtypeb}r%
{\multiply\height by 2
\putvmorphism(\xpos,\ypos)[#4`#6`{#7}]{\height}{\arrowtypea}l}%
}}

\def\putDtriangle{\@ifnextchar <{\putDtrianglep}{\putDtrianglep
    <\arrowtypea`\arrowtypeb`\arrowtypec;\height>}}
\def\Dtriangle{\@ifnextchar <{\Dtrianglep}{\Dtrianglep
   <\arrowtypea`\arrowtypeb`\arrowtypec;\height>}}
\def\Dtrianglep<#1>[#2`#3`#4;#5`#6`#7]{{
\settriparms[#1]
\width=\height                              
\diagram
\putDtrianglep<\arrowtypea`\arrowtypeb`
\arrowtypec;\height>
(0,0)[#2`#3`#4;#5`#6`{#7}]
\enddiagram
}}                                          
\def\setrecparms[#1`#2]{\width=#1 \height=#2}%

\def\recursep<#1`#2>[#3;#4`#5`#6`#7`#8]{{\m@th
\width=#1 \height=#2
\settokens`#3`
\settowidth{\tempdimen}{$\tokena$}
\ifdim\tempdimen=0pt
  \savebox{\tempboxa}{\hbox{$\tokenb$}}%
  \savebox{\tempboxb}{\hbox{$\tokend$}}%
  \savebox{\tempboxc}{\hbox{$#6$}}%
\else
  \savebox{\tempboxa}{\hbox{$\hbox{$\tokena$}\times\hbox{$\tokenb$}$}}%
  \savebox{\tempboxb}{\hbox{$\hbox{$\tokena$}\times\hbox{$\tokend$}$}}%
  \savebox{\tempboxc}{\hbox{$\hbox{$\tokena$}\times\hbox{$#6$}$}}%
\fi
\ypos=\height
\divide\ypos by 2
\xpos=\ypos
\advance\xpos by \width
\bfig
\putCtrianglep<-1`1`1;\ypos>(0,0)[`\tokenc`;#5`#6`{#7}]%
\puthmorphism(\ypos,0)[\tokend`\usebox{\tempboxb}`{#8}]{\width}{-1}b%
\puthmorphism(\ypos,\height)[\tokenb`\usebox{\tempboxa}`{#4}]{\width}{-1}a%
\advance\ypos by \width
\putvmorphism(\ypos,\height)[``\usebox{\tempboxc}]{\height}1r%
\efig
}}

\def\recurse{\@ifnextchar <{\recursep}{\recursep<\width`\height>}}

\def\puttwohmorphisms(#1,#2)[#3`#4;#5`#6]#7#8#9{{%
\puthmorphism(#1,#2)[#3`#4`]{#7}0a
\ypos=#2
\advance\ypos by 20
\puthmorphism(#1,\ypos)[\phantom{#3}`\phantom{#4}`#5]{#7}{#8}a
\advance\ypos by -40
\puthmorphism(#1,\ypos)[\phantom{#3}`\phantom{#4}`#6]{#7}{#9}b
}}

\def\puttwovmorphisms(#1,#2)[#3`#4;#5`#6]#7#8#9{{%
\putvmorphism(#1,#2)[#3`#4`]{#7}0a
\xpos=#1
\advance\xpos by -20
\putvmorphism(\xpos,#2)[\phantom{#3}`\phantom{#4}`#5]{#7}{#8}l
\advance\xpos by 40
\putvmorphism(\xpos,#2)[\phantom{#3}`\phantom{#4}`#6]{#7}{#9}r
}}

\def\puthcoequalizer(#1)[#2`#3`#4;#5`#6`#7]#8#9{{%
\setpos(#1)%
\puttwohmorphisms(\xpos,\ypos)[#2`#3;#5`#6]{#8}11%
\advance\xpos by #8
\puthmorphism(\xpos,\ypos)[\phantom{#3}`#4`#7]{#8}1{#9}
}}

\def\putvcoequalizer(#1)[#2`#3`#4;#5`#6`#7]#8#9{{%
\setpos(#1)%
\puttwovmorphisms(\xpos,\ypos)[#2`#3;#5`#6]{#8}11%
\advance\ypos by -#8
\putvmorphism(\xpos,\ypos)[\phantom{#3}`#4`#7]{#8}1{#9}
}}

\def\putthreehmorphisms(#1)[#2`#3;#4`#5`#6]#7(#8)#9{{%
\setpos(#1) \settypes(#8)
\if a#9 %
     \vertsize{\tempcounta}{#5}%
     \vertsize{\tempcountb}{#6}%
     \ifnum \tempcounta<\tempcountb \tempcounta=\tempcountb \fi
\else
     \vertsize{\tempcounta}{#4}%
     \vertsize{\tempcountb}{#5}%
     \ifnum \tempcounta<\tempcountb \tempcounta=\tempcountb \fi
\fi
\advance \tempcounta by 60
\puthmorphism(\xpos,\ypos)[#2`#3`#5]{#7}{\arrowtypeb}{#9}
\advance\ypos by \tempcounta
\puthmorphism(\xpos,\ypos)[\phantom{#2}`\phantom{#3}`#4]{#7}{\arrowtypea}{#9}
\advance\ypos by -\tempcounta \advance\ypos by -\tempcounta
\puthmorphism(\xpos,\ypos)[\phantom{#2}`\phantom{#3}`#6]{#7}{\arrowtypec}{#9}
}}

\def\setarrowtoks[#1`#2`#3`#4`#5`#6]{%
\def\toka{#1}
\def\tokb{#2}
\def\tokc{#3}
\def\tokd{#4}
\def\toke{#5}
\def\tokf{#6}
}
\def\hex{\@ifnextchar <{\hexp}{\hexp<1000`400>}}
\def\hexp<#1`#2>[#3`#4`#5`#6`#7`#8;#9]{%
\setarrowtoks[#9]
\yext=#2 \advance \yext by #2
\xext=#1 \advance\xext by \yext
\bfig
\putCtriangle<-1`0`1;#2>(0,0)[`#5`;\tokb``\tokd]
\xext=#1 \yext=#2 \advance \yext by #2
\putsquare<1`0`0`1;\xext`\yext>(#2,0)[#3`#4`#7`#8;\toka```\tokf]
\advance \xext by #2
\putDtriangle<0`1`-1;#2>(\xext,0)[`#6`;`\tokc`\toke]
\efig
}
\makeatother

\numberwithin{equation}{subsection}
\begin{document}

\maketitle

\begin{abstract}
Let $\{X_i\}$ be a sequence of compact $n$-dimensional Alexandrov spaces (e.g. Riemannian manifolds) with curvature uniformly bounded below
which converges in the Gromov-Hausdorff sense to a compact Alexandrov space $X$. The paper \cite{alesker-conjectures} outlined (without a proof) a construction
of an integer-valued function on $X$; this function carries additional geometric information on the sequence such as the limit of intrinsic volumes of $X_i$'s.
In this paper we consider sequences of closed 2-surfaces and
(1) prove the existence of such a function in this situation; and (2)
classify the functions which may arise from the construction.
\end{abstract}

\tableofcontents

\section{Introduction}\label{S:Introduction}
\subsection{Background and general overview of main results.}
\begin{paragraphlist}
\item Much work has been done on the behavior of Riemannian manifolds and, more generally, Alexandrov spaces with respect to the Gromov-Hausdorff (GH)
convergence when the sectional curvature is uniformly bounded below, see e.g. \cite{alexander-kapovitch-petrunin}, \cite{burago-burago-ivanov}, \cite{burago-gromov-perelman} and references therein.
The first author formulated in \cite{alesker-conjectures} a few conjectures on the behavior of the intrinsic volumes (or, in equivalent terminology, Lipschitz-Killing curvatures)
on such spaces. One of the new ingredients in terms of which
the conjectures were formulated and which is central for the current paper concerns the so-called constructible functions on the limiting space.
Let $\{X_i^n\}$ be  a sequence of such spaces of dimension $n$ which converges in the GH-sense to a compact Alexandrov space $X$.
Then after choosing a subsequence one can define an integer-valued function $F\colon X\to \ZZ$ (see below). In some sense the function $F$ is unique up to
the natural action of the group of isometries of $X$, see below for the precise statement.
In the no collapse case, i.e. $\dim X=n$, $F$ is equal to 1: first Petrunin \cite{petrunin-personal} mentioned that he possessed a proof, then V. Kapovitch supplied us with a proof of this fact, see Theorem \ref{T:no-collapse}. The function $F$ is constant on the strata of the Perelman-Petrunin stratification \cite{fujioka}.\footnote{Before Fujioka's work \cite{fujioka}, Petrunin \cite{petrunin-personal}
mentioned to the first author that he has a proof, although did not publish it.}
The latter notion was developed in \cite{perelman-petrunin}.

The function $F$ carries extra geometric information for collapsing sequences and was used in the formulation of conjectures in \cite{alesker-conjectures}; a very special case of this connection
is discussed in Subsection \ref{Ss:euler-charact} below.
The construction of $F$ was motivated by the construction of a nearby cycle known in algebraic geometry
(see e.g. \cite{sga7}; for its version in real analytic geometry see \cite{fu-mccrory}, Theorem 3.7) although the current technical set up is very different.

\item The goals of this paper are
\newline
(1) to give a rigorous construction of $F$ in the case of sequences of closed 2-surfaces
with metrics with curvature uniformly bounded below in the sense of Alexandrov (in particular for closed smooth surfaces with Gaussian curvature uniformly bounded from below). Although the general case was very recently
treated in \cite{fujioka}, our treatment of the 2-dimensional case is independent and more elementary.
\newline
(2) we classify the functions $F$ arising on the limit
space in the above situation. Actually we do that in a more precise form discussed below.

As a direct consequence, we verify one of the general conjectured properties of $F$ in the above situation of sequences of closed 2-surfaces.

\item Before we state the main results, let us describe a version of the nearby cycle construction.\footnote{This construction was first suggested as a conjecture by the first author. 
Then Petrunin (see the discussion in \cite{alesker-conjectures}) mentioned that he has a proof. Very recently Fujioka \cite{fujioka} posted a proof on the arxiv.} Let $\{X_i^n\}$ be a sequence of compact $n$-dimensional
Alexandrov spaces  with curvature uniformly bounded from below. Let it GH-converge to a compact Alexandrov space $X$.
 Let $\{d_i\}$ be any metrics on $X_i\coprod X$
extending the original metrics on $X_i$ and $X$ and such that the Hausdorff distance $d_{i,H}(X_i,X)\to 0$ as $i\to \infty$. Let $x\in X$. For $\eps>0$ denote
\begin{eqnarray}\label{Def:cb}
\cb_{i,x}(\eps)=\{y\in X_i|\,\, d_i(y,x)<\eps\}.
\end{eqnarray}
Then there exists a subsequence, denoted in the same way,
with the following properties.
There exists $\eps_0>0$ such that for all $0<\delta_1<\delta_2<\eps_0$ there exists $i_0\in \NN$ (depending on $x,d_i,\eps_0,\delta_1,\delta_2$, and the subsequence)
such that for any $i>i_0$, any $a\in \ZZ_{\geq 0}$,
and any field $\FF$ the image of the natural map in the $a$th cohomology
$$H^a(\cb_{i,x}(\delta_2);\FF)\to H^a(\cb_{i,x}(\delta_1);\FF)$$
has dimension independent of $i,\delta_1,\delta_2$. Let us denote this dimension by $h^a(x)$. The function $x\mapsto h^a(x)$ is unique up to isometries in the following sense.
Let a similar function $\tilde h^a$ be constructed using different metrics $\tilde d_i$ on $X_i\coprod X$ and using a subsequence of the subsequence leading to $h^a$. Then there exists an
isometry $\alp$ of $X$ such that $\tilde h^a=h^a\circ \alp$.

The function in question $F\colon X\to \ZZ$ is defined by
$F(x):=\sum_{a}(-1)^ah^a(x).$

The functions $h^a$ are the main focus of the present paper when $\{X_i\}$ are closed 2-surfaces with Alexandrov metrics with curvature uniformly bounded below.
We show, in particular, that the described construction of $h^a$ is well defined in this case of surfaces.
Furthermore we classify in all cases the $h^a$'s as functions on the limiting space.

\item Before describing the main results, let us remind a few well known facts on the topology and geometry of 2-surfaces. Any topological closed 2-surface is homeomorphic
to exactly one surface from the following list:
\newline
(1) the 2-sphere;
\newline
(2) the connected sum of $g$ copies of the 2-torus where $g\geq 1$.
\newline
(3) the connected sum of $k$ copies of the real projective plane where $k\geq 1$.

The surfaces of types (1) and (2) are orientable, while of type (3) are not orientable. It follows that a surface is defined uniquely up to a homeomorphism
by its Euler characteristic and whether it is orientable or not.

\item Consider now surfaces with metrics with curvature bounded below in the sense of Alexandrov (see Definition \ref{cbbdef} below). Such metrics include smooth Riemannian metrics
with a lower bound on the Gaussian curvature. It is well known that given a lower bound $\kappa$ on the curvature (or on the Gaussian curvature in the smooth case) and an upper bound $D$ on the diameter,
there exist only finitely many homeomorphism types of surfaces admitting
such metrics.
Hence if a sequence of closed surfaces $\{X_i\}$ with curvature uniformly bounded from below GH-converges
to a compact space $X$ then in the sequence there are only finitely many homeomorphism types. Thus after a choice of subsequence one can assume that all $X_i$ have a fixed homeomorphism type.
This will be assumed throughout the rest of text.
\end{paragraphlist}

\subsection{The main results.}

\begin{paragraphlist}
\item Let us state the main results of the paper. Consider a sequence $\{X_i\}$ of closed 2-surfaces of given homeomorphism type with curvature uniformly bounded below. Assume it GH-converges to a compact metric space $X$.
By Burago-Gromov-Perelman \cite{burago-gromov-perelman} (see also \cite{burago-burago-ivanov}) $X$ is an Alexandrov space of (integer) dimension at most 2. If $\dim X<2$ one says that a collapse occurs.
It is well known (see e.g. Proposition \ref{P:surfaces-may-collapse} below) that a collapse may happen precisely for sequences of surfaces of non-negative Euler characteristic,
i.e. for spheres, real projective planes, tori, and Klein bottles.

Theorem \ref{T:no-collapse} below says that in the no collapse case $h^0\equiv 1$ and $h^a\equiv 0$ for $a\ne 0$, in particular $F\equiv 1$; this is not only for surfaces but for any non-collapsing sequences of
compact $n$-dimensional Alexandrov spaces with curvature uniformly bounded from below.
This theorem is an easy consequence of the more general Theorem \ref{alex-conv-balls} below due to V. Kapovitch \cite{kapovitch-private-commun}.

Let us assume now that $\dim X=0$, i.e. $X$ is a point. This case is trivial: by the definition of $h^a$ one has $h^a=\dim H^a(X_i;\FF)$. Note  that all four homeomorphism types of
closed 2-surfaces admitting a collapse (i.e. sphere, torus, real projective plane, and Klein bottle) may actually collapse to a point while the curvature is uniformly bounded from below.

Let us consider the case $\dim X=1$ which is the main one for this paper. Then it is well known that $X$ is isometric either to a circle or a closed segment.
As we have mentioned, by Proposition \ref{P:surfaces-may-collapse} $X_i$ must have non-negative Euler characteristic. Let us consider them case by case.

(1) Let $\{X_i\}$ be homeomorphic to the 2-sphere. Then we show (see Theorem \ref{T:surface-t-circle}(2)) that necessarily $X$ is a segment rather than a circle.\footnote{It is likely that this result is folklore.}
Furthermore Theorem \ref{T:spheres-to-segment} says that $h^0\equiv 1$, $h^a\equiv 0$ for $a\ne 0,1$, and
\begin{eqnarray*}
h^1(x)=\left\{\begin{array}{ccc}
             1&\mbox{if}& x\in int(X),\\
             0&\mbox{if}& x\in \pt X.
            \end{array}\right.
\end{eqnarray*}

(2) Let $\{X_i\}$ be homeomorphic to the real projective plane. By Theorem \ref{T:surface-t-circle}(2) $X$ is a segment but not a circle. By Theorem \ref{T:rp2-segment}
$h^a\equiv 0$ for $a\ne 0,1$, $h^0\equiv 1$. Furthermore $h^1(x)=0$ if $x\in int(X)$, and $h^1(x)=0$ for one of the boundary points of $X$, and $h^1(x)=1$ for another boundary point.

(3) Let $\{X_i\}$ be homeomorphic to the torus. Katz \cite{katz-torii} and independently Zamora \cite{zamora-torii} have shown that $X$ must be a circle but not a segment.
Theorem \ref{T:surface-t-circle} says that in this case $h^a\equiv 0$ for $a\ne 0,1$, and $h^0=h^1\equiv 1$.

(4) Let $\{X_i\}$ be homeomorphic to the Klein bottle. In this case both options of segment and circle for $X$ are possible. In either case $h^a\equiv 0$ for $a\ne 0,1$, and $h^0=h^1\equiv 1$ by Theorem \ref{P:klein-bootles-prop}.

\item As an application of this computation of $h^a$, we verify the general conjectural property of the function $F:=\sum_a(-1)^ah^a$: its integral over the limiting space $X$ with respect to the
Euler characteristic is equal to the Euler characteristic of the $X_i$'s (which are equal to each other after a choice of subsequence), see Proposition \ref{R:euler-characteristic-integral}.

\begin{remark}\label{R:fujioka-result}
This application was proven in a much greater generality in the recent preprint by Fujioka \cite{fujioka}.
\end{remark}

\item The main tools in this paper are the Yamaguchi fibration theorem \cite{yamaguchi-1991} and Alexandrov's realization theorems \cite{Ale1}, \cite{Ale2}.
While the former works for any dimension of $\{X_i\}$, the latter is specific to dimension 2.

\item The motivation of this paper comes from conjectures \cite{alesker-conjectures} on the behavior of the intrinsic volumes (also known as the Lipschitz-Killing curvatures) on Riemannian and Alexandrov spaces under the GH-convergence.

\end{paragraphlist}

\subsection{Organization of the paper.}

 The paper is organized as follows. In Section \ref{S:convexity-hyperbolic} we review convexity in the hyperbolic space. Most of the results of this section are either well known or folklore;
we provide references whenever possible and
provide proofs otherwise. The material of this section will be used in Section \ref{S:collapse-surfaces} in the study of the collapse of 2-spheres.

In Section \ref{Ss:convergence-n-n} we study the functions $h^a$ when the GH-convergence is replaced by the Hausdorff convergence of convex bodies in the  hyperbolic space $\HH^n$.
This will be used in Section \ref{S:collapse-surfaces} in the study of the collapse of 2-spheres and real projective planes via Alexandrov's realization theorem.

In Section \ref{S:yamaguchi-map} we discuss mostly known or folklore facts on Riemannian submersions and the Yamaguchi map between smooth Riemannian manifolds.
This will be used in Section \ref{S:collapse-surfaces} in the study of the collapse of tori and Klein bottles.

In Section \ref{S:convergence-functions-actions} we introduce the notion of $G$-equivariant GH-convergence when $G$ is a finite group and prove a $G$-equivariant version of Gromov's compactness theorem.
Then a version of $G$-equivariant Yamagichi map is constructed in a very special situation sufficient for the purposes of this paper; this is done essentially by repeating Yamaguchi's construction \cite{yamaguchi-1991}.
We will use this material in Section \ref{S:collapse-surfaces} in the study of the collapse of real projective planes and Klein bottles; only the group $G=\ZZ_2$ will be needed.

In Section \ref{S:collapse-surfaces} we prove the main results of this paper.

\subsection{Relation to the work of Shioya and Yamaguchi.}

Shioya and Yamaguchi \cite{shioya-yamaguchi} studied the properties of the Yamaguchi fibration and collapse in the 3-dimensional case. In particular, they studied the homeomorphism types of the metric balls inside collapsing 3-dimensional manifolds. Their techniques could be applied also in the easier 2-dimensional case, and could be used to give an alternative proof of the results of our paper. Note, however, that the methods of our paper are more elementary. Thus their approach relies on deep (even in dimension 2) facts from the Alexandrov geometry such as the generalized soul theorem and the parameterized stability theorem. The former is not used in our paper, the latter is used only in the no-collapse case.
In our study of collapse to the segment we exploit the Alexandrov realization theorem, a totally different tool from those used in \cite{shioya-yamaguchi}.

\hfill

{\bf Acknowledgements.} We express our gratitude to V. Kapovitch who supplied the proof of Theorem \ref{alex-conv-balls}.
The first author is very grateful to A. Petrunin for numerous useful discussions.
The third author is also grateful to his postdoctoral mentor M. Eichmair for his support.
Part of this research was done while the third author was visiting the Tel Aviv University, and the third author is very grateful to the institution for its hospitality.

\section{Convexity in hyperbolic space.}\label{S:convexity-hyperbolic}
In this section we review a few facts on the geometry of convex subsets of the hyperbolic space, CBB($-1$) spaces, and Alexandrov's results on isometric realization of
CBB($-1$) metrics on the 2-sphere as boundaries of convex subsets in hyperbolic 3-space. The principal results of this section are Proposition~\ref{P:convex-hypersurf-convegence} and Theorem \ref{realiz}, but they seem to be folklore.

\subsection{Busemann-Feller lemma.} \label{Ss:buseman-feller-lemma}
The following result due to Milka \cite{milka} is a hyperbolic version of the Busemann-Feller lemma in Euclidean space.
\begin{theorem}\label{L:busemann-feller}
Let $K\subset \HH^n$ be a closed convex subset of the hyperbolic space. Then
\newline
(1) for any point $x\in \HH^n$ there exists a unique nearest point from $K$;
\newline
(2) the map $\HH^n\to K$ sending a point from $\HH^n$ to its nearest point from $K$ is a 1-Lipschitz map.
\end{theorem}
The following corollary should be well known but we have no reference.
\begin{corollary}\label{Cor-monotonicity-length}
Let $K_1\subset K_2\subset \HH^2$ be 2-dimensional convex compact subsets of the hyperbolic plane. Then
$$length(\pt K_1)\leq length(\pt K_2).$$
\end{corollary}
{\bf Proof.} Let $r\colon \pt K_2\to \pt K_1$ be the restriction to $\pt K_2$ of the nearest point map to $K_1$. By the Busemann-Feller
lemma it is 1-Lipschitz. It is easy to see that $r$ is onto. The result follows. \qed

\subsection{Klein model of hyperbolic space.}\label{Ss:model-hyperbolic} The Klein model of the $n$-dimensional hyperbolic space $\HH^n$ is the open unit Euclidean ball
$\BB^n=\{x_1^2+\dots+ x_n^2<1\}$ equipped with the Riemannian metric
$$ds^2=\frac{\sum_{i=1}^n dx_i^2}{1-|x|^2}+\frac{\left(\sum_{i=1}^n x_idx_i\right)^2}{(1-|x|^2)^2},$$
where $|x|^2=\sum_{i=1}^n x_i^2$ is the square of the Euclidean norm.

The relevant property of this metric to be used later is that its geodesic lines are precisely the intersections of affine lines in $\RR^n$ with the open ball $\BB^n$.

\subsection{CBB($-1$) metrics}

Let us recall a few definitions. Let $S$ be a topological 2-dimensional manifold, i.e. a surface. Let $d$ be a complete intrinsic metric on $S$ inducing the given topology on it.
By intrinsic we mean that the distance between each pair of points is equal to the infimum of lengths of all rectifiable paths connecting them. Then this infimum is achieved: there exists a shortest path between any two points. This is a corollary of the Arzela--Ascoli theorem, see \cite{burago-burago-ivanov}, Theorem 2.5.23.
 Let $\psi$, $\chi$ be two shortest paths in $(S, d)$ issuing from a common point $x$. Let $y \in \psi$ be the point at distance $a$ from $x$ and $z \in \chi$ be the point at distance $b$ from $x$. Consider the hyperbolic triangle with side lengths $a$, $b$ and $ d(y,z)$ and let $\lambda(a,b)$ be the angle opposite to the side of length $ d(y,z)$. The following definition was introduced in \cite{burago-gromov-perelman}, Section 2.7, in the multi-dimensional situation.

\begin{definition}
\label{cbbdef}
We say that $d$ is a \emph{CBB($-1$) metric} on $S$ if $d$ is complete, intrinsic and for each $x \in S$ there exists a neighborhood $U \ni x$ such that the function $\lambda(a,b)$ is a nonincreasing function of $a$ and $b$ for every $\psi$, $\chi$ issuing from $x$, in the range $a \in [0; a_0]$, $b \in [0; b_0]$ where the respective points $y, z$ belong to $U$.
\end{definition}

Important examples of CBB($-1$) surfaces are convex surfaces in the hyperbolic space $\HH^3$ endowed with the induced intrinsic metric; in case they are $C^2$-smooth the latter condition is equivalent to have the Gaussian curvature at least $-1$.

Another example of a CBB($-1$) surface is the double cover of a planar convex set (\cite{Ale2}, Ch. I). Let $K\subset \HH^2$ be a convex compact set with non-empty interior.
Let us consider the surface $DK$ equal to  the union of two copies of $K$ such that corresponding points of their respective boundaries are identified. The surface $DK$ is homeomorphic to the 2-sphere $\SS^2$.
It is equipped with the unique intrinsic metric such that the two copies of $K$ inside $DK$ are isometrically embedded. Then $DK$ is CBB($-1$).

\subsection{From the Hausdorff convergence to the Gromov--Hausdorff convergence}

We recall
\begin{definition}\label{D:uniform-convergence}
Let $X$ be a set. One says that a sequence of metrics $\{d_i\}$ on $X$ converges to a metric $d$ uniformly if
$$\sup_{x,y\in X}|d_i(x,y)-d(x,y)|\to 0 \mbox{ as } i\to \infty.$$
\end{definition}

The following result is well-known to experts, but we did not find an appropriate reference, so we are giving our proof.
Some parts of the argument are inspired by the classical treatment of the 3-dimensional Euclidean case by Alexandrov in~\cite{Ale2}, although Alexandrov does not consider the Gromov--Hausdorff convergence.

\begin{proposition}\label{P:convex-hypersurf-convegence}
Let $\{K_i\}\subset \HH^n$ be a sequence of $n$-dimensional convex compact sets converging in the Hausdorff sense to a convex compact set $K$.
Consider the sequence of their boundaries $\{\pt K_i\}$ equipped with the induced intrinsic metrics.
\newline
(1) If $\dim K=n$, then $\pt K_i\overset{GH}{\to}\pt K$ where $\pt K$ is equipped with the intrinsic metric.
\newline
(2) If $\dim K=n-1$, then $\pt K_i$ GH-converges to the double cover $DK$ of $K$.
\newline
(3) If $\dim K\leq n-2$, then $\partial K_i\overset{GH}{\to} K$ where $K$ is equipped with the metric induced from $\HH^n$ which is automatically intrinsic.
\newline
(4) Let us denote by $D^{in}_i$ the intrinsic distance on $\pt K_i$, and by $dist$ the distance in $\HH^n$. If $\dim K\leq n-2$, then
$$\lim_{i\to\infty}\sup_{x,y\in \pt K_i}|D^{in}_i(x,y)-dist(x,y)|= 0.$$
\end{proposition}

First let us prove a lemma.
\begin{lemma}\label{L:hausdorff-boundary-small-dim}
Let $\{K_i\}\subset \HH^n$ be a sequence of $n$-dimensional convex compact sets converging in the Hausdorff sense to a convex compact set $K$. If $\dim K\leq n-1$ then $\{\pt K_i\}\to K$ in the
Hausdorff sense when the boundaries $\{\pt K_i\}$ are equipped with the induced extrinsic metrics from $\HH^n$.
\end{lemma}
{\bf Proof.} Let $\eps>0$. It suffices to show that for large $i$
\begin{eqnarray}\label{E:000inclusion}
K_i\subset (\pt K_i)_\eps.
\end{eqnarray}

Let $\ch$ be a hyperplane containing $K$. For large $i$ one has $K_i\subset \ch_{\eps/2}$. Let $x\in K_i$. Let $y\in \ch$ be the point from $\ch$ nearest to $x$ and $l$ be the geodesic line passing through $x$, $y$ and orthogonal to $\ch$ (so $l$ is unique when $x \neq y$). Then the intersection $l\cap K_i$ is non-empty segment or a point, it is contained in $\ch_{\eps/2}$, and has length at most $\eps$. The end points of it belong to $\pt K_i$.
But each of them is within distance $\eps$ from $x$. Thus (\ref{E:000inclusion}) follows. \qed

\begin{proof}

(1) For $x, y \in \pt K$ by $d(x,y)$ we denote the intrinsic distance of $\pt K$ between them. Let $o \in int(K)$. Then for all sufficiently large $i$ we have $o \in int(K_i)$.
Pull back $D^{in}_i$ and $d$ to the unit sphere $S \subset T_o\HH^n$ via the radial map, and, abusing the notation, continue to denote the obtained metrics by $D^{in}_i$ and $d$. We show the following claim which implies case (1).
\begin{claim}\label{Cl:uniform-radial}
$D^{in}_i$ converge uniformly to $d$.
\end{claim}

For $x \in S$ we denote by $h_i(x)$, $h(x)$ the length of the geodesic from $o$ to $\pt K_i$, $\pt K$ respectively in the direction of $x$. As $\pt K_i$ converge to $\pt K$ in the Hausdorff sense, $h_i$ converge uniformly to $h$. We can choose two sequences $t'_i$ and $t''_i$ of positive real numbers, converging to zero, such that for the functions $h'_i$, $h''_i$ on $S$, defined by the equations
$$\tanh h'_i(x):=e^{-t'_i}\tanh h_i(x),$$
$$\tanh h''_i(x):=e^{t''_i}\tanh h_i(x),$$
we have $h'_i \leq h$ and $h \leq h''_i$ everywhere on $S$. Then the surfaces $\partial K'_i$, $\partial K''_i$, defined by these radius functions, converge to $\partial K$ in the Hausdorff sense.

We show that $\partial K'_i$, $\partial K''_i$ are convex surfaces. Fix $t\in\RR$ and consider a homeomorphism $H:\HH^n \rightarrow \HH^n$ sending a point $p \in \HH^n$ to the point $H(p)$ so that $H(p)$ belongs to the ray $op$ and
$$\tanh(dist(o, H(p)))=e^{-t}\tanh(dist(o, p)).$$
Let $L$ be a line in $\HH^n$ not passing through $o$ and $p$ be the closest point at $L$ to $o$. From the sine law it is easy to see that the line passing through $H(p)$ orthogonal to the ray $op$, is the image of $L$ via $H$. Hence, $H$ sends lines to lines, and, in particular, preserves convexity. Thus, $\partial K'_{i}$, $\partial K''_i$ are convex.

Define
$$\zeta'_i:=\sup_{x \in S} \big(h(x)- h'_i(x)\big),$$
$$\zeta''_i:=\sup_{x \in S} \big(h''_i(x)- h(x)\big).$$
By $d'_i$, $d''_i$ denote the intrinsic metric of $\partial K'_i$, $\partial K''_i$ transferred to $S$. Due to the Busemann--Feller lemma and the triangle inequality
\begin{eqnarray}\label{E:ppp1}
d''_i\leq d+2\zeta''_i, \,\,\, d\leq d'_i+2\zeta'_i.
\end{eqnarray}

Fix some positive real numbers $m< M$. Consider the space $\mathcal T$ of triangles (up to isometry) $oxy$ in $\HH^2$ such that the distance from the line $xy$ to $o$ is at least $m$ and the lengths $ox, oy$ are at most $M$. This space can be parametrized by three numbers: the distance $s_1$ from $o$ to the line $xy$, the length $s_2$ of the largest side among $ox$ and $oy$, and the length $s_3$ of $xy$. The advantage of this parametrization is that all the angles of $oxy$ extend continuously to the closure of $\mathcal T$ in the parameter space, obtained by adding degenerate triangles with $s_3=0$.

Let $oxy$ be a triangle from $\mathcal T$. Denote by $a$ the length $xy$, by $b$ the length $oy$, and by $\beta$ the angle $oxy$. Fix $t \in \mathbb{R}$ and let $ox'y'$ be the triangle with the same angle between the sides $ox'$, $oy'$, and with
$$\tanh(ox')=\tanh(ox)e^{-t},$$
$$\tanh(oy')=\tanh(oy)e^{-t}.$$
Denote by $a'$, $b'$ and $\beta'$ the respective components of $ox'y'$. Comparing the sine laws for the triangles $oxy$ and $ox'y'$ (note that they have a common angle) we get
$$\frac{\sinh a}{\sinh a'}=\frac{\sin\beta'\sinh b}{\sin\beta\sinh b'}=\frac{\sin\beta'\cosh b\tanh b}{\sin\beta\cosh b'\tanh b'}=\frac{\sin\beta'\cosh b}{\sin\beta\cosh b'}e^{t}.$$
Note that for fixed $t$ and for $oxy$ varying in $\mathcal T$, the angles $\beta, \beta'$ belong to a compact set bounded away from 0 and $\pi$. Also the lengths $b$ and $b'$ are bounded from above. Define
$$\xi_t:=\sup_{\mathcal T} \left(\frac{a}{a'}-1\right),$$
(hence, $\xi_t$ also depends on $m, M$, but we ignore it).
Then $\xi_t \rightarrow 0$ as $t \rightarrow 0$.

Take $m$ such that the ball $B$ of radius $m$ centred at $o$ belongs to the interior of $K_i$ for all large $i$, and $M$ such that all $h_i<M$. Any rectifiable curve on $\partial K_i$ can be approximated by an inscribed polygonal curve $\gamma$ that does not intersect $B$. Consider all its vertices, and map them to $\partial K'_i$ via the radial projection from $o$. Let $\gamma'$ be the polygonal curve passing through these points, so it is inscribed in $\partial K'_i$. Then
$$length_{\HH}(\gamma)\leq length_{\HH}(\gamma')(1+\xi_{t'_i}),$$
where $\xi_{t'_i}$ is defined as above. By passing to the limit, we also get $D^{in}_i\leq d'_i(1+\xi_{t'_i})$. From what we have just proved, $\xi_{t'_i}$ tend to zero as $i$ grows. Similarly, $d''_i \leq D^{in}_i(1+\xi_{t''_i})$ and $\xi_{t''_i}$ also tend to zero as $i$ grows. This and (\ref{E:ppp1}) imply that
$D^{in}_i$ converge uniformly to $d$. This finishes the proof of case (1).

For the next cases we need the following lemma:

\begin{lemma}
\label{plane}
Let $\{K_i\} \subset \HH^2$ be a sequence of 2-dimensional convex compact sets converging in the Hausdorff sense to a subset $I$ which is either a compact segment or a point.
Let $x_i, y_i \in \partial K_i$ be sequences of points converging to points $x, y \in I$. If both $x, y \in relint(I)$, we additionally assume that for all sufficiently large $i$, there are outer normal rays to $K_i$ at $x_i, y_i$, which belong to the same half-space with respect to the line containing $I$. Then $D^{in}_i(x_i, y_i) \rightarrow dist(x, y)$.
\end{lemma}

Here by an outer normal ray to $K_i$ at $x_i \in \partial K_i$ we mean a ray starting at $x_i$ and making both angles at least $\pi/2$ with any chord of $K_i$ at $x_i$.

{\bf Proof.}
From the Busemann--Feller lemma one easily gets
$$\liminf D^{in}_i(x_i, y_i) \geq dist(x, y).$$
Now we need a converse inequality. Let $\e_i$ be twice the Hausdorff distance between $K_i$ and $I$, so $K_i \subset I_{\e_i}$, where $I_{\e_i}$ is the $\e_i$-neighborhood of $I$.
Draw outer normal rays to $K_i$ at $x_i$ and $y_i$. If $x, y \in relint(I)$, then we assume that $i$ is sufficiently large and the rays belong to the same halfspace with respect to the line containing $I$. Let these rays intersect $\partial (I_{\e_i})$ in points $x'_i$, $y'_i$. By $d'_i(x'_i, y'_i)$ denote the intrinsic distance between $x'_i$ and $y'_i$ at $\partial (I_{\e_i})$. The nearest-point map to $K_i$ sends $x'_i$, $y'_i$ to $x$, $y$ respectively. Hence, $d'_i(x'_i, y'_i) \geq D^{in}_i(x_i, y_i)$, due to the Busemann-Feller lemma. As $\e_i \rightarrow 0$, $x'_i$ and $y'_i$ converge to $x$ and $y$ respectively. We divide the boundary of $I_{\e_i}$ into four closed arcs: arcs (i) and (iii) are half-circles centred at endpoints of $I$; arcs (ii) and (iv) are constant curvature arcs connecting arcs (i) and (iii). The length of arcs (i) and (iii) tends to zero as $\e_i \rightarrow 0$, and the lengths of arcs (ii) and (iv) tends to the length of $I$. If $x, y \in relint(I)$, then the condition on inner normals imply that $x'_i$, $y'_i$ can not belong to two different arcs among the arcs (ii) and (iv). Then $d'_i(x'_i, y'_i) \rightarrow dist(x, y)$. This implies
$$\limsup D^{in}_i(x_i, y_i) \leq dist(x, y).$$
\end{proof}

Now we return to the proof of Proposition~\ref{P:convex-hypersurf-convegence}.

(2) Let $\Pi$ be the hyperplane containing $K$. Let $K'_i$ be the orthogonal projection of $K_i$ to $\Pi$. It is evident that $K'_i$ converge to $K$ in the Hausdorff sense and $K_i'$ are convex compact sets.

Let $o \in relint(K)$. Then for all sufficiently large $i$ we have $o \in relint (K'_i)$. Let $x' \in K'_i$, $z'$ be the intersection point of the ray $ox'$ with $\partial K'_i$ and $z$ be the intersection point of this ray with $\partial K$. Consider a map $f_i: K'_i \rightarrow K$ sending $x'$ to the point $x$ at the ray $ox'$ such that
$$\frac{dist(o,x)}{dist(o, x')}=\frac{dist(o, z)}{dist(o, z')}.$$
The right side converges uniformly (for $z \in \partial K$) to 1, hence also the left side. As $K'_i$ are uniformly bounded, for every $\e>0$ and sufficiently large $i$ we get that $$dist(x,x')=|dist(o, x)-dist(o, x')|<\e.$$

Take $x', y' \in K'_i$, denote their $f_i$-images by $x,y$. Then we have
$$dist(x,y)\leq dist(x',y')+dist(x,x')+dist(y,y'),$$
$$dist(x',y')\leq dist(x,y)+dist(x,x')+dist(y,y').$$
Hence,
$$|dist(x,y)-dist(x',y')|\leq dist(x,x')+dist(y,y').$$
Thus, for every $\e>0$ and all sufficiently large $i$ we get $|dist(x,y)-dist(x',y')|<\e$, so $f_i$ is an $\e$-isometry. Clearly $f_i(\pt K_i')= \pt K$. Then $f_i$ extends to a $2\e$-isometry sending the double cover of $K'_i$ to the double cover of $K$. By \cite{burago-burago-ivanov}, Corollary 7.3.28, the Gromov--Hausdorff distance between $DK_i'$ and $DK$ is at most $4\e$.

By $DK'_i$ we denote the double cover of $K'_i$ and by $d'_i$ denote its intrinsic metric. Let $f'_i: \partial K_i \rightarrow DK'_i$ be the map coming from the orthogonal projection of $K_i$ to $K'_i$. More exactly, we call one copy of $K'_i$ in $DK'_i$ \emph{upper} and the second copy \emph{lower}. Orient $\Pi$ arbitrarily and for $x' \in relint(K'_i)$ consider the line $L$ orthogonal to $\Pi$ oriented positively with respect to the orientation of $\Pi$. Then $L$ intersects $\partial K'_i$ in two points. We map the first point (with respect to the orientation of $L$) to the lower copy of $x'$, and the second point to the upper copy of $x'$. All points of $\partial K_i$ projecting to $\partial K'_i$ are mapped naturally to their images in $DK'_i$.

By $\partial K^+_i$, $\partial K^-_i$ denote the $f'_i$-preimages of the upper and the lower copies of $relint(K'_i)$. Define the \emph{distortion} of $f'_i$:
$$s_i:=\{\sup |D^{in}_i(x,y)-d'_i(x',y')|: x,y \in \partial K_i, x'=f'_i(x), y'=f'_i(y)\}.$$

For every $\e>0$ and sufficiently large $i$, the Hausdorff distance between $K_i$ and $K'_i$ is less than $\e$. Define
$$r^+_i:=\sup_{x, y \in \pt K^+_i} \big( D^{in}_i(x, y) - dist(x,y) \big),$$
$$r^-_i:=\sup_{x, y \in \pt K^-_i} \big( D^{in}_i(x, y) - dist(x,y) \big).$$
Note that the quantities at the right are clearly non-negative for all $x, y$. Suppose that $r^+_i$ does not converge to zero. Then, up to passing to a subsequence, there exist sequences of points $x_i, y_i \in \partial K_i$, $x_i \neq y_i$, converging to $x, y \in K$ such that $D^{in}_i(x_i, y_i) \geq dist(x, y)+\e$ for some $\e>0$ (observe that $dist(x_i, y_i)$ converges to $dist(x, y)$).

By $P_i$ denote the 2-plane passing through $x_i$ and $y_i$ and orthogonal to $\Pi$, by $x'_i$, $y'_i$ denote the orthogonal projections of $x_i$, $y_i$ to $\Pi$. Up to passing to a subsequence, $P_i$ converge to a 2-plane $P$ orthogonal to $\Pi$. The intersection $P_i \cap K_i$ is a convex 2-dimensional set, and $P_i \cap K_i$ converge to a subset $I$ of $P \cap K$ in the Hausdorff sense. We have $$\dim(I)\leq\dim(P \cap K)\leq 1.$$
Map isometrically all $P_i$ to $\HH^2$ so that all $x'_i$ are mapped to the same point $\tilde x'$ and rays $x'_iy'_i$ are mapped to the same ray $\tilde \Pi^+$ originating in $\tilde x'$.  Denote the images of $P_i\cap K_i$, $x_i$, $y_i$ and $y'_i$ by $\tilde K_i$, $\tilde x_i$, $\tilde y_i$ and $\tilde y'_i$ respectively. By $\tilde D^{in}_i(\tilde x_i, \tilde y_i)$ denote the intrinsic distance of $\partial \tilde K_i$ between $\tilde x_i$ and $\tilde y_i$. Clearly, $\tilde D^{in}_i(\tilde x_i, \tilde y_i)\geq D^{in}_i(x_i, y_i)$.

Since $y'_i$ converge to $y$, points $\tilde y'_i$ converge to a point $\tilde y$. Since $P_i \cap K_i$ converge in the Hausdorff sense to $I$, $\tilde K_i$ converge in the Hausdorff sense to a (possibly degenerate) segment $\tilde I$, which is an isometric embedding of $I$, passing through $\tilde x$ and $\tilde y$. As $x_i, y_i \in \partial K^+_i$, outer normals to $\tilde K_i$ at $\tilde x_i$, $\tilde y_i$ are in the same halfspace with respect to the line $\tilde \Pi$ containing the ray $\tilde \Pi^+$. Then they satisfy the conditions of Lemma~\ref{plane} and we get $\tilde D^{in}_i(\tilde x_i, \tilde y_i) \rightarrow dist(\tilde x, \tilde y)=dist(x,y)$, which is a contradiction.

We obtained that for any $\e>0$ and all sufficiently large $i$, $r^+_i < \e$. The same holds for $r^-_i$. Finally, for all sufficiently large $i$ we have $K_i \subset \Pi_{\frac{\e}{2}}$, where $\Pi_{\frac{\e}{2}}$ is the $\frac{\e}{2}$-neighborhood of $\Pi$. This implies that if $x, y \in \pt K_i$ are such that $f'_i(x)=f'_i(y) \in \pt K'_i$, then $D^{in}(x,y)=dist(x,y)<\e$. Altogether this gives $s_i<3\e$. Then the Gromov--Hausdorff distance between $\partial K_i$ and $DK'_i$ is less than $6\e$, see \cite{burago-burago-ivanov}, Theorem 7.3.25. As $DK'_i$ converges to $DK$ in the Gromov--Hausdorff sense, the same holds for $\partial K_i$.

(3) Let  us show that (3) follows from (4). Indeed (4) implies that the uniform distance and hence the GH-distance between $(\pt K_i,D^{in}_i)$ and $(\pt K_i,dist)$ tends to 0. By Lemma \ref{L:hausdorff-boundary-small-dim}
$(\pt K_i,dist)\overset{GH}{\to} (K,dist)$. This implies (4).

It remains to prove (4).
We proceed similarly to the second half of the proof of case (2). Define
$$r_i:=\sup_{x, y \in \pt K_i} \big( D^{in}_i(x, y) - dist(x,y) \big).$$
We show that $r_i \rightarrow 0$.

Indeed, otherwise, up to passing to a subsequence, we have two sequences of points $x_i, y_i \in \partial K_i$, $x_i \neq y_i$, converging to points $x, y \in K$ such that $D^{in}_i(x_i, y_i)\geq dist(x, y)+\e$ for some $\e>0$. We note that it means that for sufficiently large $i$ the segments $x_iy_i$ do not entirely belong to $\partial K_i$. We first suppose that $x \neq y$. 

Let $n_i$ be an outer normal ray to $K_i$ at $x_i$. Note that as $x_iy_i$ does not entirely belong to $\partial K$, $n_i$ is not orthogonal to $x_iy_i$. By $Q_i \subset \HH^n$ denote a geodesic subspace of dimension $n-2$ passing through $x_i$ and orthogonal to the line $x_iy_i$ and $n_i$. (Generically it is unique, except the case when $x_iy_i$ and $n_i$ are collinear.) Up to passing to a subsequence, $Q_i$ converge to a geodesic subspace $Q$ of dimension $n-2$, containing $x$ and orthogonal to $xy$. Since $\dim(K) \leq n-2$, $Q$ contains a direction $m$ orthogonal to the span of $K$. Consider a sequence of directions $m_i \in Q_i$ converging to $m$. By $P_i$ denote the 2-plane spanned by $m_i$ and the line $x_iy_i$. Up to passing to a subsequence, they converge to the 2-plane $P$ containing $m$ and $xy$.

The intersection $P_i \cap K_i$ is a convex 2-dimensional set, and $P_i \cap K_i$ converge to a subset $I$ of $P \cap K$ in the Hausdorff sense. We have $$1\leq\dim(I)\leq\dim(P \cap K)\leq 1.$$
Embed isometrically all $P_i \cap K_i$ to $\HH^2$ as convex sets $\tilde K_i$ such that all $x_i$ are mapped the same point $\tilde x \in \HH^2$ and rays $x_iy_i$ are mapped to the same ray originating at $\tilde x$. By $\tilde y_i$ we denote the image of $y_i$ and by $\tilde D^{in}_i(\tilde x, \tilde y_i)$ denote the intrinsic distance between $\tilde x$ and $\tilde y_i$ in $\partial \tilde K_i$. Clearly, $\tilde D^{in}_i(\tilde x, \tilde y_i) \geq D^{in}_i(x_i, y_i)$.

Since $y_i$ converge to $y$, $\tilde y_i$ converge to a limit point $\tilde y$. Since $P_i \cap K_i$ converge in the Hausdorff sense to $I$, $\tilde K_i$ converge in the Hausdorff sense to a segment $\tilde I$, which is the isometric embedding of $I$, passing through $\tilde x$ and $\tilde y$. By the choice of $P_i$, the orthogonal projection of $n_i$ to $P_i$ is nonzero, is collinear to $x_iy_i$ and is an outer normal ray to $P_i \cap K_i$. This implies that $\tilde x$ is an endpoint of $\tilde I$. Hence, $\tilde x$ and $\tilde y_i$ satisfy the conditions of Lemma~\ref{plane}. Then $\tilde D^{in}_i(\tilde x, \tilde y_i) \rightarrow dist(\tilde x, \tilde y)=dist(x,y)$, which is a contradiction.

If $x=y$, then we just choose $P$ to be a 2-plane orthogonal to the span of $K$ at $x$, and as $P_i$ we choose any sequence of 2-planes containing $x_iy_i$ and converging to $P$. Then doing everything as above, we see that $I$ and $\tilde I$ are just points, and Lemma~\ref{plane} shows that $\tilde D^{in}_i(\tilde x, \tilde y_i) \rightarrow 0=dist(x,y)$, so we get a contradiction again. \qed

\subsection{Isometric imbedding of 2-sphere with antipodal involution.}\label{Ss:sphere-with-involution}

The following result is the Alexandrov imbedding theorem.
\begin{theorem}[Alexandrov \cite{Ale2}, Ch. XII, \S 2]\label{T:alexandrov-isometric}
Given the 2-sphere $\SS^2$ with CBB($-1$) metric $d$. Then there exists a compact convex set $K \subset \HH^3$ of dimension either 3 or 2 such that $(\SS^2, d)$ is isometric to $\partial K$ in the former case, or it is isometric to the double cover $DK$ in the latter case.
\end{theorem}

In this section we are going to prove the following its variation:

\begin{theorem}
\label{realiz}
Let $d$ be a CBB($-1$) metric on the 2-sphere $\SS^2$ invariant with respect to the antipodal involution $\iota: \SS^2 \rightarrow \SS^2$.
Then there exists a compact convex set $K \subset \HH^3$ of dimension either 3 or 2 such that $(\SS^2, d)$ is isometric to its boundary (in the former case),
or it is isometric to the double cover $DK$ (in the latter case) and $K$ is symmetric with respect to a point so that this symmetry $I: \HH^3 \rightarrow \HH^3$ induces $\iota$.
\end{theorem}


This result follows from the combination of Theorem~\ref{T:alexandrov-isometric} and Pogorelov's rigidity theorem in $\HH^3$. The latter says that the realization is unique in a strong sense. For 3-dimensional convex bodies in $\HH^3$ this means that every isometry of the boundaries is induced by an isometry of $\HH^3$. We do not put further restrictions on the boundaries, so they may be neither smooth, nor polyhedral. One needs to be more careful to include also the 2-dimensional cases, compare with the polyhedral version below, Theorem~\ref{rigp}. However, proofs of Pogorelov's rigidity theorem are notoriously difficult even in the Euclidean space, see~\cite{Pog}. For the hyperbolic case Pogorelov provided in~\cite{Pog} an intricate outline how to reduce it to the Euclidean case, this outline was completed by Milka in~\cite{Mil2}.

We would like to point out that this seems to be slightly excessive for Theorem~\ref{realiz}. Its proof can be obtained from the classic approach of Alexandrov to Theorem~\ref{T:alexandrov-isometric}, see~\cite[Chapter VII]{Ale2}, combined with the rigidity of convex polyhedra in $\HH^3$, which seem to us much more accessible rather than the rigidity of general convex surfaces in $\HH^3$ (and the proofs do not differ much between the Euclidean and hyperbolic cases). Hence, here we sketch a proof of Theorem~\ref{realiz} using these tools. Some steps of Alexandrov's approach to the approximation of CBB metrics by cone-metrics were also verified in CBB($-1$) case by Richard in~\cite[Annex A]{Ric}.

Let us have a few preparations.
\begin{definition}
A \emph{hyperbolic cone-metric} $d$ on $\SS^2$ is locally isometric to the metric of hyperbolic plane except finitely many points called \emph{conical points}. At a conical point $v$ the metric $d$ is locally isometric to the metric of a hyperbolic cone with angle $\lambda_v\neq 2\pi$. A hyperbolic cone-metric is called \emph{convex} if for every conical point $v$ we have $\lambda_v < 2\pi$.
\end{definition}

We will use the following results.

\begin{theorem}[Alexandrov's realization theorem]
\label{realp}
Let $d$ be a convex hyperbolic cone-metric on the 2-sphere $\SS^2$. Then there exists a closed convex polyhedron $K \subset \HH^3$ of dimension either 3 or 2 such that $(\SS^2, d)$ is isometric to its boundary (in the former case), or it is isometric to the double cover $DK$ (in the latter case).
\end{theorem}

In the Euclidean case this is Theorem in~\cite[Section 4.3]{Ale1}. The proof works just the same in the hyperbolic 3-space, as it is noted in~\cite[Section 5.3]{Ale1}.

\begin{theorem}[Alexandrov's rigidity theorem]
\label{rigp}
1) Let $K_1$, $K_2$ be two 3-dimensional compact convex polyhedra in $\HH^3$ and $f: \partial K_1 \rightarrow \partial K_2$ be an isometry. Then there exists an isometry $F: \HH^3 \rightarrow \HH^3$ inducing $f$.

2) Let $K_1$, $K_2$ be two 2-dimensional compact convex polyhedra in $\HH^3$ and $f: DK_1 \rightarrow DK_2$ be an isometry. By $h_1: DK_1 \rightarrow K_1$ and $h_2: DK_2 \rightarrow K_2$ we denote the natural projections of $DK_i$ to $K_i$. Then there exists an isometry $F: K_1 \rightarrow K_2$ such that $F\circ h_1=h_2\circ f$, hence in particular each copy of $K_1$ in $DK_1$ is mapped isometrically onto a copy $K_2$ in $DK_2$.
The isometry $f$ is uniquely determined by $F$ and the image of (for example) the first copy of $K_1$ in $DK_1$.

3) Let $K_1, K_2$ be two compact convex polyhedra in $\HH^3$ such that $K_1$ is 3-dimensional and $K_2$ is 2-dimensional. Then there exists no isometry between $\partial K_1$ and $DK_2$.
\end{theorem}

Again, in the Euclidean case this is the Theorem in~\cite[Section 3.3.2]{Ale1}. Similarly, the proof holds for the hyperbolic 3-space without changes, see~\cite[Section 3.6.4]{Ale1}.

\begin{theorem}
\label{triang}
Let $d$ be a CBB($-1$) metric on $\SS^2$ invariant with respect to $\iota$.
Then it admits a triangulation that is invariant with respect to $\iota$ and consists of finitely many arbitrarily small convex geodesic triangles.
\end{theorem}

This follows from~\cite[Chapter II.6]{Ale2} applied to the induced metric on the projective plane $\iota\backslash \SS^2$. Alternatively, one can also use~\cite[Lemma A.1.2]{Ric}.

\begin{theorem}[Blaschke Selection Theorem]
\label{blas}
A sequence of convex compact sets $K_i$ in $\HH^3$ of uniformly bounded diameters passing through the same point contains a subsequence converging to a compact convex set $K$ in the Hausdorff sense.
\end{theorem}

This version follows easily from \cite{burago-burago-ivanov}, Theorem 7.3.8.

{\bf Proof of Theorem~\ref{realiz}.}
From Theorem~\ref{triang} we consider a $\iota$-invariant triangulation $\mathcal T_i$ of $(\SS^2, d)$ sufficiently fine so that each angle of each triangle is not less than the angle of the respective hyperbolic triangle. One can do this due to the compactness of $\SS^2$. The sum of the angles of all triangles at each vertex is at most $2\pi$. See~\cite[Lemma A.2.1]{Ric}, which adapts the proof of Theorem 2 in~\cite[Chapter VII.4]{Ale2} from the CBB(0) case. Thus, if we replace each triangle of $\mathcal T_i$ by the respective hyperbolic comparison triangle, we obtain a $\iota$-invariant convex hyperbolic cone-metric $d_i$.

The metrics $d_i$ converge to $d$ uniformly. In the Euclidean case this is shown in~\cite[Chapter VII.6]{Ale2}. In the CBB($-1$) case this is done in~\cite[Annex A.2]{Ric}. In particular, the diameters of $d_i$ are bounded.

Due to Theorem~\ref{realp}, there exists a polyhedron $K'_i \subset \HH^3$ such that $(\SS^2, d_i)$ is isometric to its boundary (if the dimension of $K'_i$ is 3) or to $DK'_i$ (if the dimension is 2). In the former case, due to Theorem~\ref{rigp}, the action by $\iota$ is induced by an isometry $I_i: \HH^3 \rightarrow \HH^3$ of order two fixing $\partial K'_i$ as a set, but having no fixed points in it. Due to the classification of isometries of $\HH^3$, $I_i$ is a symmetry of $\HH^3$ around a point. In the latter case, Theorem~\ref{rigp} implies that $K'_i$ is a centrally-symmetric polygon. Then by $I_i$ we consider the central symmetry of $\HH^3$ around the center of the polygon. We compose each $K'_i$ with another isometry to obtain compact convex polyhedra $K_i$ symmetric with respect to the same isometry $I$, which is the central symmetry around a point $o \in \HH^3$.

As the diameters of $d_i$ are uniformly bounded, the diameters of $K_i$ in $\HH^3$ are also uniformly bounded. Thus, by Theorem~\ref{blas} there exists a subsequence converging in the Hausdorff metric to a compact convex set $K \subset \HH^3$, which is $I$-invariant.

Due to Proposition~\ref{P:convex-hypersurf-convegence}(3), $K$ is not a segment or a point. It remains to say that Theorem 4 in~\cite{Ale3} shows that if $dim(K)=3$, then the induced metric on $\partial K$ is $d$. If $dim(K)=2$, then it says that the induced metric on $DK$ is isometric to $d$.
\qed

\section{Collapse of $n$-dimensional convex bodies in hyperbolic $n$-space.}\label{Ss:convergence-n-n}
The main results of this subsection are Propositions \ref{P:isomorphisms-sphere} and \ref{P:isomorph-sphere-boundary}. 

Let a sequence $\{K_i\}\subset \HH^n$ of $n$-dimensional convex compact sets converge to a convex compact set $K$ with $1\leq \dim K \leq n-2$.
We note that for our applications we need only the case $n=3$ and $K$ be a non-degenerate segment, but parts of the proof remain the same in a greater generality, so we consider it as an additional support for the validity of our construction.
By $L$ we denote the geodesic subspace of $\HH^n$ spanned by $K$, i.e., having the same dimension and containing $K$, by $q: \HH^n \rightarrow L$ we denote the nearest-point map to $L$.

Let us fix a point $x\in K$. We denote by $dist$ the distance on $\HH^n$. Set
$$\cb_{i,x}(\delta)=\{z\in\pt K_i|\,\, dist(z,x)<\delta\}.$$
Denote
$$\ca_{i,x}(\delta)=\{z\in \pt K_i|\, \, dist(q(z),x)<\delta\}.$$

\begin{lemma}\label{L:lemma-9}
Let $0<\delta_1<\delta_2$. Then there exists $i_0\in \NN$ such that for any $i>i_0$ and for any $x\in K$ one has
$$\ca_{i,x}(\frac{\delta_1}{2})\subset \cb_{i,x}(\delta_1)\subset \ca_{i,x}(\frac{\delta_1+\delta_2}{2})\subset \cb_{i,x}(\delta_2)\subset \ca_{i,x}(2\delta_2).$$
\end{lemma}
{\bf Proof.}
If the points $z$, $q(z)$ and $x$ are pairwise distinct, the triangle formed by them has right angle at $q(z)$. Thus, $dist(q(z),x) \leq dist(z,x)$ (this is also clearly true if some of these points coincide) and for any $0<\delta<\delta'$ we have $\cb_{i, x}(\delta) \subset \ca_{i, x}(\delta')$. Next, for every $\e>0$ there exists $i_0 \in \NN$ such that for all $i>i_0$ we have $K_i \subset L_{\e}$, where $L_{\e}$ is the $\e$-neighborhood of $L$ in $\HH^n$. For $\delta>0$ by $B(x, \delta)$ we denote the $\delta$-ball in $\HH^n$ centered at $x$, by $A(x, \delta)$ we denote the set of points $z \in \HH^n$ such that $dist(q(z), x)<\delta$. For every $0<\delta<\delta'$ there exists $\e>0$ such that $$\big(A(x, \delta)\cap L_{\e} \big) \subset B(x,\delta').$$
Indeed, one can take $\e$ to be the distance between $L$ and any point from $\big(\partial A(x, \delta)\big)\cap \big(\partial B(x, \delta')\big)$. Hence, it follows that for any $0< \delta<\delta'$ there exist $\e>0$ and $i_0 \in \NN$ such that for all $i>i_0$ we have
$$\ca_{i,x}(\delta)=\big(A(x,\delta)\cap \partial K_i\big) = \big(A(x, \delta)\cap L_{\e} \cap \partial K_i \big) \subset \big(B(x, \delta')\cap \partial K_i\big)= \cb_{i,x}(\delta').$$
 \qed

\begin{lemma}\label{Cor:hyperb-fibration}
Let $K \subset \HH^n$ be an n-dimensional convex compact subset of the hyperbolic n-space, $L \subset \HH^n$ be a totally geodesic subset of dimension $k$, $1\leq k \leq n-2$, $q: \HH^n \rightarrow L$ be the nearest-point map, and $A:=int(q(K))$. Then the restriction of $q$ to $\partial K$ is a topological fibre bundle over $A$ with fiber homeomorphic to the sphere $\SS^{n-k}$.
\end{lemma}

\begin{proof}
This could be easier to perceive in the Klein model. We note that if in the Klein model $L$ passes through the origin, then orthogonality to $L$ is the same both in hyperbolic and in Euclidean metrics.


Let $R:=\max_{x \in K}dist(x, L)$ and $L_R$ be the closed $R$-neighborhood of $L$ in $\HH^n$, hence $K \subset L_R$. Then $\partial L_R$ is homeomorphic to $L\times \SS^{n-k}$ so that the restriction of $q$ to $\partial L_R$ is the projection to the first component. For $x \in \partial K \cap q^{-1}(A)$ by $l_x$ denote the line passing through $x$ orthogonal to $L$. Note that as $l_x\cap L=q(x) \in A$, $l_x$ does not belong to any supporting hyperplane to $K$, hence $l_x$ intersects $int(K)$, thus it intersects $\partial K$ exactly in two points contained in the segment $l_x \cap L_R$. The ray belonging to $l_x$ starting at $x$ outwards $K$ intersects $\partial L_R$ in a unique point, which we denote $f(x)$. The map $f: \partial K \cap q^{-1}(A) \rightarrow \partial L_\e$ is  injective and continuous. Therefore, it is a homeomorphism onto the image, which is $q^{-1}(A)\cong A\times \SS^{n-k}$. By construction, the composition of $f$ with the projection to the first component equals $q$. Now the lemma follows.
\end{proof}

\begin{proposition}\label{P:isomorphisms-sphere}
Let $\{K_i\}\subset \HH^n$ be a sequence of $n$-dimensional convex compact sets converging in the Hausdorff metric to a convex compact set $K$ of dimension $k$, $1 \leq k \leq n-2$.
Let $\kappa>0$ be a constant. Let $0<\delta_1<\delta_2<\frac{\kappa}{2}$.
Then there exists $i_0\in \NN$ such that for any $i>i_0$, for any $x\in int(K)$ with $dist(x,\pt K)>\kappa$ and for an arbitrary commutative group $A$ one has
$$ Im(H^a(\cb_{i,x}(\delta_2);A)\to H^a(\cb_{i,x}(\delta_1);A))\simeq \left\{\begin{array}{ccc}
                                                                       A&\mbox{if}& a=0,n-k\\
                                                                       0&\mbox{otherwise} &
                                                                       \end{array}.\right.$$
\end{proposition}
{\bf Proof.} Consider the natural maps in cohomology induced by inclusions in Lemma \ref{L:lemma-9}
\begin{eqnarray}\label{E:diagram-aaa}
H^*(\ca_{i,x}(\frac{\delta_1}{2}))\overset{a_1}{\leftarrow} H^*(\cb_{i,x}(\delta_1))\overset{a_2}{\leftarrow} H^*(\ca_{i,x}(\frac{\delta_1+\delta_2}{2}))\overset{a_3}{\leftarrow} H^*(\cb_{i,x}(\delta_2))\overset{a_4}{\leftarrow} H^*(\ca_{i,x}(2\delta_2)).
\end{eqnarray}

\hfill

\underline{Step 1.} Let us show that the maps $a_1\circ a_2$ and $a_3\circ a_4$ are isomorphisms. For it suffices to show that for any $0<\delta<\kappa$ the natural map
$$H^*(q^{-1}(x)\cap \pt K_i)\leftarrow H^*(\ca_{i,x}(\delta))$$
is an isomorphism for large $i$. But for large $i$ the set $\{x\in K|\, dist(x,\pt K)>\kappa\}$ belongs to the interior of the image $q(K_i)$. Thus, for large $i$ the restriction $q|_{\partial K_i}$ is a topological fiber bundle over this set with fiber homeomorphic to the sphere $\SS^{n-k}$ by Lemma \ref{Cor:hyperb-fibration}.

\hfill

\underline{Step 2.} It follows from Step 1 and the diagram (\ref{E:diagram-aaa}) that the image
$Im(H^a(\cb_{i,x}(\delta_2))\to H^a(\cb_{i,x}(\delta_1)))$ is isomorphic to $H^a(\ca_{i,x}(\frac{\delta_1}{2}))\simeq H^a(q^{-1}(x)\cap \pt K_i)=H^a(\SS^{n-k})$. The result follows. \qed

\hfill

Now we need to study points at $\partial K$. Here for simplicity we consider only the case when $K$ is a non-degenerate segment $I$. Let $x \in \pt I$ be a point and $0<\delta < length(I)$ be a real number.
The boundary of the set $\ca_{i,x}(\delta)$ are two totally geodesic hyperplanes at distance $\delta$ from $x$ orthogonal to the line $L$. We denote them by $\ch_\delta$ and $\cg_\delta$ such that $\ch_\delta$ intersects $I$.

\begin{lemma}\label{L:topology-boundary}
Let $x\in \pt I$ and $0<\delta<length(I)$. Then for large $i$ the set $\ca_{i,x}(\delta)$ is homeomorphic to a disk. In particular for any commutative group $A$ one has
$$H^a(\ca_{i,x}(\delta); A)\simeq\left\{\begin{array}{ccc}
                                                                       A&\mbox{if}& a=0\\
                                                                       0&\mbox{otherwise} &
                                                                       \end{array}\right.$$
\end{lemma}
{\bf Proof.} In the Klein model $\ch_\delta$ is the usual Euclidean hyperplane intersected with the unit Euclidean ball. For large $i$ the hyperplane $\ch_\delta$ intersects the interior of $K_i$ since both open halfspaces bounded by $\ch_\delta$ must contain points of $K_i$. Also for large $i$ the hyperplane $\cg_\delta$ does not intersect $K_i$. Let $\ch_\delta^+$ be the open halfspace bounded by $\ch_\delta$ and containing $x$.
Then for large $i$ we have $\ca_{i,x}(\delta)=\ch_\delta^+\cap \pt K_i$, which is an open cap homeomorphic to a disc. \qed

\begin{proposition}\label{P:isomorph-sphere-boundary}
Let $\{K_i\}\subset \HH^n$ be a sequence of $n$-dimensional convex bodies converging in the Hausdorff metric to a (non-degenerate) segment $I$.
Let $0<\delta_1<\delta_2<\frac{length(I)}{2}$. Let $x\in \pt I$. Then there exists $i_0$ such that for any $i>i_0$ and any commutative group $A$ one has
$$ Im(H^a(\cb_{i,x}(\delta_2);A)\to H^a(\cb_{i,x}(\delta_1);A))\simeq\left\{\begin{array}{ccc}
                                                                       A&\mbox{if}& a=0\\
                                                                       0&\mbox{otherwise} &
                                                                       \end{array}.\right.$$
\end{proposition}
{\bf Proof.}
By Lemma~\ref{L:lemma-9} for large $i$ the inclusions hold
\begin{eqnarray}\label{E:inclusions-again}
\ca_{i,a}(\frac{\delta_1}{2})\subset \cb_{i,a}(\delta_1)\subset \ca_{i,a}(\frac{\delta_1+\delta_2}{2})\subset \cb_{i,a}(\delta_2)\subset \ca_{i,a}(2\delta_2).
\end{eqnarray}

Let us consider the sequence in cohomology induced by inclusions (\ref{E:inclusions-again})
$$H^*(\ca_{i,x}(\frac{\delta_1}{2}))\overset{a_1}{\leftarrow} H^*(\cb_{i,x}(\delta_1))\overset{a_2}{\leftarrow} H^*(\ca_{i,x}(\frac{\delta_1+\delta_2}{2}))\overset{a_3}{\leftarrow} H^*(\cb_{i,x}(\delta_2))\overset{a_4}{\leftarrow} H^*(\ca_{i,x}(2\delta_2)).$$

It follows from Lemma \ref{L:topology-boundary} that $a_1\circ a_2$ and $a_3\circ a_4$ are isomorphisms. Then it follows from the above diagram that
$$Im(H^*(\cb_{i,x}(\delta_2))\to H^*(\cb_{i,x}(\delta_1)))$$
is isomorphic to, say, $H^*(\ca_{i,x}(\frac{\delta_1}{2}))$. Then proposition follows. \qed

\section{Riemannian submersions and the Yamaguchi map.}\label{S:yamaguchi-map}

\subsection{Riemannian submersions}\label{Ss:riem-submersions}
Let us review a few facts on Riemannian submersions. Let $M,N$ be smooth complete Riemannian manifolds. Let $\pi\colon M\to N$ be a Riemannian submersion.
At any point of $M$ the tangent space of it splits into the direct sum of the vertical subspace, i.e. the tangent space to the fiber of $\pi$, and the horizontal subspace, i.e.
the orthogonal complement to the vertical subspace.

Let $C\colon [a,b]\to N$ be a smooth curve. Fix $c\in [a,b]$ and $\tilde c\in \pi^{-1}(c)$. Then there exists a unique smooth curve $\tilde C\colon[a,b]\to M$, called a horizontal lift of $C$, such that
$$C=\pi\circ \tilde C,$$
 $\frac{d}{dt}\tilde C(t)$ is horizontal for any $t\in [a,b]$, and $\tilde C(c)=\tilde c$.

\begin{lemma}[\cite{cheeger-ebin}, Proposition 3.31]\label{L:geod-horizontal-lift}
In the above notation $C$ is a geodesic if and only if its horizontal lift $\tilde C$ is.
\end{lemma}

\begin{lemma}[\cite{michor}, Corollary 26.12]\label{L:geodes-orthog}
Let $\pi\colon M\to N$ be a Riemannian submersion. Let $\gamma$ be a geodesic in $M$. Assume $\gamma$ is orthogonal to a fiber of $\pi$ at some point. Then $\gamma$
is orthogonal to each fiber it intersects.
\end{lemma}

A geodesic in $M$ orthogonal to all fibers it intersects will be called {\itshape horizontal} geodesics. It follows from Lemmas \ref{L:geod-horizontal-lift}, \ref{L:geodes-orthog} that a geodesic in $M$ is horizontal if and only if it is
a horizontal lift of its image under $\pi$. It is easy to see that the length of any horizontal geodesic and of its image under $\pi$ are equal.

Below $inj(N)$ denotes the injectivity radius of $N$.

\begin{lemma}\label{L:tube-diffeo}
Let $\pi\colon M\to N$ be a Riemannian submersion. Let $x\in N$. Let $0<\delta <inj(N)$.
Let $\cn_\delta$ denote the $\delta$-neighborhood of zero section of the normal bundle of the fiber $\pi^{-1}(x)$.
\newline
(1) For any $ x\in N,p\in M$ one has
$$dist(p,\pi^{-1}(x))= dist(\pi(p),x).$$
\newline
(2) For any point from $\pi^{-1}(B(x,\delta))$ there is exactly one path minimizing the distance from this point to the fiber $\pi^{-1}(x)$.
This path is necessarily a horizontal geodesic.
\newline
(3) For any vector $v\in \cn_\delta$ one has
$$dist(\exp(v),\pi^{-1}(x))=|v|.$$
\newline
(4) If $\pi$ is proper then the exponential map $\exp\colon \cn_\delta\to M$ is a homeomorphism onto $\pi^{-1}(B(x,\delta))$.
\end{lemma}
{\bf Proof.}
(1) First since $\pi$ is 1-Lipschitz one has
$$dist(p,\pi^{-1}(x))\geq dist(\pi(p),x).$$
To prove the opposite inequality let $\gamma$ be a geodesic in $N$  minimizing the distance between $\pi(p)$ and $x$. Let $\tilde\gamma$ be its horizontal lift started at $p$. $\tilde\gamma$ ends necessarily on $\pi^{-1}(x)$.
$\gamma$ and $\tilde \gamma$ have equal length. Hence the opposite inequality follows.

\hfill

(2) Let now $p\in \pi^{-1}(B(x,\delta))$. Let us assume that there exist two geodesics $\gamma_1,\gamma_2$ connecting $p$ to $\pi^{-1}(x)$ and having length $dist(p,\pi^{-1}(x))$.
By minimality both have to be orthogonal to the fiber $\pi^{-1}(x)$. Hence they are horizontal and by Lemma \ref{L:geod-horizontal-lift} their images $\pi\circ \gamma_1,\pi\circ\gamma_2$ are geodesics in $N$; they have the same length
$dist(p,\pi^{-1}(x))\overset{\mbox{part }(1)}{=}dist(\pi(p),x)<\delta$. Since $\delta<inj(N)$ the two geodesics must coincide $\pi\circ \gamma_1=\pi\circ\gamma_2$. Since $\gamma_i$ is the horizontal lift of $\pi\circ \gamma_i$, $i=1,2$,
with the common end point $p$ it follows that $\gamma_1=\gamma_2$.

\hfill

(3) Let $v\in \cn_\delta$.
Let $\gamma(t):=\exp(tv),\, t\in[0,1]$. Then $\gamma$  is a horizontal geodesic since it is orthogonal to the fiber $\pi^{-1}(x)$. When $t\in [0,1]$ its length is equal to the length of $\pi\circ\gamma$ on the one hand, and equals to $|v|$ on the other hand.
Clearly $length(\gamma)=|v|<\delta$. Hence $\gamma\subset \pi^{-1}(B(x,\delta))$. There is a path $\gamma_0$ minimizing the distance between $\exp(v)$ and $\pi^{-1}(x)$; it is necessarily a horizontal geodesic by part (2).
Then the images $\pi\circ \gamma$ and $\pi\circ \gamma_0$ are geodesics contained in $B(x,\delta)$ with equal endpoint $x$ and $\pi(\exp(v))$. Since $\delta<inj(N)$ it follows $\pi\circ \gamma=\pi\circ \gamma_0$. Hence their
horizontal lifts are equal: $\gamma=\gamma_0$. Part (3) follows.

\hfill

(4) Since $\pi$ is 1-Lipschitz, $\exp(\cn_\delta)\subset \pi^{-1}(B(x,\delta))$.

Let us show the opposite inclusion. Let $p\in \pi^{-1}(B(x,\delta))$. Let $\gamma$ be a normal geodesic starting at $p$ and minimizing distance from $p$ to $\pi^{-1}(x)$. By part (3) $\gamma$ is a horizontal geodesic.
Let us denote its end point by $z\in \pi^{-1}(x)$. By part (1)
$$l:=length(\gamma)=dist(p,z)=dist(p,\pi^{-1}(x))=dist(\pi(p),x)<\delta.$$
Then $v:=\frac{d\gamma}{dt}\big|_z$ belongs to the unit normal bundle to $\pi^{-1}(x)$. Then $\gamma(t)=\exp(tv)$, $t\in[0,l]$. In particular $p=\exp(lv)$.
Since $l<\delta$ the converse inclusion follows and hence
$$\exp(\cn_\delta)= \pi^{-1}(B(x,\delta)).$$

Let us prove the injectivity of $\exp$ on $\cn_\delta$. Otherwise there exist $p\in \pi^{-1}(B(x,\delta))$ and two horizontal geodesics $\gamma_1,\gamma_2$ starting on $\pi^{-1}(x)$, ending at $p$, contained in $\pi^{-1}(B(x,\delta))$.
Then their images $\pi\circ \gamma_1$ and $\pi\circ \gamma_2$ are geodesics contained in $B(x,\delta)$, both start at $x$ and end  at $\pi(p)$. Since $\delta< inj(N)$, $\pi\circ \gamma_1=\pi\circ \gamma_2$.
Hence $\gamma_1=\gamma_2$.

Let us show that the inverse map of $\exp$ is continuous. Let $p_i\to p$ be a sequence in $\pi^{-1}(B(x,\delta))$. Set
\begin{eqnarray*}
(z_i,v_i):=\exp^{-1}(p_i)\in\cn_\delta,\\
(z,v):=\exp^{-1}(p)\in\cn_\delta,
\end{eqnarray*}
where  $z_i,z\in \pi^{-1}(x)$ and $v_i\in\cn_\delta|_{z_i},\, v\in \cn_\delta|_{z}$. Since $\pi$ is proper after a choice of subsequence we may assume that $(z_i,v_i)\to (y,w)\in \cn_\delta$.
Then by continuity of $\exp$ one has
$$\exp_y(w)=\exp_z(v).$$
Since $\exp$ is bijective on $\cn_\delta$ it follows that $(y,w)=(z,v)$. \qed

\begin{corollary}\label{Cor:homot-retract}
Let $\pi\colon M\to N$ be a proper smooth Riemannian submersion. Let $\delta<inj(N)$. For any point $x\in N$ the natural imbedding
$$\pi^{-1}(x)\inj \pi^{-1}(B(x,\delta))$$
is a homotopy retraction.
\end{corollary}
{\bf Proof.} The following diagram is commutative
$$\square[\pi^{-1}(x)`\pi^{-1}(x)`\cn_\delta`\pi^{-1}(B(x,\delta));id` ` `\exp] $$
where the horizontal arrow are homeomorphisms by Lemma \ref{L:tube-diffeo}(4) and the vertical ones are closed imbeddings. Since the left arrow is a homotopy
retraction, so is the right one. \qed

\begin{proposition}\label{P:submersion-cohomology}
Let $(M,g)$, $(N,h)$ be smooth Riemannian manifolds. Let $\pi\colon M\to N$ be a smooth proper submersion. Let $0<\delta_1<\delta_2<\frac{1}{10}\min\{inj(N),1\}$.
Let $0<\eps<\frac{\min\{\delta_1,\delta_2-\delta_1\}}{100(1+diam(N))}$. Let us assume that $\pi$ is $\eps$-almost Riemannian submersion, i.e.
$$e^{-\eps}<\frac{|d\pi(v)|}{|v|}<e^{\eps}$$
for any non-zero tangent to $M$ vector $v$ which is horizonal, i.e. orthogonal to the corresponding fiber of $\pi$.

Let $d$ be a metric on $M\bigsqcup N$ extending the original metrics on $M$ and $N$ and such that
\begin{eqnarray}\label{E:haus-distance-technical}
d_H(M,N)<\eps.
\end{eqnarray}
Assume finally that
\begin{eqnarray}\label{E:distance-to-the-image}
d(p,\pi(p))<\eps \mbox{ for any } p\in M.
\end{eqnarray}
For $x\in N$ consider the set $\cb_x(\delta):=\{p\in M|\, d(p,x)<\delta\}.$

Then (evidently)
$$\pi^{-1}(x)\subset \cb_x(\delta_1)\subset \cb_x(\delta_2)$$
and the obvious map
$$Im[H^*(\cb_x(\delta_2))\to H^*(\cb_x(\delta_1))]\to H^*(\pi^{-1}(x))$$
is an isomorphism, where the cohomology has coefficients in an arbitrary commutative group.
\end{proposition}
{\bf Proof.} Let us define the new Riemannian metric $\tilde g$ on $M$ which coincides with $g$ on the vertical subspaces, has the same as $g$ horizontal subspaces, and
on these horizontal subspaces it is the pull-back of the metric $h$. Thus
$$\pi\colon (M,\tilde g)\to (N,h)$$
is a Riemannian submersion. It is easy to see that for any tangent vector $v$ to $M$ one has
$$e^{-\eps} <\frac{|v|_{\tilde g}}{|v|_g}<e^{\eps}.$$
Consequently the lengths of  any curve $c$ on $M$ with respect to the metrics $g$ and $\tilde g$ satisfy
$$e^{-\eps} <\frac{length_{\tilde g}(c)}{length_g(c)}<e^{\eps}.$$
Let us denote by $d_g$ and $ d_{\tilde g}$ the intrinsic metrics on $(M,g)$ and $(M,\tilde g)$ respectively. Then it follows that
$$e^{-\eps} <\frac{d_{\tilde g}(x,y)}{d_g(x,y)}<e^{\eps}\mbox{ for any } x,y\in M.$$
By (\ref{E:haus-distance-technical}) one has $diam(M,g)\leq diam(N,h)+2\eps<diam(N,h)+1$.
Hence
\begin{eqnarray*}\label{E:uniform-dist-comparison}
|d_{\tilde g}(x,y)-d_g(x,y)|\leq (e^{\eps}-1)diam(M,g)\leq 2\eps diam(M,g)<2\eps(diam(N,h)+1)<\\
\frac{\min\{\delta_1,\delta_2-\delta_1\}}{50}.
\end{eqnarray*}

By Lemma \ref{L:tube-diffeo}(1) for $0<\delta<inj(N,h)$ one has
$$\pi^{-1}(B(x,\delta))=\{p\in M|\, d_{\tilde g}(p,\pi^{-1}(x))<\delta\}.$$
Then it follows that
\begin{eqnarray}\label{E:long-inclusion}
\pi^{-1}(B(x,\frac{\delta_1}{2}))\subset \cb_x(\delta_1)\subset \pi^{-1}(B(x,\frac{\delta_1+\delta_2}{2}))\subset \cb_x(\delta_2)\subset\pi^{-1}(B(x,2\delta_1)).
\end{eqnarray}
Also all these 5 spaces contain the fiber $\pi^{-1}(x)$.
Then we get the maps on cohomology
$$H^*(\pi^{-1}(B(x,\frac{\delta_1}{2})))\leftarrow H^*(\cb_x(\delta_2))\leftarrow H^*(\pi^{-1}(B(x,\frac{\delta_1+\delta_2}{2})))\leftarrow H^*(\cb_x(\delta_1))\leftarrow H^*(\pi^{-1}(B(x,2\delta_1)))$$
which all naturally map to $H^*(\pi^{-1}(x))$. For the first, third, and the fifth spaces the latter map is an isomorphism by Corollary \ref{Cor:homot-retract}. A simple diagram chase implies the proposition. \qed

\subsection{The Yamaguchi map.}\label{Ss:reminder-yamaguchi}
Yamaguchi \cite{yamaguchi-1991} has proven the following result.
\begin{theorem}[Yamaguchi]\label{T:yamaguchi}
Fix $m\in \NN$ and $\mu>0$. There exist $\eps_m(\mu)>0$ depending on $m,\mu$ only with the following properties.
Let $M,N$ be smooth closed Riemannian manifolds, $\dim M=m$. Let $sec(M)\geq -1$, $|sec(N)|\leq 1$, and the injectivity radius $inj(N)>\mu$. Assume there exists a metric $d$ on $M\bigsqcup N$
extending the original metrics on $M$ and $N$ and such that
$d_H(M,N)<\eps$, where $0<\eps<\eps_m(\mu)$ is any number.
Then there exists a smooth map $f\colon M\to N$ such that
\newline
(1) $f$ is a $\tau(\eps)$-almost Riemannian submersion, i.e.
$$e^{-\tau(\eps)}<\frac{|df(\xi)|}{|\xi|}<e^{\tau(\eps)},$$
where $\xi$ is any non-zero tangent to $M$ vector orthogonal to the fibers of $f$, and $\tau(\eps)$ is a positive number depending on $m,\mu,\eps$ and such that $\lim_{\eps\to +0}\tau(\eps)=0$;
\newline
(2) $d(z,f(z))<\tau(\eps)$ for any $z\in M$;\footnote{This is equation (2.5) in \cite{yamaguchi-1991}; see references therein.}
\newline
(3) the fibers of $f$ are connected.
\end{theorem}

\begin{corollary}\label{Cor:fiber-Yamag-cor}
Let $\{M_i^m\}$ be a sequence of smooth closed $m$-dimensional Riemannian manifolds with uniformly bounded below sectional curvature which GH-converges to a smooth closed Riemannian manifold $N$.
Let $d_i$ be a metric on $M_i\bigsqcup N$ extending the original metrics on $M_i$ and $N$ and such that
$$d_{i,H}(M_i,N)\to 0.$$
Let $0<\delta_1<\delta_2<\frac{1}{100}\min\{inj(N),1\}.$

Let $f_i\colon M_i\to N$ be the Yamaguchi maps from Theorem \ref{T:yamaguchi} where for $\eps$ one takes $\eps_i:=d_{GH}(M_i,N)+1/i$.
Then there exists $i_0$ such that for all $i>i_0$ the following is satisfied: for any $x\in N$ the natural map
$$Im[H^*(\cb_{i,x}(\delta_2))\to H^*(\cb_{i,x}(\delta_1))]\to H^*(\pi_i^{-1}(x))$$
is an isomorphism, where by definition $\cb_{i,x}(\delta):=\{p\in M_i|\, d_i(p,x)<\delta\}$, and the cohomology is taken with coefficients in an arbitrary commutative group.
\end{corollary}
{\bf Proof.} This follows from Theorem \ref{T:yamaguchi} and Proposition \ref{P:submersion-cohomology}. \qed

\section{GH-convergence of spaces acted on by a finite group.}\label{S:convergence-functions-actions}
Most probably the results of this section are folklore. The main result of this section is Theorem \ref{T:equi-sequence} below.
Let us state a consequence of it:
\newline
Let a sequence $\{X_i\}$ of compact metric spaces GH-converge to a compact metric space $X$. Let each $X_i$ be given an action of a finite group $G$ by isometries. Then there is an action of $G$ on $X$ by isometries such that for a subsequence of $X_i$
one has $G\backslash X_i\to G\backslash X$ in the GH-sense.

\subsection{Equivariant GH-distance.}\label{Ss:equivariant}
Let us fix throughout this subsection a finite group $G$.
\begin{definition}\label{D:equivar-distance}
Let $(X,d_X)$ and $(Y,d_Y)$ be semi-metric spaces  such that the induced quotient metric spaces are compact. Let us given actions of $G$ on $X$ and $Y$ by isometries.
Define the $G$-equivariant GH-distance between them as
$$d^G_{GH}(X,Y):=\inf_d d_H(X,Y),$$
where the infimum is taken over all $G$-invariant semi-metrics on the disjoint union $X\bigsqcup Y$ extending the original semi-metrics on $X$ and $Y$ and such that
$d(x,y)>0$ for any $x\in X,y\in Y$.\footnote{The only difference between this definition and the
original GH-distance is that $d$ has to be $G$-invariant.}
\end{definition}
\begin{remark}
If $(X,d)$ is a semi-metric space acted by $G$ by isometries then the canonical quotient metric space $(\bar X,\bar d)$ carries the induced action of $G$ by isometries and $d^G_{GH}(X,\bar X)=0$.
\end{remark}
\begin{lemma}\label{L1}
Let $G$, $(X,d_X)$, and $(Y,d_Y)$ be as in Definition \ref{D:equivar-distance}.
\newline
(1) Then $d_{GH}(X,Y)\leq d^G_{GH}(X,Y)<\infty$.
\newline
(2) $d^G_{GH}$ is a semi-metric on the class of semi-metric spaces with an isometric $G$-action as in Definition \ref{D:equivar-distance}.
\newline
(3) Assume in addition that $(X,d_X)$ and $(Y,d_Y)$ are metric (rather than semi-metric) spaces. Then $d^G_{GH}(X,Y)=0$ if and only if there exists a $G$-equivariant isometry $X\tilde\to Y$.
\end{lemma}
{\bf Proof.} (1) The first inequality is obvious. For the second one, fix a constant $L> \max\{diam(X),diam(Y)\}$. Define the semi-metric $d$ on $X\bigsqcup Y$, extending the original metrics on $X$ and $Y$, by
\begin{eqnarray*}
d(x,y):=L \mbox{ for any } x\in X,y\in Y.
\end{eqnarray*}
It is clear that $d$ is $G$-invariant and satisfies all the other conditions from the Definition \ref{D:equivar-distance}.

(2) The symmetry of $d^G_{GH}$ is obvious. Let us prove the triangle inequality:
$$d^G_{GH}(X_1,X_3)\leq d^G_{GH}(X_1,X_2)+d^G_{GH}(X_2,X_3).$$
Given $\eps>0$ and two semi-metrics $d_{12}$ and $d_{23}$ on $X_1\bigsqcup X_2$ and $X_2\bigsqcup X_3$ respectively as in Definition \ref{D:equivar-distance}, let us define a new semi-metric $d_{13}$ on  $X_1\bigsqcup X_3$ using the following well known construction
(see e.g. \cite{burago-burago-ivanov}):
$$d_{13}(x_1,x_3)=\inf_{x_2\in X_2}\{d_{12}(x_1,x_2)+d_{23}(x_2,x_3)\}+\eps$$
for any $x_1\in X_1,x_3\in X_3$, and define $d_{13}$ on $X_1$ and on $X_3$ to coincide with the original metrics on those spaces. It is easy to see that $d_{13}$ is a semi-metric satisfying all the requirements in Definition \ref{D:equivar-distance}.
Moreover it is easy to see that the Hausdorff distance satisfies
$$d_H(X_1,X_3)\leq d_H(X_1,X_2)+d_H(X_2,X_3)+\eps.$$
Taking $\inf$ over $d_{12},d_{23}$, and $\eps>0$ we get the triangle inequality for $d^G_{GH}$.

(3) The 'if' part is trivial. For the 'only if' part, let $(X,d_X)$ and $(Y,d_Y)$ be compact metric spaces with $G$-actions by isometries. Let us assume that $d^G_{GH}(X,Y)=0$. Then there exists a sequence  $\{d_i\}$ of $G$-invariant metrics
on $X\bigsqcup Y$ extending the original metrics on $X$ and $Y$ such that the Hausdorff distances with respect to $d_i$ satisfy $d_{i,H}(X,Y)\to 0$ as $i\to \infty$. We are going to construct a $G$-equivariant isometry $\sigma\colon X\to Y$.

Let us fix an ultra-filter $\lam$ on $\NN$. Let $x\in X$. There exist $y_i\in Y$ such that
$d_i(x,y_i)\to 0.$
Define $$\sigma(x):=\lim_\lam y_i.$$
Let us check that $\sigma(x)$ is well defined, i.e. is independent of $y_i$. Indeed if $y_i'\in Y$ is another sequence such that $ d_i(x,y_i')\to 0$ then clearly
$d_Y(y_i,y_i')\to 0$. Hence $\lim_\lam y_i'=\lim_\lam y_i$ as required.

Let us show that $\sigma$ preserves distances:
\begin{eqnarray}\label{E:sigma-isometr}
d_Y(\sigma(x),\sigma(\tilde x))=d_X(x,\tilde x)
\end{eqnarray}
for any $x,\tilde x\in X$. For let us chose $y_i,\tilde y_i\in Y$ such that
$$d_i(x,y_i)\to 0,\, d_i(\tilde x,\tilde y_i)\to 0.$$
By the triangle inequality
$$|d_Y(y_i,\tilde y_i)-d_X(x,\tilde x)|\leq d_i(x,y_i)+ d_i(\tilde x,\tilde y_i)\to 0.$$
This immediately implies (\ref{E:sigma-isometr}).

Let us show that $\sigma$ is an isometry, i.e. onto. Similarly there exists a distance preserving transformation $\tau\colon Y\to X$. Thus
$\sigma\circ \tau\colon Y\to Y$ preserves distances in $Y$. By Lemma 1.2 in \cite{petrunin-98} $\sigma\circ \tau $ is onto. Hence $\sigma$ is onto.

Let us show that $\sigma $ is $G$-equivariant. Let $g\in G,x\in X$. Let $y_i\in Y$ satisfy $d_i(x,y_i)\to 0$. Since $d_i$ are $G$-invariant $d_i(g(x),g(y_i))\to 0$.
Then we have
\begin{eqnarray*}
\sigma(g(x))=\lim_\lam g(y_i)=g(\lim_\lam y_i)=g(\sigma(x)).
\end{eqnarray*}
\qed

\hfill

Let $(\cx,d_{GH})$ be the space of isometry classes of compact metric spaces equipped with the GH-distance. Let $(\cx^G,d^G_{GH})$ denote the space of equivalence classes of compact metric space equipped with $G$-action by isometries with the metric $d^G_{GH}$
(two such spaces are called equivalent if there exists a $G$-equivariant isometry between them).

One has the canonical map $$\Theta\colon \cx^G\to \cx$$
forgetting the action of $G$. By Lemma \ref{L1}(1)  $\Theta$ is 1-Lipschitz.

\begin{proposition}\label{P:quotient-convrgent}
Let $X_i\to X$ in $\cx^G$. Then $G\backslash X_i\to G\backslash X$ in $\cx$.
\end{proposition}
{\bf Proof.} There exist $G$-invariant metrics $d_i$ on $X\bigsqcup X_i$ extending the original metrics on $X$ and $X_i$ such that $d_{i,H}(X,X_i)\to 0$.
Let
$$\pi_i\colon X\bigsqcup X_i\to (G\backslash X)\bigsqcup  (G\backslash X_i)$$
be the obvious quotient map. $\pi_i$ is 1-Lipschitz when the source is equipped with $d_i$ and the target with the quotient metric $\bar d_i$. We claim that for the Hausdorff distances one has
$$d_{GH}(G\backslash X,G\backslash X_i)\leq \bar d_{i,H}(G\backslash X, G\backslash X_i)\leq  d_{i,H}(X,X_i)\to 0.$$
The first inequality is obvious by the definition of $d_{GH}$.
The second inequality follows from the following general simple fact: Let $F\colon Z_1\to Z_2$ be a 1-Lipschitz map of metric spaces. Let $A,B\subset Z_1$ be compact subsets. Then
$$d_H(F(A),F(B))\leq d_H(A,B).$$
A proof is left to the reader. \qed

\begin{proposition}\label{R:properness}
The canonical map $\Theta\colon \cx^G\to \cx$ is proper.
\end{proposition}
{\bf Proof.} Let $\{X_i\}\subset \cx^G$ be a sequence such that $\Theta(X_i)$ has a limit in $\cx$. It follows (see e.g. Proposition 7.4.12 in \cite{burago-burago-ivanov})
that for any $k\in \NN$ there exists $N^{(k)}\in \NN$ and an $1/k$-net\footnote{Recall that a subset $S\subset X$ is called $\eps$-net if its $\eps$-neighborhood equals $X$: $X=S_\eps$.} $S^{(k)}_i\subset X_i$ with at most
$N^{(k)}$ elements. We may and will assume that $S^{(k)}_i$ is $G$-invariant; indeed it suffices to replace it with $G\cdot S^{(k)}_i$ and replace $N^{(k)}$ with $|G|\cdot N^{(k)}$.
We may and will assume that for any $i$
$$S^{(1)}_i\subset S^{(2)}_i\subset S^{(3)}_i\subset \dots\subset X_i.$$
Furthermore after a choice of subsequence one may assume that for any $k$ the number
of elements in $S^{(k)}_i$ is independent of $i$. This number will be denoted by $N^{(k)}$ again.
Let $I:=\{x^1,x^2,x^3,\dots\}$ be either finite or countable set whose cardinality equals the (independent of $i$) cardinality of $S_i:=\cup_{k=1}^\infty S_i^{(k)}$. Let $I^{(k)}:=\{x^p\}_{p\leq N^{(k)}}$. For any $i$ there is a bijection
$\iota_i\colon I\tilde\to  S_i$
such that $$\iota_i(I^{(k)})=S_i^{(k)}.$$
Since $S_i^{(k)}$ is acted by $G$, the bijection $\iota_i$ induces an action of $G$ on $I^{(k)}$ which possibly may depend on $i$. Since a finite group may have only finitely many actions on a given finite set, after a choice of a subsequence we may and will
assume that all these actions of $G$ on $I^{(k)}$ are the same for any $k$. Clearly this action of $G$ on $I^{(k+1)}$ restricts to the corresponding action on $I^{(k)}$. Hence there exists a unique action on $I=\cup_{k=1}^\infty I^{(k)}$
which restricts to the corresponding actions on $I^{(k)}$ for any $k$ and such that the bijections $\iota_i\colon I\to S_i$ are $G$-equivariant.

Let us define on $I$ a sequence of metrics $\{D_i\}$ by
$$D_i(x^p,x^q):=d_{X_i}(\iota_i(x^p),\iota_i(x^q)).$$
Since $I$ is at most countable we may choose a subsequence such that for any $p,q\in I$
$$D_i(x^p,x^q)\to D(x^p,x^q),$$
where $D$ is a semi-metric on $I$. Then $G$ preserves $D$ since $G$ preserves $D_i$ for any $i$.

Let us construct a limit of $\{X_i\}$ in $\cx^G$. Let $(\bar I, \bar D)$ be the metric space corresponding to the semi-metric space $(I,D)$. The former is obtained from the latter by quotient modulo the equivalence relation on $I$:
$x\sim y$ if and only if $D(x,y)=0$ (see Proposition 1.1.5 in \cite{burago-burago-ivanov} for details).

Since $G$ preserves $D$,  this action of $G$ preserves the above equivalence classes and hence induces an action of $G$ on $(\bar I,\bar D)$ by isometries.

The metric space $(\bar I,\bar D)$ is pre-compact. Indeed denote by $\bar I^{(k)}$ the image of $I^{(k)}$ in $\bar I$. It is easy to see that $\bar I^{(k)}$ is $1/k$-net in $(\bar I, \bar D)$ for any $k$.

Let $(X,d_X)$ be the completion of $(\bar I,\bar D)$. It is necessarily compact. The action of $G$ on $\bar I$ extends uniquely to an action by isometries on $X$.

It remains to show that $X_i\to X$ in $\cx^G$. By the triangle inequality for $d^G_{GH}$ we have
\begin{eqnarray*}\label{E:0-1}
d^G_{GH}((X_i,d_{X_i}),(X,d_X))\leq d^G_{GH}((X_i,d_{X_i}),(S_i^{(k)},d_{X_i}))+\\
d^G_{GH}((S_i^{(k)},d_{X_i}), (I^{(k)},D_i))+\\
d^G_{GH}((I^{(k)},D_i),(I^{(k)},D))+\\
d^G_{GH}((I^{(k)},D),(\bar I^{(k)},\bar D))+\\
d^G_{GH}((\bar I^{(k)},\bar D),(X,d_X)).
\end{eqnarray*}

The following estimates are clear
\begin{eqnarray}\label{E:001}
d_{GH}^G((X_i,d_{X_i}),(S_i^{(k)},d_{X_i}))\leq d_H(S_i^{(k)},X_i)\leq 1/k,\\\label{E:002}
d^G_{GH}((\bar I^{(k)},\bar D),(X,d_X))\leq d_H( \bar I^{(k)}, X)\leq 1/k,\\\label{E:003}
d^G_{GH}((I^{(k)},D),(\bar I^{(k)},\bar D))=0,\\\label{E:004}
d^G_{GH}((S_i^{(k)},d_{X_i}),(I^{(k)},D_i))=0.
\end{eqnarray}
Substituting these inequalities into the previous one we get for any $k$
\begin{eqnarray}\label{E:005}
d^G_{GH}((X_i,d_{X_i}),(X,d_X))\leq 2/k+d^G_{GH}((I^{(k)},D_i),(I^{(k)},D)).
\end{eqnarray}
We claim that $d^G_{GH}((I^{(k)},D_i),(I^{(k)},D))\to 0$ as $i\to \infty$; that will imply Proposition \ref{R:properness}.
This follows from the following more general lemma (compare Example 7.4.4 in \cite{burago-burago-ivanov}).
\begin{lemma}\label{L:uniform-metric}
Let $X$ be a set acted by a finite group $G$. Let $\{d_i\}$ be a sequence of $G$-invariant semi-metrics on $X$ which uniformly converges to a semi-metric $d$, i.e.
$d_i\to d$ uniformly on $X\times X$. Then $d^G_{GH}((X,d_i), (X,d))\to 0$ as $i\to \infty$.
\end{lemma}
To finish the proof of Proposition \ref{R:properness} it remains to prove this lemma. Let us define a $G$-invariant semi-metrics $\hat d_i$ on the disjoint union of two copies of $X$, $X\bigsqcup X$, such that
$\hat d_i$ extends $d$ on the first copy of $X$ and $d_i$ on the second one, and for $x$ and $x'$ from the first and the second copies of $X$ respectively set
$$\hat d_i(x,x'):=\inf_{z\in X}\{d(x,z)+\eps_i+d_i(z,x')\},$$
where $\eps_i:=\sup_{a,b\in X}\{|d(a,b)-d_i(a,b)|\}+\frac{1}{i}$.
It is easy to see that $\hat d_i$ is a $G$-invariant semi-metric, and the Hausdorff distance between the two copies of $X$ with respect to it is at most $\eps_i\to 0$. \qed

Propositions \ref{R:properness} and \ref{P:quotient-convrgent} imply immediately the main result of this section:
\begin{theorem}\label{T:equi-sequence}
Let $G$ be a finite group acting by isometries on compact metric spaces $\{X_i\}$. Let $X_i\to X$ in the GH-sense, i.e. in $\cx$, disregarding the action.
Then
\newline
(1) there exists an action of $G$ on $X$ by isometries and a subsequence of $\{X_i\}$ which converges to $X$ in the $G$-equivariant sense, i.e. in $\cx^G$;
\newline
(2) For this subsequence $G\backslash X_i\to G\backslash X$ in GH-sense.
\end{theorem}

\subsection{A simple lemma on finite-group quotients.}\label{Ss:balls-quotients}
Let a finite group $G$ act on a metric space $(Z,d_Z)$ by isometries. Let
$$\pi\colon Z\to G\backslash Z$$
be the canonical quotient map; it is a 1-Lipschitz map when the target is equipped with the quotient metric $\bar d_Z$.

Let $\bar z\in G\backslash Z$. Let $$\pi^{-1}(\bar z)=\{z_1,\dots,z_k\}.$$

Let $\bar A\subset G\backslash Z$ be a subset such that
\begin{eqnarray}\label{E:A-distance}
\alp:=\sup_{\bar a\in \bar A} dist(\bar z, \bar a)<\frac{1}{10}\min_{i\ne j}d_Z(z_i,z_j).
\end{eqnarray}

Let
$$A_p:=\pi^{-1}(\bar A)\cap \{x\in Z|\, d(z_p,x)\leq \alp\}.$$
\begin{lemma}\label{L:disjoint-isometr}
In the above notation we have
\newline
(1) $A_p\cap A_q=\emptyset $ for $p\ne q$.
\newline
(2) For any $1\leq p\leq k$ and any $g\in G$ there exists $1\leq q\leq k$ such that $g(A_p)=A_q$.
\newline
(3) $\pi^{-1}(\bar A)=\bigsqcup_{p=1}^k A_p$.
\newline
(4) $\pi(A_p)=\bar A$ for any $p=1,\dots,k$.
\newline
(5) The natural map
$$Stab(x_p)\backslash A_p\to \bar A$$
is an isometry for any $p=1,\dots,k$.
\end{lemma}
{\bf Proof.} (1) Let $z\in A_p\cap A_q$, $p\ne q$. Then $d_Z(z_p,z_q)\leq 2\alp\overset{(\ref{E:A-distance})}{<}d_Z(z_p,z_q)$; this is a contradiction.

\hfill

(2) Let $q$ be such that $g(z_p)=z_q$. Let $x\in A_p$. Then $d_Z(g(x), z_q)=d_Z(x,z_p)\leq \alp$. Hence $g(A_p)\subset A_q$.
By the symmetry the opposite inclusion also holds.

\hfill

(3) Let $x\in \pi^{-1}(\bar A)$. By the definition of $\bar d_Z$ we have
$$\min_{i}d_Z(x,z_i)=\bar d_Z(\pi(x),\bar z)\leq \alp.$$
Hence there exists $p$ such that $d_Z(x,z_p)\leq \alp$, i.e. $x\in A_p$.

\hfill

(4) Let $\bar x\in A$. Let $\pi(x)=\bar x$. By (2) there exists $q$ such that $x\in A_q$. There exists $g\in G$ such that $g(z_q)=z_p$. By part (3)
$g(A_q)=A_p$. Hence $g(x)\in A_p$ and $\pi(g(x))=\bar x$. Hence $\pi(A_p)=\bar A$.

\hfill

(5) Let $x,y\in A_p$ and $\pi(x)=\pi(y)$. Then there exists $g\in G$ such that $y=g(x)$. Due to disjointness of different $A_i$'s, one has $g(A_p)=A_p$.
Hence $g\in Stab (z_p)$. Hence the map $Stab(z_p)\backslash A_p\to \bar A$ is injective and due to part (4) it is bijective. It remains to show that this map preserves distances.

Let $x,y\in A_p$ be arbitrary points. Let $g\in G$. Assume that $g(A_p)=A_q$ with $q\neq p$. Then
\begin{eqnarray*}
d_Z(x,g(y))\geq\\
 d_Z(z_p,z_q)-d_Z(x,z_p)-d_Z(g(y),z_q)\geq
 d_Z(z_p,z_q)-d_Z(x,z_p)-d_Z(y,z_p)>\\
 >8\alp.
\end{eqnarray*}
On the other hand if $g(A_p)=A_p$, or equivalently $g(z_p)=z_p$,
$$d_Z(x,g(y))\leq d_Z(x,z_p)+d_Z(g(y),z_p)\leq 2\alp.$$
Consequently
$$\bar d_Z(\pi(x),\pi(y))=\min_{g\in G}d_Z(x,g(y))=\min_{g\in Stab(z_p)}d_Z(x,g(y)).$$
Part (5) follows. \qed

\subsection{$G$-equivariant Yamaguchi map.}\label{Ss:equivariant-Yamaguchi}
Let $G$ be a finite group. We would like to have a $G$-equivariant version of the Yamaguchi theorem \ref{T:yamaguchi} which we prove here only in a very special case needed below, namely when $G=\ZZ_2$ and the target manifold $N$ is a circle.
Let us formulate the general question.
\begin{question}\label{Q:question}
Let a sequence $\{M_i^m\}$ of closed smooth Riemannian manifolds with uniformly bounded below sectional curvature and equipped with a $G$-action by isometries, converge in $\cx^G$ to a smooth closed Riemannian manifold $N$ acted by $G$ by isometries.
Let $d_i$ be $G$-invariant metrics on $M_i\bigsqcup N$ extending the original metrics on $M_i$ and $N$ and such that $d_{i,H}(M_i,N)\to 0$.

The question is whether there exist positive numbers $\eps_i\to 0$ and smooth maps $f_i\colon M_i\to N$ such that for large $i$
\newline
(1) $f_i$ is $\eps_i$-almost Riemannian submersions;
\newline
(2) $f_i$ is $G$-equivariant;
\newline
(3) $d_i(x,f_i(x))<\eps_i$ for any $x\in M_i$;
\newline
(4) the fibers of $f_i$ are connected.
\end{question}

\begin{remark}
One can show that (4) follows from (1)-(3).
\end{remark}

We are going to show that this question has  positive answer in a very special situation sufficient for applications in this paper:
\begin{proposition}\label{P:circle-reflection-equivariant}
In the notation of Question \ref{Q:question} let us assume that $N$ is a circle (imbedded standardly to a the plane),
the group $G=\ZZ_2$ acts on $N$ by reflection with respect to a line passing through the center of the circle $N$.
Then Question \ref{Q:question} has positive answer, i.e. there exist $\ZZ_2$-equivariant maps $f_i\colon M_i\to N$ satisfying (1)-(4) there.
\end{proposition}

The proof of this proposition will be reduced to the proof of Yamaguchi's Theorem \ref{T:yamaguchi}. Let us remind the main construction used in the proof of the latter theorem.

Let $M,N$ be, as in Theorem \ref{T:yamaguchi}, smooth closed Riemannian manifolds such that $\dim M=m\geq \dim N$. Let $sec(M)\geq -1$, $|sec(N)|\leq 1$, and the injectivity radius $inj(N)>\mu$, where $\mu>0$ is a fixed number.
Let $0<\eps\ll \sigma  \ll \frac{1}{2}\min\{\mu,\pi/2\}$.
 Let $d$ be a metric on $M\bigsqcup N$
extending the original metrics on $M$ and $N$ and such that
$d_H(M,N)<\eps$.

One can show (see \cite{yamaguchi-1991} and references therein) that there exists a finite set (of indices) $\ci$ and two maps   $\frak{m}\colon \ci\to M,\, \frak{n}\colon \ci\to N$ such that
\begin{eqnarray}\label{E:jjj1}
 M\subset (Im(\frak{m}))_{5\eps}, \,\, N\subset (Im(\frak{n}))_{5\eps}\\\label{E:jjj2}
d(\frak{m}(i),\frak{m}(j))>\eps, \, d(\frak{n}(i),\frak{n}(j))>\eps \mbox{ for all } i\ne j\in \ci \\\label{E:jjj3}
d(\frak{m}(i),\frak{n}(i))<\eps.
\end{eqnarray}

In addition to these subsets let us fix a $C^\infty$-smooth function $h\colon \RR\to [0,1]$ such that
\begin{eqnarray*}
\left\{\begin{array}{ccc}
h(0)=1& &\\
h(t)=0& \mbox{if}& t\geq \sigma\\
-\frac{4}{\sigma}<h'(t)<-\frac{3}{\sigma}&\mbox{if}&\frac{3\sigma}{8}<t<\frac{5\sigma}{8}\\
-\frac{4}{\sigma}h'(t)<0&\mbox{if}&\frac{2\sigma}{8}<t<\frac{3\sigma}{8}\mbox{ or } \frac{5\sigma}{8}<t<\frac{6\sigma}{8}\\
-\sigma<h'(t)<0&\mbox{if}& 0<t<\frac{2\sigma}{8}\mbox{ or } \frac{6\sigma}{8}<t<\sigma\\
|h''(t)|<\frac{50}{\sigma}&\mbox{for all }t.
\end{array}\right.
\end{eqnarray*}
Define the map $ \Phi_N\colon N\to \RR^{\ci}$ by
$$\Phi_N(x)=[i\mapsto h(d(x,\frak{n}(i)))].$$
In was shown in \cite{katsuda} that if $\sigma$ and $\eps/\sigma$ are smaller than a constant depending on $m=\dim M$ and $\mu$ then $ \Phi_N$ satisfies the following properties:
\newline
(i) $\Phi_N$ is a smooth imbedding;
\newline
(ii) Let us denote by $\cn$ the normal bundle of $ \Phi_N(N)$ in $\RR^\ci$, and put
$$K:=\sup_{x\in N}\{\sharp\{i\in \ci|\, d(x,\frak{n}(i))\}<\sigma\}\},$$
i.e. $K$ is the maximal number of points of $Im(\frak{n})$ in balls of radius $\sigma$ in $N$. Denote by $\cn_C(N)\subset \cn$ the subset of normal vectors of Euclidean norm at most $C$. The claim is that
the restriction to $\cn_{C_1K^{1/2}}$ of the normal exponential map is a diffeomorphism onto its image which we denote by $\cu$, where $C_1>0$ is a number depending on $m,\mu,\sigma$. We denote by
$$P\colon \cu\to \Phi_N(N)$$
be the projection along the fibers of the normal bundle.

Let us define a $C^1$-smooth map $\Phi_M\colon M\to \RR^\ci$. For each $i\in\ci$ and each $x\in M$ put
$$\phi_i(x)=\frac{1}{vol(B_M(\frak{m}(i),\eps))}\int_{y\in B_M(\frak{m}(i),\eps)}d(y,x)dy.$$
Then $\phi_i$ is $C^1$-smooth. Define $\Phi(x):=[i\mapsto h(\phi_i(x))]$.

It is not hard to check that $\Phi_M(M)$ is contained in $\cu$. Finally define $f\colon M\to N$ by
$$f=\Phi_N^{-1}\circ P\circ \Phi_M.$$
The claim is that this $f$ satisfies all the conclusions of the Yamaguchi Theorem \ref{T:yamaguchi} except $f$ is only $C^1$-smooth rather than $C^\infty$.
It can be smoothened further to get a $C^\infty$-map as it is shown below in Lemma \ref{L-approxim-equivar}.

\hfill

Let us now discuss how the construction of the map $f$ could be generalized into $G$-equivariant situation where $G$ is a finite group. The following Lemma is obvious.
\begin{lemma}\label{L:eqivariant-yamag-existence}
 Let us assume that there exist a finite set $\ci$ with a (left) action of $G$ and maps
$\frak{m}\colon \ci\to M,\, \frak{n}\colon \ci\to N$ as above which are in addition $G$-equivariant. Then the maps $\Phi_N,P,\Phi_M$ and hence $f=\Phi_N^{-1}\circ P\circ \Phi_M$ are $G$-equivariant.
\end{lemma}

{\bf Proof of Proposition \ref{P:circle-reflection-equivariant}.} Let $N=C$ be a circle of length $L$. We consider it as the standard circle of length $L$ on the plane $(x,y)$ centered at the origin.
Let the group $G=\ZZ_2$ act by reflection
with respect to the axis $x$. The injectivity radius of $C$ is $L/2$. By our choice of $\eps\ll inj(C)$ there exists a natural number $N$ in the interval
$$N\in (0.4\frac{L}{\eps},0.45\frac{L}{\eps}).$$

Let us define $\ci$ to be the set of all points on the circle $C$ of the form
$$\ci:=\{\frac{L}{2\pi} \left(\cos(\frac{\pi k}{N}),\sin (\frac{\pi k}{N})\right)|\, 0\leq k\leq 2N-1, \, k\ne 0, N\}.$$
Observe that we have omitted the two points on the axis $x$. $\ci$ is invariant under the reflection with respect to the axis $x$; that defines an action of $\ZZ_2$ on $\ci$.
The identical imbedding $\frak{n}\colon \ci\to N=C$ is the desired one, it is obviously $\ZZ_2$-equivariant.

It follows that $\frak{n}(\ci)=\ci$ is $1.1\eps$-separated and $3\eps$-net, i.e. $(\ci)_{3\eps}=C$.

Let us assume that $d_H(M,C)<0.01 \eps$. We are going to define $\ZZ_2$-equivariant map $\frak{m}\colon \ci \to M$  satisfying the properties (\ref{E:jjj1})-(\ref{E:jjj3}).
Let us denote by $\ci^+$  the subset of $\ci$ contained in the upper half plane. Then $\ci=\ci^+\bigsqcup a(\ci^+)$ where $a\in \ZZ_2$ is the non-zero element corresponding to the reflection. For any $i\in \ci^+$ there exists
$x_i\in M$ such that $d(x_i,\frak{n}(i))<0.01 \eps$. Let us define the map $\frak{m}\colon \ci \to M$ as follows: for any $i\in \ci^+$ set
\begin{eqnarray*}
\frak{m}(i):=x_i,\\
\frak{m}(a(i)):=a(x_i),
\end{eqnarray*}
where $a(i)$ means the reflection of $i\in C$ with respect to the $x$-axis, and $a(x_i)$ means the given action of $a\in \ZZ_2$ on $M$.

Clearly $\frak{m}$ is $\ZZ_2$-equivariant. It is easy to see that $d(\frak{m}(i),\frak{m}(j))>\eps $ for $i\ne j\in \ci $ and $(Im(\frak{m}))_{4\eps}=M$.
Then the assumptions of Lemma \ref{L:eqivariant-yamag-existence} are satisfied, and we obtain a $C^1$-smooth $\ZZ_2$-equivariant $\tau(\eps)$-almost Riemannian submersion
$f\colon M\to C$ such that $d(x,f(x))<\tau(\eps)$. In order to make $f$ to be $C^\infty$-smooth we will use the following well known lemma which is also proven for the sake of completeness.

\begin{lemma}\label{L-approxim-equivar}
Let $G$ be a compact Lie group (in particular a finite group) acting smoothly on smooth closed Riemannian manifolds $M,N$. Let $f\colon M\to N$ be a $G$-equivariant $C^1$-smooth map.
The for any $\delta>0$ there exists a $C^\infty$-smooth $G$-equivariant map $f_\delta\colon M\to N$ such that
$$||f-f_{\delta}||_{C^1}<\delta.$$
\end{lemma}
{\bf Proof.} By \cite{palais} there exists a smooth $G$-equivariant imbedding of $N$ into a finite dimensional Euclidean space equipped with a linear orthogonal action of the group $G$:
$$\iota \colon N\inj V.$$
Let $\cn_N$ be the normal bundle of $\iota(N)$. There exists an open neighborhood $\co\subset \cn_N$ of the zero section such that
$$\exp\colon \co\to V$$
has an open image $\co'$ and is a diffeomorphism onto it. The composition $q\colon \co'\to N$ of $\exp^{-1}$ with the natural projection $\cn_N\to N$ is the nearest point map from $\co'$ to $N$.
Clearly $q$ is  $C^\infty$-smooth and $G$-equivariant. By reducing $\co$ we may assume that there exists a constant $C$ such that
\begin{eqnarray}\label{R:lip-q}
||q||_{C^1}<C.
\end{eqnarray}

Fix $\eps>0$ such that the $\eps$-neighborhood of $N$ in contained in $\co'$. There exists a $C^\infty$-smooth map $h_1\colon M\to V$ such that
$$||\iota\circ f-h_1||_{C^1}<\eps.$$
By averaging it over the group $G$ define $h_2\colon M\to V$ as
$$h_2(x)=\int_G g(h_1(g^{-1}(x)))dg,$$
where $dg$ in the probability Haar measure on $G$. It is easy to see that $h_2$ is $C^\infty$-smooth, $G$-equivariant and
\begin{eqnarray}\label{E:pppp}
||h_2-\iota\circ f||_{C^1}<\eps.
\end{eqnarray}

Define $h:=q\circ h_2\colon M\to N$. Clearly $h$ is $C^\infty$-smooth and $G$-equivariant. Let us show that $h$ is $C^1$-close to $f$. First we have (below $||\cdot ||$ denotes the Euclidean norm)
\begin{eqnarray*}
||(\iota\circ h)(x)-(\iota\circ f)(x)||=||\iota(q(h_2(x)))-\iota(f(x))||\leq\\
||h_2(x)-(\iota(f))(x)||+||\iota(q(h_2(x)))-h_2(x)||\leq\\ 2||h_2(x)-\iota(f(x))||<2\eps,
\end{eqnarray*}
where the second inequality is due to the fact that $q$ is the nearest point map to $N$.

For the first derivatives we have
\begin{eqnarray*}
||\pt_i[h(x)-f(x)]||=||\pt_i\left[q(h_2(x))-q((\iota\circ f)(x))\right]||=||\sum_j \pt_j q (\pt_i (h_2)_j(x)-\pt_i(\iota\circ f)_j(x))||\leq \\
C'\sum_j|\pt_i (h_2)_j(x)-\pt_i(\iota\circ f)_j(x)|\overset{(\ref{E:pppp})}{\leq } C''\eps.
\end{eqnarray*}
The result follows.
\qed

\section{Convergence of closed surfaces.}\label{S:collapse-surfaces}

\subsection{Reduction between sequences of metrics.}\label{Ss:reduction-smooth}
The goal of this subsection is to prove a technical lemma, which allows to reduce the study of the nearby cycle construction from a sequence of metrics to a sequence of metrics with more feasible behavior.

Let us introduce notation. Let $\{X_i\}$ be a sequence of sets equipped with two sequences of metrics $\{d_{X_i}\}$ and $\{d'_{X_i}\}$ each of them making $X_i$ a compact space. Let us assume that
\begin{eqnarray}\label{E:metric-uniformly}
\sup_{x,y\in X_i} |d_{X_i}(x,y)-d'_{X_i}(x,y)|\to 0.
\end{eqnarray}
Further assume that $\{(X_i,d_{X_i})\}$ GH-converges to a compact metric space $X$. This and the condition (\ref{E:metric-uniformly}) imply that $\{(X_i,d'_{X_i})\}$ GH-converges to $X$.
Let $d_i$ (resp. $d_i'$) be a semi-metric on $X_i\coprod X$ extending the original metric on $X$ and the metric $d_{X_i}$ (resp. $d'_{X_i}$) on $X_i$.
Assume that
\begin{eqnarray}\label{E:uu1}
d_{i,H}(X_i,X)\to 0 \mbox{ and } d'_{i,H}(X_i,X)\to 0.
\end{eqnarray}
Fix a subset $Z\subset X$, and $x\in Z$.

Denote
 \begin{eqnarray}\label{star-star}
 \cb_{i,x}(\eps):=\{y\in X_i|\, \, d_i(x,y)<\eps\},\\
 \cb'_{i,x}(\eps):=\{y\in X_i|\, \, d'_i(x,y)<\eps\}.
 \end{eqnarray}

 In the next lemma the cohomology is taken with coefficients in a field.
\begin{lemma}\label{L:reduction-smooth}
Assume the properties (\ref{E:uu1}).
Let $\kappa>0$. Let $a$ be a non-negative integer. Assume that for any $0<\eps_1<\eps_2<\kappa$ and any $x\in Z$ there exists $i_0=i_0(\eps_1, \eps_2,x)$ such that for any $i>i_0$
the dimension
$$\dim Im[ H^a(\cb'_{i,x}(\eps_2))\to H^a(\cb'_{i,x}(\eps_1))]$$
is independent of $i,\eps_1,\eps_2$ (but may depend on $x\in Z$).

Then any $0<\eps_1<\eps_2<\kappa$ and any $x\in Z$ there exists $i_1=i_1(\eps_1, \eps_2,x)\geq i_0$ such that for all $i>i_1$ we have
$$\dim Im[ H^a(\cb_{i,x}(\eps_2))\to H^a(\cb_{i,x}(\eps_1))]=\dim Im[ H^a(\cb'_{i,x}(\eps_2))\to H^a(\cb'_{i,x}(\eps_1))].$$
In particular, the former dimension is independent of $i,\eps_1,\eps_2$.
\end{lemma}
{\bf Proof.}  Clearly (\ref{E:metric-uniformly}) implies that there exists $i_0 \geq i'_0$ such that for all $i>i_0$, all $x\in Z$, and all $0<\eps_1<\delta_1<\delta_2<\eps_2<\frac{\kappa}{2}$ one has
$$\cb'_{i,x}(\eps_1/2)\subset \cb_{i,x}(\eps_1)\subset \cb'_{i,x}(\delta_1)\subset \cb'_{i,x}(\delta_2)\subset \cb_{i,x}(\eps_2)\subset \cb'_{i,x}(2\eps_2). $$
Then in cohomology we have the linear maps
$$H^a(\cb'_{i,x}(2\eps_2))\to H^a(\cb_{i,x}(\eps_2))\to H^a(\cb'_{i,x}(\delta_2))\to H^a(\cb'_{i,x}(\delta_1))\to H^a(\cb_{i,x}(\eps_1))\to H^a(\cb'_{i,x}(\frac{\eps_1}{2})).$$
It is easy to see that for a sequence of linear maps
$$A\to X\to Y\to B$$
one has $\dim Im[X\to Y]\geq \dim Im[A\to B]$. Hence
\begin{eqnarray*}
\dim Im[ H^a(\cb'_{i,x}(2\eps_2))\to H^a(\cb'_{i,x}(\frac{\eps_1}{2}))]\leq \dim Im[H^a(\cb_{i,x}(\eps_2))\to H^a(\cb_{i,x}(\eps_1))]\leq\\ \dim Im[H^a(\cb'_{i,x}(\delta_2))\to H^a(\cb'_{i,x}(\delta_1))].
\end{eqnarray*}
But the first and the third dimensions are equal to each other by the initial assumption. Hence the result follows. \qed

\subsection{Which surfaces may collapse.}\label{Ss:surfaces may collapse}
The  main result of this subsection is Proposition \ref{P:surfaces-may-collapse} which seems to be well known.
For the lack of reference and for the sake of completeness we present a proof.
 We will need a few preparations which will also play a role later.
The following result is a combination of Lemmas 1.13 and 1.14 in \cite{slutskiy-compact-domains}.
\begin{lemma}[Slutskiy, \cite{slutskiy-compact-domains}]\label{slutskiy-lemma}
Let $S$ be a closed 2-dimensional surface with a metric with curvature at least -1 in the sense of Alexandrov.
Then there exists a sequence of $C^\infty$-smooth metrics on $S$ with the Gaussian curvature ar least -1 which converges uniformly to the given metric.
\end{lemma}

\begin{lemma}\label{L:smooth-not-double}
Given a $C^\infty$-smooth metric on the sphere $S^2$ with Gaussian curvature at least -1. Then its isometric imbedding into $\HH^3$ cannot be a double cover of a convex set.
\end{lemma}
{\bf Proof.} Assume the opposite. Then the Gaussian curvature $K=-1$ everywhere: indeed, this is clearly the case outside of the boundary, and since the boundary has measure 0 the same is true
on the boundary by the $C^\infty$-smoothness. By the Gauss-Bonnet formula we have
$$4\pi=2\pi\chi(S^2)=\int_{S^2}KdA=-area(S^2)<0.$$
This is a contradiction. \qed

\hfill
\begin{proposition}\label{P:surfaces-may-collapse}
Let $\{X_i\}$ be a sequence of closed 2-surfaces of a given homeomorphism type with CBB($-1$) metrics which GH-converges to a compact metric space of dimension less than 2.
Then $X_i$ are homeomorphic to one of the following surfaces: sphere, real projective plane, torus, Klein bottle.
\end{proposition}
{\bf Proof.} By Lemma \ref{slutskiy-lemma} we may assume that the metrics on $X_i$ are $C^\infty$-smooth. By the Gauss-Bonnet formula
$$\chi(X_i)=\frac{1}{2\pi}\int_{X_i}KdA\geq -area(X_i).$$
Since $\{X_i\}$ collapses, $area(X_i)\to 0$ by \cite{burago-gromov-perelman}, Theorem 10.8. Hence $\chi(X_i)\geq 0$. In the orientable case that means that $X_i$ are homeomorphic either to
the 2-sphere or 2-torus. In the non-orientable case $X_i$ is homeomorphic to a connected sum of $k$ copies of $\RR\PP^2$. The Euler characteristic of it is $2-k\geq 0$.
Thus either $k=1$, i.e. $X_i$ is  $\RR\PP^2$, or $k=2$, i.e. $X_i$ is the Klein bottle. \qed

\subsection{Spheres collapsing to a segment.}\label{Ss:spheres-segment}
The main result of this subsection is the following theorem. Below the cohomology is taken with coefficients in a field.
\begin{theorem}\label{T:spheres-to-segment}
Let $\{X_i\}$ be a sequence of 2-spheres with metrics having the curvature at least $-1$ in the sense of Alexandrov.
Assume it GH-converges to a (non-degenerate) segment $I$. Let $\{d_i\}$ be a sequence of metrics on $I\bigsqcup X_i$ extending the original metrics on $I$ and $X_i$ and such that
$$d_{i,H}(I,X_i)\to 0\mbox{ as } i\to \infty.$$
\newline
(1) Let $\kappa>0$. Let $0<\eps_1<\eps_2<\frac{\kappa}{100} $. Then there exists $i_0$ such that for any $i>i_0$ and any $x\in I$ wit $dist(x,\pt I)>\kappa$ one has
(below $\cb_{i,x}(\eps)$ is defined (\ref{star-star}))
\begin{eqnarray*}
 \dim Im\left[H^a(\cb_{i,x}(\eps_2))\to H^a(\cb_{i,x}(\eps_1))\right] \simeq\left\{\begin{array}{ccc}
                                                  1&\mbox{if}& a=0,1\\
                                                  0&\mbox{otherwise}&\\
                                                      \end{array}\right.
\end{eqnarray*}
(2) Let $x\in \pt I$. Let $0<\eps_1<\eps_2<\frac{1}{100}length(I)$. Then for large $i$ one has
\begin{eqnarray*}
\dim Im\left[H^a(\cb_{i,x}(\eps_2))\to H^a(\cb_{i,x}(\eps_1))\right]\simeq\left\{\begin{array}{ccc}
                                                  1&\mbox{ if }& a=0\\
                                                  0&\mbox{ otherwise }&\\
                                                      \end{array}\right.
\end{eqnarray*}
\end{theorem}

{\bf Proof of Theorem \ref{T:spheres-to-segment}.} Assume first that the metrics $d_{X_i}$ are $C^\infty$-smooth. By the Alexandrov imbedding theorem \ref{T:alexandrov-isometric} $(X_i,d_{X_i})$ admits an isometric realization in $\HH^3$. By Lemma \ref{L:smooth-not-double}
the isometric realizations are boundaries of 3-dimensional convex bodies.
We may assume that $\{K_i\}$ are contained in a bounded region. By the Blaschke selection theorem \ref{blas} we may choose a subsequence denoted as the original sequence converging to a convex compact set $K$ in the Hausdorff metric: $K_i\to K$.
Since $X_i\overset{GH}{\to} I$, Proposition \ref{P:convex-hypersurf-convegence} implies that $K$ is a segment isometric to $I$. Let us consider semi-metrics $d'_i$ on $I \bigsqcup X_i$ induced by the restriction of the metric $dist$ of $\HH^3$ to $\partial K_i$ and $K$.
Define
$$\cb'_{i,x}(\eps):=\{y\in X_i|\, \, d'_i(x,y)<\eps\}.$$
Then Propositions
\ref{P:isomorphisms-sphere} and \ref{P:isomorph-sphere-boundary} imply that the dimensions
$$\dim Im[ H^a(\cb'_{i,x}(\eps_2))\to H^a(\cb'_{i,x}(\eps_1))]$$
stabilize to the claimed ones under two conditions: (1) for $0<\eps_1<\eps_2<\frac{\kappa}{2}$ and $d'_i(x, \pt I)>\kappa$; (2) for $0<\eps_1<\eps_2<\frac{length(I)}{2}$ and $x\in \pt I$. On the other hand, by Proposition~\ref{P:convex-hypersurf-convegence}(4) we have
$$\sup_{x,y\in X_i} |d_{X_i}(x,y)-d'_{X_i}(x,y)|\to 0.$$
Hence, we can apply Lemma~\ref{L:reduction-smooth} and get the desired result.

Let us consider the general case. By Slutskiy's lemma \ref{slutskiy-lemma} there exists a $C^\infty$-smooth metric $d'_{X_i}$ on $X_i$ with the Gaussian curvature at least -1 such that
$$\sup_{x,y\in X_i}|d_{X_i}(x,y)-d'_{X_i}(x,y)|<1/i.$$
Now the result again follows from Lemma \ref{L:reduction-smooth}. \qed

\hfill

We will also need the following result.

\begin{proposition}\label{P:involution-spheres}
Let $\{\tilde X_i\}$ be a sequence of 2-spheres with CBB($-1$) metrics invariant with respect to the antipodal involution (i.e. $\ZZ_2$-invariant).
Assume it converges to a non-degenerate segment $\tilde I$ in $\cx^{\ZZ_2}$ when
$\ZZ_2$ acts on $\tilde I$ by reflection with respect to the middle point. Let $\{\tilde d_i\}$ be a sequence of metrics on $\tilde I\bigsqcup \tilde X_i$ extending the original metrics on $\tilde I$ and $\tilde X_i$ and such that
$$\tilde d_{i,H}(\tilde I,\tilde X_i)\to 0\mbox{ as } i\to \infty.$$ Let $0<\delta_1<\delta_2<\frac{1}{100}length(\tilde I)$. Let $\tilde x_0\in\tilde I$ be the middle point. Set
$$\tilde\cb_{i,\tilde x_0}(\delta)=\{y\in \tilde X_i|\tilde d_i (\tilde x_0,y)<\delta\}.$$ Then for large $i$ one has
\begin{eqnarray*}
 \dim Im\left[H^a(\ZZ_2\backslash\tilde\cb_{i,\tilde x_0}(\delta_2))\to H^a(\ZZ_2\backslash\tilde\cb_{i,\tilde x_0}(\delta_1))\right]\simeq \left\{\begin{array}{ccc}
                                                  1&\mbox{if}& a=0,1\\
                                                  0&\mbox{otherwise}&\\
                                                      \end{array},\right.
\end{eqnarray*}
where the cohomology is taken with coefficients in an arbitrary field.
\end{proposition}
{\bf Proof.} Let us assume first that the metrics on $\tilde X_i$ are $C^\infty$-smooth. By Theorem \ref{realiz} and Lemma \ref{L:smooth-not-double} $\tilde X_i$ can be realized isometrically as the boundary of a 3-dimensional convex set $K_i\subset \HH^3$
such that the antipodal involution is induced by a reflection with respect to a point $o\in \HH^3$ which might be assumed independent of $i$. After a choice of a subsequence one may assume
that $\{ K_i\}$ converges in the Hausdorff metric to a convex compact set $J$ which is necessarily invariant with respect to the reflection $s$. By Proposition \ref{P:convex-hypersurf-convegence} one has $\dim J=1$, i.e. $J$ is a segment. By Lemma \ref{L:hausdorff-boundary-small-dim}
$\{\pt K_i\}$ converges to $J$ in the Hausdorff metric. Hence $(\pt K_i,dist)\overset{d_{GH}^{\ZZ_2}}\to (J,dist)$. Then Proposition \ref{P:convex-hypersurf-convegence}(4) implies that
$(\pt K_i,D_i^{in})\overset{d_{GH}^{\ZZ_2}}\to (J,dist)$. Hence $\tilde I=J$ in $\cx^{\ZZ_2}$.

Let $q$ be the nearest point map from $\HH^3$ to the line passing through $J$. Denote
$$\ca_{i,\tilde x_0}(\delta)=\{z\in \pt K_i|dist (q(z),\tilde x_0)<\delta\},$$
$$\cb'_{i,\tilde x_0}(\delta)=\{z\in \pt K_i|dist (z,\tilde x_0)<\delta\}.$$
For $0<\delta<\frac{length(I)}{100}$ Lemma \ref{Cor:hyperb-fibration} implies that $\ca_{i,\tilde x_0}(\delta)$ is $\ZZ_2$-equivariantly homeomorphic to $(-\delta,\delta)\times S^1$ when the involution acts
on the latter space as
$$(t,x)\mapsto (-t,-x).$$
Then the quotient space $\ZZ_2\backslash \ca_{i,\tilde x_0}(\delta)$ is homeomorphic to the M\"obius band. Moreover for $0<\delta<\delta'$ the inclusion
$$\ZZ_2\backslash \ca_{i,\tilde x_0}(\delta)\subset \ZZ_2\backslash \ca_{i,\tilde x_0}(\delta')$$
is a homotopy equivalence.

By Lemma \ref{L:lemma-9} for large $i$ one has inclusions
$$\ca_{i,\tilde x_0}(\frac{\delta_1}{2})\subset \cb'_{i,\tilde x_0}(\delta_1)\subset \ca_{i,\tilde x_0}(\frac{\delta_1+\delta_2}{2})\subset \cb'_{i,\tilde x_0}(\delta_2)\subset \ca_{i,\tilde x_0}(2\delta_2).$$

Then we have the induced maps in cohomology
\begin{eqnarray*}
H^*(\ZZ_2\backslash \ca_{i,\tilde x_0}(\frac{\delta_1}{2}))\overset{a_1}{\leftarrow} H^*(\ZZ_2\backslash \cb'_{i,\tilde x_0}(\delta_1))\overset{a_2}{\leftarrow} H^*(\ZZ_2\backslash\ca_{i,\tilde x_0}(\frac{\delta_1+\delta_2}{2}))\overset{a_3}{\leftarrow} \\
\overset{a_3}{\leftarrow} H^*(\ZZ_2\backslash \cb'_{i,\tilde x_0}(\delta_2))\overset{a_4}{\leftarrow} H^*(\ZZ_2\backslash\ca_{i,\tilde x_0}(2\delta_2)).
\end{eqnarray*}
The maps $a_1\circ a_2$ and $a_3\circ a_4$ are isomorphisms. Since the M\"obius map is homotopy equivalent to the circle, the dimensions
$$\dim Im[ H^a(\ZZ_2\backslash\cb'_{i,x}(\eps_2))\to H^a(\ZZ_2\backslash\cb'_{i,x}(\eps_1))]$$
stabilize to the claimed ones. Then the combination of Proposition~\ref{P:convex-hypersurf-convegence}(4) and Lemma~\ref{L:reduction-smooth} imply the result.

\hfill

Let us consider the case of a general CBB($-1$) metric on $\tilde X_i$. For any $i$ the convex body $K_i$ can be approximated in the Hausdorff metric arbitrarily close by $C^\infty$-smooth $o$-symmetric convex body $K_i'$.
Claim \ref{Cl:uniform-radial} implies that
for the given metric $d_{\tilde X_i}$ on $\tilde X_i$ there exists a $C^\infty$-smooth CBB($-1$) metric $d_{\tilde X_i}'$ which is invariant under the antipodal involution and
$$\sup_{x,y\in \tilde X_i} |d_{\tilde X_i}(x,y)-d_{\tilde X_i}'(x,y)|<1/i.$$

Let us denote by $\bar d_{\tilde X_i}$ and $\bar d_{\tilde X_i}'$ the corresponding quotient metrics on $\ZZ_2\backslash \tilde X_i$. Then clearly
$$\sup_{x,y\in \tilde X_i} |\bar d_{\tilde X_i}(x,y)-\bar d_{\tilde X_i}'(x,y)|<1/i.$$

Lemma \ref{L:reduction-smooth} implies the proposition. \qed

\subsection{Collapse to the circle.}\label{Ss:collapse-circle}
\begin{theorem}\label{T:surface-t-circle}
Let $\{X_i\}$ be sequence of closed 2-surfaces with CBB($-1$) metrics.
Assume that it GH-converges to a circle $C$. Let $\{d_i\}$ be a sequence of metrics on $C\bigsqcup X_i$ extending the original metrics on $C$ and $X_i$ and such that
$$d_{i,H}(C,X_i)\to 0\mbox{ as } i\to \infty.$$
Let $0<\eps_1<\eps_2<\frac{1}{100} length(C)$.
\newline
(1) Then there exists $i_0$ such that for any $i>i_0$ one has for any $x\in C$
\begin{eqnarray*}
 \dim Im\left[H^a(\cb_{i,x}(\eps_2))\to H^a(\cb_{i,x}(\eps_1))\right]=\left\{\begin{array}{ccc}
                                                  1&\mbox{if}& a=0,1\\
                                                  0&\mbox{otherwise}&\\
                                                      \end{array},\right.
\end{eqnarray*}
where the cohomology is taken with coefficients in an arbitrary field, and $\cb_{i,x}(\eps)$ is defined by (\ref{star-star}).

(2) Let $X_i$ have a fixed homeomorphism type. Then $\chi(X_i)=0$. In particular sphere and real projective planes cannot collapse to circle.
\end{theorem}
{\bf Proof.} By Slutskiy's lemma \ref{slutskiy-lemma} and Lemma \ref{L:reduction-smooth} we may assume that the metrics $d_i$ on $X_i$ are $C^\infty$-smooth.
In the latter case the result is a special case of Corollary \ref{Cor:fiber-Yamag-cor}. Indeed the fibers of the Yamaguchi map $X_i\to C$ are 1-dimensional and connected, hence circles.
By multiplicativity of the Euler characteristic for fibrations we get $\chi(X_i)=0$.
\qed

\subsection{Collapse of $\RR\PP^2$.}\label{Ss:rp2-1dim}
In the next theorem the cohomology is taken with coefficients in an arbitrary field.
\begin{theorem}\label{T:rp2-segment}
Let $\{X_i\}$ be a sequence of $\RR\PP^2$ with CBB($-1$) metrics. Let us assume that is collapses to a 1-dimensional space $I$.
\newline
(1) Then $I$ is a segment.
\newline
(2) Furthermore let $d_i$ be metrics on $I\bigsqcup X_i$ extending the original metrics on $I$ and $X_i$ and such that
$$d_{i,H}(I,X_i)\to 0 \mbox{ as } i\to 0.$$
\newline
(a)  Let $0<\kappa<\frac{length(I)}{100}$.  Let $0<\eps_1<\eps_2<\kappa$. Then there exists $i_0$ such that for $i>i_0$ and for all $x\in I$ with $dist(x,\pt I)>\kappa$ one has
\begin{eqnarray*}
\dim Im\left[H^a(\cb_{i,x}(\eps_2))\to H^a(\cb_{i,x}(\eps_1))\right]=\left\{\begin{array}{ccc}
                                                  1&\mbox{if}& a=0,1\\
                                                  0&\mbox{otherwise}&\\
                                                      \end{array},\right.
\end{eqnarray*}
where $\cb_{i,x}(\eps)$ is defined by (\ref{star-star}).
\newline
(b)  Let $x\in \pt I$. Then for large $i$  one has
\begin{eqnarray*}
\dim Im\left[H^a(\cb_{i,x}(\eps_2))\to H^a(\cb_{i,x}(\eps_1))\right]=\left\{\begin{array}{ccc}
                                                  1&\mbox{if}& a=0\\
                                                  0&a\ne 0,1&\\
                                                      \end{array}\right.
\end{eqnarray*}
while for $a=1$ for one of the end points of $I$ one has for large $i$ after a choice of a subsequence
$$\dim Im\left[H^1(\cb_{i,x}(\eps_2))\to H^1(\cb_{i,x}(\eps_1))\right]=0,$$
and for the other end point one has
$$\dim Im\left[H^1(\cb_{i,x}(\eps_2))\to H^1(\cb_{i,x}(\eps_1))\right]=1.$$
\end{theorem}

First we will need two lemmas.
\begin{lemma}\label{L:ivanov}
Let $M$ be a closed Riemannian manifold. Let $\tilde M$ be $m$-fold cover with the lifted metric. Then
$$diam(\tilde M)\leq m\cdot diam(M).$$
\end{lemma}
This was proved by Sergei Ivanov \cite{ivanov-mo} in his answer on MO to Petrunin's question:
Petrunin himself claimed in that post a weaker estimate $diam(\tilde M)\leq 2(m-1) diam(M)$ although he does not present an argument.
This would suffice for our purposes.

The next lemma was proved by Petrunin~\cite{petrunin-mo-antipodal} on MO; we reproduce below his argument.
\begin{lemma}\label{L:Petrunin}
Let $g$ be a smooth Riemannian metric on 2-sphere $\SS^2$ invariant under the antipodal involution $x\mapsto -x$.
Then there exists point $a$ such that
$$dist(a,-a)>\frac{diam (\SS^2,g)}{100}.$$
\end{lemma}
{\bf Proof.} Let $h$ be the induced Riemannian  metric on $\RR\PP^2=\ZZ_2\backslash \SS^2$. Let $D:=diam (\SS^2,g)$. By Lemma \ref{L:ivanov}
\begin{eqnarray}\label{E:ineq-diameter}
diam(\RR\PP^2,h)\geq \frac{D}{2}.
\end{eqnarray}
Let $\gamma$ be the shortest non-contractible geodesic on $(\RR\PP^2,h)$.
Denote its length (systole) by $l$. Let $\tilde \gamma$ be the preimage of $\gamma$ in $\SS^2$. $\tilde \gamma$ is a connected closed geodesic of length $2l$; it is invariant under the antipodal involution.
By the Jordan theorem the complement of $\tilde \gamma$ is a union of two disjoint disks whose closures we denote by $\Delta $ and $\Delta'$. Let us show that
\begin{eqnarray}\label{E:delta-diam}
diam(\Delta)\geq \frac{D}{2}.
\end{eqnarray}
First observe that the antipodal involution is a homeomorphism $\Delta \tilde\to \Delta'$. Indeed
otherwise the antipodal involution would be a homeomorphism $\Delta\to \Delta$ and had no fixed points; this is impossible by the Brouwer fixed point theorem.

By (\ref{E:ineq-diameter}) there exist two points $x,y\in \RR\PP^2$ such that the $h$-distance $dist_h(x,y)\geq \frac{D}{2}$. Let $\tilde x,\tilde y\in \Delta$ be their respective lifts, and
let $\tilde y'\in \Delta'$ be the antipodal image of $\tilde y$. Then we have
\begin{eqnarray*}
diam(\Delta)\geq dist(\tilde x,\tilde y)\geq \min\{dist(\tilde x,\tilde y),dist(\tilde x,\tilde y')\}=dist(x,y)\geq \frac{D}{2}.
\end{eqnarray*}
Thus (\ref{E:delta-diam}) is proved.

For any $\tilde z\in \tilde\gamma$
$$dist(\tilde z,-\tilde z)=l/2$$
since $l$ is a systole.
Hence in the case $l\geq \frac{D}{20}$ lemma follows.

Let us assume that $l< \frac{D}{20}$. This and the inequality (\ref{E:delta-diam}) imply that there is a point $\tilde w\in \Delta$ such that
$$dist(\tilde w,\pt \Delta)\geq D/5,$$
where $\pt \Delta =\tilde\gamma$. Since $-\tilde w\in \Delta'$ it follows that
$$dist(\tilde w,-\tilde w)\geq D/5.$$
\qed

\hfill


{\bf Proof of Theorem \ref{T:rp2-segment}.} By Slutskiy's lemma \ref{slutskiy-lemma} and Lemma \ref{L:reduction-smooth} we may assume that the metrics $d_i$ on $X_i$ are $C^\infty$-smooth.
 Let $\tilde X_i$ denote the universal 2-sheeted cover of $X_i$ with the lifted metrics. $\tilde X_i$ can be considered as a sphere $(\SS^2,g_i)$ where $g_i$ is a smooth Riemannian metric invariant under the antipodal involution.

By Lemma \ref{L:ivanov} $diam(\tilde X_i)\leq 2diam(X_i)$ although  a weaker estimate would be sufficient.
The group $\ZZ_2$ of deck transformations acts by isometries on each $\tilde X_i$.  Since the Gaussian curvature of $\tilde X_i$ is uniformly bounded below,
the Gromov compactness theorem and Proposition \ref{R:properness} imply that one may choose a subsequence such that $\tilde X_i\to \tilde X$ in $\cx^{\ZZ_2}$.
By Proposition \ref{P:quotient-convrgent}
$$X_i\to \ZZ_2\backslash \tilde X \mbox{ in } \cx.$$

Since in our situation $\dim \ZZ_2\backslash \tilde X=1$ it follows that $\dim\tilde X=1$. Since by Theorem \ref{T:surface-t-circle}(2) spheres cannot collapse to a circle it follows that $\tilde X$ is a segment.

Any isometry of a segment is either identity or the reflection with respect to the origin. Let us show that the action of $\ZZ_2$ on $\tilde X$ is the latter one.
There exist $\ZZ_2$-invariant metrics $d_i$ on $\tilde X\bigsqcup \tilde X_i$ extending the original metrics on $\tilde X$ and $\tilde X_i$ and such that
$$d_{i,H}(\tilde X,\tilde X_i)\to 0.$$
Let us denote by $s\in \ZZ_2$ the non-zero element; it acts as the antipodal involution on $\tilde X_i$. By Lemma \ref{L:Petrunin} there exists $\tilde x_i\in\tilde X_i$
such that for large $i$
$$d_i(\tilde x_i,s(\tilde x_i))>\frac{diam(\tilde X_i)}{100}>\frac{diam(\tilde X)}{200}.$$
One may choose a subsequence such that
$$d_i(\tilde x_i,\tilde x)\to 0 \mbox{ for some }\tilde x\in \tilde X.$$
Since $s$ preserves $d_i$ it follows that
$$d_{\tilde X}(\tilde x,s(\tilde x))\geq \frac{diam(\tilde X)}{200}>0.$$
Hence the action of $s$ on $\tilde X$ is non-trivial, so it is the reflection with respect to the middle point.

\hfill

Let $\pi\colon \tilde X\to \ZZ_2\backslash \tilde X=I$ be the canonical map. Let $\tilde x\in \tilde X$, $ x:=\pi(\tilde x)\in I$.
Let $B_i(x,\eps)\subset I\bigsqcup  X_i$ be the $\eps$-ball with respect to the metric $\bar d_i$ which is the $\ZZ_2$-quotient metric of  $d_i$.
Let $\eps>0$ be less than $1/10$ times the diameter of the $\ZZ_2$-orbit of $x$ in case it consists of two distinct elements, and with no other restriction otherwise.
 Lemma \ref{L:disjoint-isometr} implies that $\pi^{-1}(B_i( x,\eps))$ is a disjoint union of $\eps$-balls with respect to the metric $d_i$ centered at all different points of the $\ZZ_2$-orbit of $x$, and
the obvious map $$Stab(\tilde x)\backslash \tilde B_i(\tilde x,\eps)\to B_i( x,\eps)$$
is an isometry, where $\tilde B_i(\tilde x,\eps)$ is the $\eps$-ball in $\tilde X\bigsqcup \tilde X_i$ centered at $\tilde x$. It follows that the obvious map
\begin{eqnarray}\label{E:isometry-ball-1}
Stab(\tilde x)\backslash (\tilde B_i(\tilde x,\eps)\cap \tilde X_i)\to B_i( x,\eps)\cap X_i
\end{eqnarray}
is an isometry.


\underline{Case 1.} Assume that $\tilde x$ is not the middle point of the segment $X$. Then $Stab(x)=\{id\}$. Then the isometry (\ref{E:isometry-ball-1}) means that
$$\tilde B_i(\tilde x,\eps)\cap \tilde X_i\to B_i( x,\eps)\cap X_i$$
is an isometry. Then the theorem follows in this case from Theorem \ref{T:spheres-to-segment}.

\underline{Case 2.} Assume that $\tilde x$ is the middle point of the segment $X$. Then $Stab(\tilde x)=\ZZ_2$.

The subset $\tilde B_i(\tilde x,\eps)\cap \tilde X_i$ is an open subset of $\tilde X_i\simeq\SS^2$ invariant under the antipodal involution and the obvious map
$$\ZZ_2\backslash (\tilde B_i(\tilde x,\eps)\cap \tilde X_i)\to B_i( x,\eps)\cap X_i$$
is an isometry. Now the result follows from Proposition \ref{P:involution-spheres}. \qed

\subsection{Collapse of Klein bottles.}\label{Ss:Klein-bottles}
\begin{theorem}\label{P:klein-bootles-prop}
Let $\{X_i\}$ be a sequence of Klein bottles with CBB($-1$) metrics. Let us assume that it GH-converges to a 1-dimensional limit $C$.
Furthermore let $d_i$ be metrics on $C\bigsqcup X_i$ extending the original metrics on $C$ and $X_i$ and such that
$$d_{i,H}(C,X_i)\to 0 \mbox{ as } i\to 0.$$
Let $z\in C$.
Let us choose $0<\eps_1<\eps_2<\frac{1}{100} length(C)$, and if $C$ is a segment and $z\in int(C)$ we require in addition that $\eps_2<\frac{1}{100} dist(z,\pt C)$.

Then for large $i$ one has
\begin{eqnarray*}
\dim Im\left[H^a(\cb_{i,z}(\eps_2))\to H^a(\cb_{i,z}(\eps_1))\right]=\left\{\begin{array}{ccc}
                                                  1&\mbox{if}& a=0,1\\
                                                  0&\mbox{otherwise}&\\
                                                      \end{array},\right.
\end{eqnarray*}
where $\cb_{i,z}(\eps)$ is defined by (\ref{star-star}).
\end{theorem}
{\bf Proof.} By Slutskiy's lemma \ref{slutskiy-lemma} and Lemma \ref{L:reduction-smooth} we may assume that the metrics on $X_i$ are $C^\infty$-smooth.

\hfill

 \underline{Case 1.} Let us assume that $C$ is a circle. Then by Theorem \ref{T:yamaguchi}
there is Yamaguchi smooth $\eps_i$-almost Riemannian submersion $f_i\colon X_i\to C$ with connected fibers, where $\eps_i\to 0$.
Its fibers are connected and 1-dimensional, hence circles.
Proposition \ref{P:submersion-cohomology} implies the result in the case of circle.

\hfill

\underline{Case 2.} Let  us assume that $C$ is a segment. Every $X_i$ has a 2-sheeted oriented cover $\tilde X_i$; it will be equipped with the pull-back of the metric on $X_i$.
By Lemma \ref{L:ivanov} $diam(\tilde X_i)\leq 2diam(X_i)$ although  a weaker estimate would be sufficient.
The group $\ZZ_2$ of deck transformations acts by isometries on each $\tilde X_i$.  Since the Gaussian curvature of $\tilde X_i$ is uniformly bounded below,
the Gromov compactness theorem and Proposition \ref{R:properness} imply that one may choose a subsequence such that $\tilde X_i\to \tilde X$ in $\cx^{\ZZ_2}$.
By Proposition \ref{P:quotient-convrgent}
$$X_i\to \ZZ_2\backslash \tilde X=C \mbox{ in } \cx.$$
It follows that $\tilde X$ is 1-dimensional and hence is either a segment or a circle. By \cite{katz-torii}, \cite{zamora-torii} tori cannot collapse to a segment. Hence $\tilde X$ is a circle. The group $\ZZ_2$ can act on a circle by isometries
only in three possible ways: identical action; symmetry with respect to the center of the circle (i.e. rotation by $\pi$); reflection with respect to a line. If the quotient by this action is a segment then the first two options are
impossible, and $\ZZ_2$ acts on the circle $\tilde X$ by reflection with respect to a line.

Let $\pi_i\colon \tilde X\bigsqcup \tilde X_i\to C\bigsqcup X_i$ be the quotient map by $\ZZ_2$. Let us fix $ x\in C$.

\underline{Case 2a.} Let us assume that $ x$ belongs to the interior of the segment $C$. Then $ x$ has exactly two different preimages: $\pi^{-1}_i( x)=\{\tilde x_1,\tilde x_2\}$.
Then $dist_{\tilde X}(\tilde x_1,\tilde x_2)=2dist_X( x,\pt C)$. Let $0<\eps_1<\eps_2<\frac{1}{100}dist_X( x,\pt C)$. Consider the open balls $B_i( x,\eps_{1}) \subset B_i( x,\eps_{2}) \subset C\bigsqcup X_i$.
By Lemma \ref{L:disjoint-isometr}
$$\pi_i^{-1}(B_i( x,\eps_{1,2})\cap  X_i)=(\tilde B_i(\tilde x_1,\eps_{1,2})\cap\tilde X_i)\bigsqcup (\tilde B_i(\tilde x_2,\eps_{1,2})\cap\tilde X_i),$$
where $\tilde B_i(\tilde x_{1,2},\eps_{1,2})$ are open balls in $\tilde X\bigsqcup \tilde X_i$.
The natural map $$Stab(\tilde x_1)\backslash (\tilde B_i(\tilde x_1,\eps_{1,2})\cap \tilde X_i)\to B_i( x,\eps_{1,2})\cap X_i$$
is an isometry. Since $Stab(\tilde x_1)=\{id\}$ we get that the natural map
\begin{eqnarray}\label{E:isometry-0009}
\tilde B_i(\tilde x_1,\eps_{1,2})\cap\tilde X_i\to B_i( x,\eps_{1,2})\cap X_i
\end{eqnarray}
is an isometry.

For large $i$ there exist Yamaguchi maps $\tilde X_i\to \tilde X$. Its fibers are circles. Hence by Proposition \ref{P:submersion-cohomology} one has
\begin{eqnarray*}
\dim Im\left[H^a(\tilde B_i(\tilde x_1,\eps_2)\cap \tilde X_i)\to H^a(\tilde B_i(\tilde x_1,\eps_1)\cap \tilde X_i)\right]=\left\{\begin{array}{ccc}
                                                  1&\mbox{if}& a=0,1\\
                                                  0&\mbox{otherwise}&\\
                                                      \end{array}\right.
\end{eqnarray*}

Since the isometries (\ref{E:isometry-0009}) commute with the imbeddings of $\eps_{1,2}$-balls into each other the result follows in the case $ x\in int(C)$.

\hfill

\underline{Case 2b.} Let us assume that $ x\in \pt C$. Then it has exactly one preimage: $\pi^{-1}_i( x)=\{\tilde x\}$. Let $0<\eps_1<\eps_2<\frac{1}{100}length(C)$.
By Lemma \ref{L:disjoint-isometr}
\begin{eqnarray*}
\pi_i^{-1}(B_i( x,\eps_{1,2}))=\tilde B_i(\tilde x,\eps_{1,2}),
\end{eqnarray*}
and the natural map
\begin{eqnarray}\label{E:isometry-0008}
\ZZ_2\backslash (\tilde B_i(\tilde x,\eps_{1,2})\cap \tilde X_i)\to B_i( x,\eps_{1,2})\cap X_i
\end{eqnarray}
is an isometry. Let
$$f_i\colon \tilde X_i\to \tilde X$$
be the $\ZZ_2$-equivariant Yamaguchi maps which exist by Proposition \ref{P:circle-reflection-equivariant}.

Let us denote by $\tilde B(\tilde x,\eps)\subset \tilde X$ the $\eps$-ball in $\tilde X$ centered at $\tilde x$. Then by (\ref{E:long-inclusion}) (with the current notation) we have inclusions for large $i$
\begin{eqnarray*}
f_i^{-1}(\tilde B(\tilde x,\frac{\eps_1}{2}))\subset\tilde B_i(\tilde x,\eps_1)\cap \tilde X_i\subset f_i^{-1}(\tilde B(\tilde x,\frac{\eps_1+\eps_2}{2}))\subset\tilde B_i(\tilde x,\eps_2)\cap \tilde X_i\subset f_i^{-1}(\tilde B(\tilde x,2\eps_2)).
\end{eqnarray*}
Since all 5 sets in this sequence are $\ZZ_2$-invariant, similar inclusions hold for their quotients by $\ZZ_2$ and hence induce the maps in cohomology
\begin{eqnarray*}
H^*(\ZZ_2\backslash f_i^{-1}(\tilde B(\tilde x,\frac{\eps_1}{2})))\overset{a_1}{\leftarrow} H^*(\ZZ_2\backslash(\tilde B_i(\tilde x,\eps_1)\cap \tilde X_i))\overset{a_2}{\leftarrow} H^*(\ZZ_2\backslash f_i^{-1}(\tilde B(\tilde x,\frac{\eps_1+\eps_2}{2})))
\overset{a_3}{\leftarrow}\\\overset{a_3}{\leftarrow}
H^*( \ZZ_2\backslash(\tilde B_i(\tilde x,\eps_2)\cap \tilde X_i))\overset{a_4}{\leftarrow} H^*( \ZZ_2\backslash f_i^{-1}(\tilde B(\tilde x,2\eps_2))).
\end{eqnarray*}
We claim that the maps $a_1\circ a_2$ and $a_3\circ a_4$ are isomorphisms. This follows from the following more precise statement.
\begin{claim}\label{Cl:moebius-strip}
Let $0<\delta_1<\delta_2<\frac{1}{100}length(C)$. Then the natural map
$$H^*(\ZZ_2\backslash f_i^{-1}(\tilde B(\tilde x,\delta_1)))\leftarrow H^*(\ZZ_2\backslash f_i^{-1}(\tilde B(\tilde x,\delta_2)))$$
is an isomorphism, and each set $\ZZ_2\backslash f_i^{-1}(\tilde B(\tilde x,\delta_{1,2}))$ is homeomorphic to the M\"obius band.
\end{claim}
Let us postpone the proof of this claim. The diagram chase implies that
$$Im[H^*( \ZZ_2\backslash(\tilde B_i(\tilde x,\eps_1)\cap \tilde X_i))\overset{a_2\circ a_3}{\leftarrow} H^*( \ZZ_2\backslash(\tilde B_i(\tilde x,\eps_2)\cap \tilde X_i))]$$
is isomorphic to the cohomology of the M\"obius band, and hence that would imply Theorem \ref{P:klein-bootles-prop}.

It remains to prove Claim \ref{Cl:moebius-strip}. Let us equip $\tilde X_i$ with the new metric $\tilde g_i$ coinciding with the original metric on $\tilde X_i$ on the vertical subspaces, having the same horizontal subspaces as the latter metric, and
coinciding on the horizontal subspaces with the pull back via $f_i$ of the metric on $\tilde X$.
Thus $f_i$ is the Riemannian submersion when $\tilde X_i$ is equipped with $\tilde g_i$. Let $\cn_\delta$ denote the $\delta$ neighborhood of the zero section of the normal bundle $\cn$ of $f_i^{-1}(\tilde x)$ with respect
to the metric $\tilde g_i$. By Lemma \ref{L:tube-diffeo}
$$\exp\colon \cn_{\delta_1} \to f^{-1}_i(\tilde B(\tilde x,\delta_1)),\, \exp\colon \cn_{\delta_2} \to f^{-1}_i(\tilde B(\tilde x,\delta_2))$$
are homeomorphisms.

The fiber $f_i^{-1}(\tilde x)$ is $\ZZ_2$-invariant. Hence $\ZZ_2$ acts naturally on the total space of the normal bundle $\cn$.
It is easy to see that $\exp\colon \cn\to \tilde X_i$ commutes with the action of $\ZZ_2$.
After taking $\ZZ_2$-quotient we have the diagram:
\begin{eqnarray*}
\square<-1`1`1`-1;1300`700>[H^*(\ZZ_2\backslash f_i^{-1}(\tilde B(\tilde x,\delta_1)))`H^*(\ZZ_2\backslash f_i^{-1}(\tilde B(\tilde x,\delta_2)))`H^*(\ZZ_2\backslash \cn_{\delta_1})` H^*(\ZZ_2\backslash \cn_{\delta_2}); `\exp^\ast `\exp^\ast `]
\end{eqnarray*}
where the vertical lines are isomorphisms.  $\cn_{\delta_1}\subset \cn_{\delta_2}$ is a $\ZZ_2$-equivariant retract; the retraction is given by multiplication by $t\in [\frac{\delta_1}{\delta_2},1]$. Hence the bottom line is an isomorphism.

Furthermore $\cn_{\delta_{a}}$, $a=1,2$, is $\ZZ_2$-equivariantly homeomorphic to $S^1\times(-\delta_{a},\delta_{a})$ when the $\ZZ_2$-action is given by
$$(\theta,\tau)\mapsto (-\theta,-\tau)\mbox{ for any } \theta\in S^1,\tau\in (-\delta_{a},\delta_{a}).$$
(Indeed the $\ZZ_2$-action on $f_i^{-1}(\tilde x)\simeq S^1$ has no fixed point and hence is equivalent to antipodal involution on $S^1$.)
The quotient of $S^1\times(-\delta_{a},\delta_{a})$ by this action is homeomorphic to the M\"obius band. \qed

\subsection{The no collapse case.}\label{Ss:no-collapse-case}
The main result of this subsection is
\begin{theorem}\label{T:no-collapse}
Let $\{X_i^n\}$ be a sequence of compact $n$-dimensional Alexandrov spaces with curvature uniformly bounded from below which GH-converges to a compact Alexandrov space $X^n$ of dimension $n$. Let $d_i$
be metrics on $X\coprod X_i$ extending the original metrics on $X$ and $X_i$ and such that the Hausdorff distance $d_{i,H}(X_i,X)\to 0$. For $\eps>0$ and $x\in X$ denote as in (\ref{Def:cb})
$$\cb_{i,x}(\eps):=\{y\in X_i|\,\, d_i(y,x)<\eps\}.$$
Then for any $x\in X$ there exist $\eps_0>0$, $i_0\in \NN$ (depending on $x,d_i$) such that for any $0<\delta_1<\delta_2<\eps_0$, any $i>i_0$, any $a\in \ZZ$,
and any field $\FF$ the image of the natural map in the $a$th cohomology satisfies
$$h^a(x):=\dim Im[H^a(\cb_{i,x}(\delta_2);\FF)\to H^a(\cb_{i,x}(\delta_1);\FF)]=\left\{\begin{array}{ccc}
                                                                          1&\mbox{if}&a=0\\
                                                                          0&\mbox{if}&a\ne 0.
                                                                          \end{array}\right.$$
In particular $F:=\sum_a (-1)^a h^a\equiv 1$.
\end{theorem}

The proof will be an easy consequence of the following more general result due to V. Kapovitch \cite{kapovitch-private-commun} who also supplied us with a proof below.

\begin{theorem}\label{alex-conv-balls}
Let $(X_i^n,p_i)\to (X,p)$ be a sequence of $n$-dimensional Alexandrov spaces with curvature uniformly bounded from below which GH-converges (as pointed spaces) to an $n$-dimensional Alexandrov space $X$.

Then there exists $r_0=r_0(p)>0$ such that the following holds

\begin{enumerate}
\item \label{conv-tangent} For any  $0<r\le r_0$ we have that sphere $S_r(p)\subset X$ is homeomorphic to the space of directions $\Sigma_pX$ and the closed ball $\bar B_r(p)\subset X$
is homeomorphic to the unit closed ball $\bar B_1(0)$ in $T_pX$ with the spheres $S_r(p)$ mapping to the unit sphere $S_1(0)\subset T_pX$ which is isometric to $\Sigma_pX$.
\item\label{conv-sequence}  For any  $0<r\le r_0$  for all large $i$ the closed ball $\bar B_r(p_i)\subset X_i$ is homeomorphic to the closed ball $\bar B_r(p)\subset X$ by a homeomorphism sending $S_r(p_i)$ to $S_r(p)$
\end{enumerate}

\end{theorem}

Let us postpone a proof of this theorem and prove Theorem \ref{T:no-collapse}.

{\bf Proof of Theorem \ref{T:no-collapse}.} 
Fix $x\in X$. Let us choose sequence of points $x_i\in X_i$ such that $d_i(x_i,x)\to 0$. Denote
$$B_i(\eps):=\{y\in X_i|\, d_i(y,x_i)<\eps\}$$
to be the open $\eps$-ball in $X_i$ centered at $x_i$. Let $r_0>0$ be as in Theorem \ref{alex-conv-balls}. For $0<\delta_1<\delta_2<r_0/2$ and large $i$ one has inclusions
$$B_i(\frac{\delta_1}{2})\subset \cb_{i,x}(\delta_1)\subset B_i(\frac{\delta_1+\delta_2}{2})\subset \cb_{i,x}(\delta_2)\subset B_i(2\delta_2),$$
and the first, the third, and the fifth spaces are homeomorphic to an open ball $B_1(0)\subset T_xX$ by Theorem \ref{alex-conv-balls}. This ball in contractible and hence has cohomology of a point.
In cohomology with coefficients in an arbitrary field we have maps
$$H^*(B_i(\frac{\delta_1}{2}))\overset{a_1}{\leftarrow} H^*(\cb_{i,x}(\delta_1))\overset{a_2}{\leftarrow} H^*(B_i(\frac{\delta_1+\delta_2}{2}))\overset{a_3}{\leftarrow} H^*(\cb_{i,x}(\delta_2))\overset{a_4}{\leftarrow} H^*(B_i(2\delta_2)).$$
The maps $a_1\circ a_2$ and $a_3\circ a_4$ are isomorphisms since they are maps of cohomology of contractible spaces. Hence by linear algebra it follows that $Im[H^a(\cb_{i,x}(\delta_2))\to H^a(\cb_{i,x}(\delta_1))]$ is isomorphic to the $a$th cohomology of a point. \qed

\hfill

{\bf Proof of Theorem ~\ref{alex-conv-balls}.}

Since $(\lambda X, p)\to (T_px,0)=(C\Sigma_p,0)$ as $\lambda\to \infty$ we have that  for small $r$ any point $x$ with $d(p,x)=r$ is $o(r)$ close to the midpoint of a geodesic $[p,q]$ of length $2r$. Hence the comparison angle $\tilde \angle qxp\ge \pi-o(r)$ and by Toponogov comparison the same holds for the actual angle $ \angle qxp$ for any shortest geodesics $[xp], [xq]$ . Therefore if $r_0>0$ is small enough it holds that $f=d(\cdot,  p)$ has no critical points in
$B_{2r_0}(p)\setminus\{p\}$.

Moreover, this also shows that for any fixed $0<r<r_0$ for all large $i$ the function $f_i=d_{p_i}(\cdot)$ has no critical points in the annulus $A(r/2,2r,p_i)=\{y\in X_i\mid r/2\le d(x,p_i)\le 2r\}$.

By the Perelman's Parameterized Stability Theorem \cite{kapovitch-stablility}, Theorem 7.8, or \cite{perelman-stability}, Theorem 4.3,  for all large $i$ there exists a homeomorphism $\phi_i: \bar B_{r}(p_i)\to \bar B_r(p)$ which commutes with the distance to the center  outside of the $r/2$ balls.

That is $f\circ \phi_i=f_i$ on  $A(r/2,r,p_i)$.  In particular it sends metric spheres around $p_i$ of radii between $r/2$ and $r$ to the corresponding metric spheres around $p$. This proves item \eqref{conv-sequence}.

Applying the same argument to the convergence $(\lambda X, p)\underset{\lambda\to\infty}{\longrightarrow} (T_px,0)=(C\Sigma_p,0)$  gives part \eqref{conv-tangent}.
This finished the proof of Theorem ~\ref{alex-conv-balls}. \qed

\subsection{An application to integration with respect to the Euler characteristic.}\label{Ss:euler-charact}
As an application of our main results of closed surfaces let us prove the property (\ref{limit-intrinsic-volumes}) for intrinsic volumes. As we have mentioned in the introduction,
the only relevant cases of intrinsic volumes are $V_0=\chi$ and $V_2=area$. For the latter the result follows from \cite{burago-gromov-perelman}, Theorem 10.8.
Let us consider the former case of the Euler characteristic $\chi$.

First we have to provide more explanations on the right hand side of (\ref{limit-intrinsic-volumes}). We need to briefly remind the notion of integration of a 'constructible'
function with respect to the Euler characteristic. The technical definition of constructibility depends on the context and will not be specified here. However in the application
in this subsection this will not lead to misunderstanding due to simplicity of the situation. Let $X$ be a 'nice' topological space, e.g. it might be a real analytic manifold. For the general conjectures in \cite{alesker-conjectures}
$X$ should be a compact finite dimensional Alexandrov space. In this subsection it suffices to consider the case of closed topological surfaces, segments, and points.

Let $F\colon X\to \CC$ be a compactly supported function with finitely many values and 'nice' level sets. Under broad assumptions satisfied in relevant for us case $F$ can be written as
a finite linear combination
\begin{eqnarray}\label{E:represent-function}
F=\sum_i c_i\One_{Z_i}
\end{eqnarray} where $c_i\in \CC$ and $Z_i$ are compact sets which are 'nice' enough to have a well defined Euler characteristic (thus e.g. the Cantor set is excluded).
One defines the integral of $F$ with respect to the Euler characteristic by
$$\int F d\chi:=\sum_i c_i\chi (Z_i). $$
In a number of technical setups one can show that this integral is well defined, i.e. is independent of the presentation (\ref{E:represent-function}).
\begin{remark}\label{Rem-integral-euler}
(1) The fact that the integral is well defined was proved by Groemer \cite{groemer} for the class of functions on $X=\RR^n$ such that the level sets are finite unions of convex compact sets.
When $X$ is a complex analytic manifold and the level sets of functions are complex analytic subvarieties this fact is due to Viro \cite{viro}.
For a real analytic manifold $X$ and functions with level sets being subanalytic this fact was proven by Schapira \cite{schapira}, while a more special situation was previously considered
by Khovanskii and Pukhlikov \cite{khovanskii-pukhlikov}.

(2) For the purposes of general conjectures \cite{alesker-conjectures} the relevant class of constructible functions consists of functions on a compact Alexandrov space $X$ which are constant on
the strata of the Perelman-Petrunin stratification. In the present paper a more elementary situation will be sufficient: constant functions on topological surfaces and functions on a closed segment which
are constant in its interior. In the former case the integral with respect to the Euler characteristic is trivially well defined once we consider only constant functions.
The latter case is a special case either of the Khovanskii-Pukhlikov \cite{khovanskii-pukhlikov} or the Schapira \cite{schapira} approaches, although it can be treated directly by elementary methods.
\end{remark}
\begin{proposition}\label{R:euler-characteristic-integral}
Let $\{X_i\}$ be a sequence  of closed topological 2-surfaces of a fixed homeomorphism type with CBB($-1$) metrics. Let us denote their Euler characteristic by $\alp$. Assume that the sequence GH-converges to a compact Alexandrov space $X$.
Let $F(x):=\sum_a(-1)^a h^a(x)$ be the corresponding function on $X$ which is necessarily well defined. Then
$$\int_X Fd\chi=\alp.$$
\end{proposition}
{\bf Proof.} \underline{Case 1.} Let $X$ be a point. Then by construction $h^a(X)=\dim H^a(X_i;\FF)$. Hence $F=\alp$, and $\int_X F d\chi=\alp$.

\hfill

\underline{Case 2.} Let $\dim X=2$. By the Perelman stability theorem \cite{perelman-stability} (see also \cite{kapovitch-stablility}) all $X_i$ are homeomorphic to $X$. By Theorem \ref{T:no-collapse} one has
$F\equiv 1$. Hence $\int_X Fd\chi=\chi(X)=\alp$.

\hfill

\underline{Case 3.} Let $\dim X=1$.

\underline{Subcase 3a.} Assume that $\{X_i\}$ are homeomorphic to the 2-sphere $\SS^2$. By Theorem \ref{T:surface-t-circle}(2) $X$ is a segment.
By Theorem \ref{T:spheres-to-segment} one has $F=\One_{\pt X}$. Clearly
$$\int_X Fd\chi=2=\chi(\SS^2)$$
as necessary.

\underline{Subcase 3b.} Assume that $\{X_i\}$ are homeomorphic to the real projective plane $\RR\PP^2$. By Theorem \ref{T:surface-t-circle}(2) $X$ is a segment.
By Theorem \ref{T:rp2-segment} the function $F$ vanishes everywhere on $X$ but one point. Hence
$$\int_X Fd\chi =1=\chi(\RR\PP^2)$$
as required.

\underline{Subcase 3c.} Assume that $\{X_i\}$ are homeomorphic to the 2-torus $\TT^2$. Since by \cite{katz-torii} and \cite{zamora-torii} torus cannot collapse to a segment, $X$ is a circle.
By Theorem \ref{T:surface-t-circle}(1) one has $F\equiv 0$. Then
$$\int_X Fd\chi =0=\chi(\TT^2).$$

\underline{Subcase 3d.} Assume that $\{X_i\}$ are homeomorphic to the Klein bottle $\KK$. Then $X$ could be both circle and segment.
If $X$ is a circle then by Theorem \ref{T:surface-t-circle}(1) $F\equiv 0$, and if $X$ is a segment  then by Theorem \ref{P:klein-bootles-prop} one has $F=0$ again. Hence in either case
$$\int_X Fd\chi =0=\chi(\KK).$$
\qed

\end{document}